\documentclass{article}
\usepackage{braket,tikz,float,mathrsfs,verbatim,bbm,dsfont, amssymb,mathtools,amsthm}
\usepackage{float}
\usepackage{algorithm}
\usepackage{algorithmic}
\usepackage[margin=1.0in]{geometry}
\usepackage{xcolor}
\usepackage{hyperref}
\hypersetup{hidelinks}

\newtheorem{Thm}{Theorem}[section]
\newtheorem{Def}[Thm]{Definition} 
\newtheorem{Lemma}[Thm]{Lemma}

\newtheorem{Prop}[Thm]{Proposition}
\newtheorem{Remark}{Remark}[section]
\newtheorem{Asm}{Assumption}[section]

\newcommand{\E}{\mathbb{E}}
\newcommand{\R}{\mathbb{R}}

\newcommand{\Hb}{\mathbf{H}}
\newcommand{\RPC}{\mathrm{RPC}}
\newcommand{\pRPC}{\mathrm{pRPC}}
\newcommand{\mRPC}{\mathrm{mRPC}}
\newcommand{\Lb}{\mathbf{L}}

\newcommand{\Ib}{\mathbf{I}}

\newcommand{\nnz}{\mathrm{nnz}}

\newcommand{\Ecal}{\mathcal{E}}

\newcommand{\tm}{\mathrm{m}}
\newcommand{\ILS}{\mathcal{J}_{L,S}}
\newcommand{\m}{\mathrm{Int}}
\newcommand{\multiset}{(\alpha,\beta,\gamma)\in \ILS}

\numberwithin{equation}{section}

\title{A Recursive Polynomial Chaos Evolution Method for Stochastic Differential Equations\thanks{Submitted to the editors on DATE.  }}

\author{
Guillaume Bal\thanks{Departments of Statistics and Mathematics, University of Chicago, Chicago, IL 60637, USA. Email: \texttt{guillaumebal@uchicago.edu}.}
\and
Shengbo Ma\thanks{Department of Mathematics, The University of Hong Kong, Pokfulam Road, Hong Kong SAR, P.R. China. Email: \texttt{u3011677@connect.hku.hk}.}
\and
Su Zhang\thanks{Department of Mathematics, The University of Hong Kong, Pokfulam Road, Hong Kong SAR, P.R. China. Email: \texttt{zhangsu24@connect.hku.hk}.}
\and
Zhiwen Zhang\thanks{Department of Mathematics, The University of Hong Kong, Pokfulam Road, Hong Kong SAR, P.R. China. Email: \texttt{zhangzw@hku.hk}.}
}

\date{}

\begin{document}

\maketitle

\begin{abstract}
Numerical simulation of stochastic differential equations over long time intervals poses significant computational challenges. In this paper, we propose a novel recursive polynomial chaos evolution method that achieves model reduction without sampling by exploiting the Markov property to maintain a fixed low-dimensional representation throughout the time evolution. At each time step, we construct orthogonal polynomial bases adapted to the current probability measure, and project the one-step-ahead solution onto this new basis together with the new Brownian increments. This dynamic updating strategy effectively reduces the dimension of the random variables during long-time evolution. Under appropriate assumptions, we prove the convergence of the method, specifically that the distributions generated by the method preserve convergence in the Wasserstein-1 distance. We present numerical results demonstrating that the method can accurately capture complex dynamical behaviors with high accuracy and low computational cost.
\end{abstract}

\medskip
\noindent\textbf{Keywords:} Uncertainty quantification, model reduction, stochastic differential equations (SDEs), polynomial chaos expansions, long-time integration, convergence analysis.

\medskip
\noindent\textbf{AMS subject classifications:} 65C30, 60H35, 41A10, 65M70.

\section{Introduction}

Stochastic differential equations (SDEs) play an important role in many areas of engineering and applied sciences, such as solid mechanics, filtering, stochastic control, and finance \cite{HOU2006687,intermtestmodel1,papanicolaou2019introductionstochasticdifferentialequations}. Of particular interest is the long-time dynamical behavior of SDEs, including convergence to invariant measures, periodic phenomena, and intermittency \cite{intermitemodel2,FANG2023133919}. Since obtaining analytical solutions to SDEs is infeasible for most practical problems, reliable and efficient numerical methods need to be developed.

While standard Monte Carlo (MC) methods and their variants are widely used for SDE simulation \cite{Kloeden1992,james1980monte,Caflisch_1998,giles2015multilevel}, their need for a large number of samples makes them computationally expensive. To reduce computational costs, model reduction techniques have been proposed that approximate the random space using low-dimensional representations. A representative class of such methods is the polynomial chaos expansion (PCE) framework \cite{ghanem1991,xiu2003,Arbitrartygpc,HOU2006687,luo2006wiener}. The PCE solution is represented as a truncated spectral expansion in an orthogonal polynomial basis that corresponds to the probability measure of the input randomness. This representation transforms the stochastic problem into a deterministic system of equations for the expansion coefficients. However, standard PCE encounters the curse of dimensionality when applied to long-time integration of systems driven by white noise.  The number of random variables required to represent Brownian motion grows linearly in time, which leads to exponential growth in the number of basis functions, as a fixed polynomial basis must incorporate all variables from the outset $\cite{HOU2006687,luo2006wiener,intermitemodel2}$. Several dynamically evolving basis methods were developed to address these limitations, such as the time-dependent generalized polynomial chaos (TDgPC) method \cite{GERRITSMA20108333,heuveline2014hybrid} for random differential equations and the dynamical generalized polynomial chaos (DgPC) \cite{Bal2,bal2017300,lacour2021dynamic} for system driven by Brownian motions.

Another class of model reduction methods seeks to evolve the solution dynamically on the manifold of low-rank tensors. Here, by a low-rank tensor, we mean a reasonably small sum of products of functions involving a subset of the independent (such as spatial, angular, stochastic) variables. Dynamically orthogonal (DO) \cite{Sapsis2009DynamicallyOF}, dynamically bi-orthogonal (DyBO) \cite{CHENG2013753,CHENG2013843}, and dynamical low-rank approximation (DLRA) \cite{Lubich050639703} are representative methods, which project the temporal derivative to the tangent space of the low-rank manifold and solve the time-dependent differential equations of the low-rank solution. The low-rank structure can correspondingly reduce computational and storage requirements. These methods may suffer from certain numerical instabilities: for DO, when singular values become too small, and for DyBO, at singular value crossings. To overcome these numerical issues, projector-splitting integrators in DLRA \cite{Lubich2014,Ceruti2022} provide split evolution strategies for low-rank solutions based on time discretizations, which enhance robustness compared with standard numerical integrators. Most theoretical and numerical developments of these methods focus on PDEs and RPDEs (\cite{Ceruti20234,Einkemmer2018,Kazashi2021}), while they rarely solve equations with stochastic forcing terms.

Recent extensions of low-rank approximation methods to SDEs were analyzed in \cite{kazashi2025dynamical,Lujianfengdynamical}. Deriving equations for the low-rank approximations becomes challenging because random forcing is no longer smooth in time. When dealing with stochastic problems, these low-rank methods still require MC or PCE techniques as computational components, which further limits their effectiveness for long-time integration. The use of model reduction techniques for long-time SDE evolutions thus remains challenging. 

Building on these ideas, we propose a recursive polynomial chaos (RPC) method for the long-time integration of SDEs. Our key innovation lies in recursively updating the orthonormal polynomial basis at each discrete time step. This basis allows us to represent the drift and diffusion terms in a manner adapted to the time-evolving probability measure. Specifically, we leverage the Markov property of SDEs through the following approach: at each time step, we characterize the distribution using the orthonormal polynomials of the current probability measure. We then represent the next-step solution by projecting it onto this basis combined with the Brownian increment, and proceed to construct the orthonormal polynomials for this new distribution. We present two complementary implementations of RPC: pRPC (see Algorithm \ref{algorithm3}) directly evolves the orthonormal basis and algebraic structures via Itô-Taylor expansion, while mRPC (see Algorithm \ref{algorithm}) evolves moments and reconstructs the basis from moments at each step, yielding different computational trade-offs to suit different problem requirements.  This framework can be implemented for a wide range of numerical schemes, e.g., Euler-type, Milstein, and higher-order methods \cite{Kloeden1992}.

We next provide a theoretical analysis of the convergence properties of mRPC. Under suitable conditions on the drift and diffusion terms, we construct a compactly supported approximation to our dynamical model. Assuming some natural stability properties on mRPC that we do not justify theoretically but may verify numerically, we construct the evolving orthogonal polynomials and present convergence estimates for the statistical moments generated by mRPC (See Theorem~\ref{thm: moment error}). This allows us to show, under further suitable assumptions on the dynamics, an appropriate error estimate in the Wasserstein distance between the exact and mRPC distributions in Theorem~\ref{W1distance}.

Numerical simulations are finally performed for several nonlinear stochastic systems, including moderately high-dimensional systems, systems with random coefficients or random initial conditions, and systems with complex long-time dynamics such as intermittency, non-Gaussian behavior, and convergence to an invariant measure. Our numerical results demonstrate that RPC accurately captures long-time behavior while maintaining high computational efficiency by reducing the dimension of the stochastic space inherent in PCE and avoiding the sampling burden of MC methods. In contrast to PCE, which may produce oscillations during the time evolution, and MC, which 
suffers from convergence-related statistical fluctuations, our approach exhibits smooth and stable evolution throughout the simulations. These results confirm that our method can efficiently 
capture long-time dynamical behavior. This opens promising avenues for future research: combining the RPC approach with dynamical low-rank techniques (DO, DyBO, and DLRA) can potentially tackle the challenging problem of long-time evolution for high-dimensional SDEs and SPDEs, where RPC will replace the expensive MC or PCE components that are typically employed in such frameworks $\cite{CHENG2013843,Lujianfengdynamical}$.
 
The paper is organized as follows. In Section~\ref{sec:ourmethod}, we present the RPC method and state the main convergence results. In Section~\ref{sec:proof of main result}, we provide the necessary theoretical background and detailed proofs of the main convergence results. In Section~\ref{sec:numericalres}, we present numerical results of RPC and comparisons with existing numerical methods. Concluding remarks are given in Section~\ref{sec:conclusions}. Additional auxiliary proofs are provided in Appendix Section~\ref{more_proofs}.

\section{Description of proposed method}
\label{sec:ourmethod}
    In this section, we present the RPC method for solving SDEs and state our main convergence results. We consider the following $d$-dimensional SDE driven by an $m$-dimen\-sion\-al Brownian motion on a probability space $(\Omega, \mathscr{F}, \mathbb{P})$:
    \begin{equation}
        \label{original SDE}
        du(t) = b(u(t))dt + \sigma(u(t))dW(t), \quad u(0) = u_0, \quad t \in [0,T],
    \end{equation}
     where $u(t) \in \mathbb{R}^d$, $b=(b^{(i)})_{1\leq i\leq d}$ is the drift function, $\sigma = (\sigma^{(i,j)})_{ 1 \leq i \leq d,1 \leq j \leq m }$ is the diffusion function, with $b^{(i)},\sigma^{(i,j)}: \mathbb{R}^d \rightarrow \mathbb{R}$, and $W(t)$ is an $m$-dim\-ension\-al standard Brownian motion with respect to a filtration $\mathscr{F}_t$.

    For the Numerical discretization, the time domain  $[0,T]$ is divided into uniform intervals with time nodes $0 = t_0 < \cdots <t_n=T$ and time step $ h $ with $t_k = kh$. Throughout the paper, we denote by $u(t)$ the exact solution of the SDE~\eqref{original SDE} and by $v_k \in \mathbb{R}^d$ its numerical approximation at time $t_k$. We also introduce the following standard multi-index notation. Let $\alpha = (\alpha_1, \cdots, \alpha_d) \in \mathbb{N}^d$ be a $d$-dimensional multi-index with $|\alpha| = \sum_{i=1}^d \alpha_i$ its total order. For any vector $x = (x_1, \cdots, x_d) \in \mathbb{R}^d$, we define the monomial $x^\alpha = x_1^{\alpha_1} \cdots x_d^{\alpha_d}$. The factorial of a multi-index is $\alpha! = \prod_{i=1}^d \alpha_i!$, and for multi-indices $\alpha, \beta$ with $\beta \leq \alpha$ (meaning $\beta_i \leq \alpha_i$ for all $i$), the multinomial coefficient is
    \begin{equation*}
        \binom{\alpha}{\beta} = \frac{\alpha!}{\beta!(\alpha-\beta)!} = \prod_{i=1}^d \binom{\alpha_i}{\beta_i}.
    \end{equation*}
    
    For two $d$-dimensional vectors $x = (x_1, \cdots, x_d)$ and $y = (y_1, \cdots, y_d)$, we define the Hadamard product of $d$-dimensional vectors by 
    \begin{equation}\label{Hadamard product}
        x \circ y = (x_1 y_1, x_2 y_2, \cdots, x_d y_d).
    \end{equation}
    To simplify notation, we omit the symbol $ \circ $ in the rest of the paper, but still clarify this type of vector product when necessary.

    \subsection{Formulation}
        \label{Formulation}
        The core idea of the proposed method lies in the treatment of the stochastic evolution. Instead of directly approximating the Brownian motion over the entire time span $ [0,T] $ as traditional PCE requires \cite{ghanem1991}, we treat it as a cumulative sum of independent Gaussian increments. When evolving from time  $ t_{k} $ to $ t_{k+1} $, we decompose the solution $ u(t_{k+1})$ in the polynomial chaos expansion basis constructed from the distribution of $u(t_{k})$ at time $t_{k}$ and the Gaussian increment over $[t_{k},t_{k+1}]$. This approach leverages the Markov property of the SDE solution so that the numerical evolution forms a discrete Markov process. See Remark \ref{rem:nonMarkovian} below for extensions beyond the (strictly) Markovian case.

        We illustrate and analyze RPC for the  Euler-type scheme \eqref{Euler type}: for $0\leq k\leq n-1$,
    \begin{equation}\label{Euler type}
    \begin{aligned}
        v_{k+1} = v_k + b_h(v_k) h + \sigma_{h}(v_k)  z_k\sqrt{h}, \quad v_0 = u_0, 
    \end{aligned}
\end{equation}    
     where the diffusion step is treated explicitly using independent standard Gaussian random variables $z_k = (z_k^{(j)})_{1 \leq j \leq m}$. The subscript $h$ in $b$ and $\sigma$ indicates that both the drift and diffusion operators can be numerically adjusted according to the time step size. Many numerical schemes can be represented within the framework of \eqref{Euler type}, including explicit Euler-type methods such as the Euler-Maruyama method \cite{Kloeden1992} with $b_h(x) = b(x)$ and $\sigma_h(x) = \sigma(x)$, truncated Euler-Maruyama methods \cite{mao2015truncated,li2021strong} featuring $h$-dependent truncations of $b(x)$ and $\sigma(x)$, and semi-implicit schemes such as the split-type backward Euler method \cite{higham2002strong} for which $b_h$ and $\sigma_h$ are obtained by solving an implicit one-step equation.        
    
    In the RPC method, we do not directly describe the distribution of $v_k$; instead, we represent $v_k$ via its orthogonal polynomials, and it is these polynomials that are evolved over time. For each multi-index $\alpha$, let $T_{k,\alpha}(x)$ be the orthonormal polynomial of degree $|\alpha|$ with respect to the distribution of $v_k$, expressed in the monomial basis $\{x^{\beta}\}_{|\beta|<\infty}$. Since orthogonal polynomials depend on the monomial ordering, this set is not unique; we fix the ordering to be consistent with the multi-index ordering. We write $T_{k,\alpha}(x)$ for the polynomial function and $T_{k,\alpha}(v_k)$ for the induced random variable. Denote by $\Gamma_{k,\alpha\beta\gamma}:= \mathbb{E}\left[T_{k,\alpha}(v_k)T_{k,\beta}(v_k)T_{k,\gamma}(v_k)\right]$ the expectation of the triple product of orthonormal polynomials with respect to the distribution of $v_k$. 
    
    Assume we have $\{T_{k,\alpha}(x)\}_{|\alpha|<\infty}$ and $\{\Gamma_{k,\alpha\beta\gamma}\}_{|\alpha|<\infty}$ at time $t_k$. The approximation $v_{k+1}$ is then projected onto the space spanned by $\{T_{k,\alpha}(v_k),z_k T_{k,\alpha}(v_k)\}_{|\alpha|<\infty}$. Corresponding to the Euler-type scheme~\eqref{Euler type}, we write the projection scheme as
    \begin{equation}\label{vector projection}
        \begin{aligned}
            v_{k+1} &= v_k + h\sum_{|\alpha|<\infty} b_{k,\alpha}  T_{k,\alpha}(v_k) + \sqrt{h} \sum_{|\alpha|<\infty } \sigma_{k,\alpha}  z_k  T_{k,\alpha}(v_k),
        \end{aligned}
    \end{equation}
    where the projection coefficient vectors $b_{k,\alpha} \in \mathbb{R}^d$, and the projection coefficient matrices $\sigma_{k,\alpha} \in \mathbb{R}^{d \times m}$ represent the drift and diffusion components, respectively. 
     
    Comparing~\eqref{vector projection} with~\eqref{Euler type}, these coefficients are determined by the Galerkin projection of $b_h(v_k)$ and $\sigma_h(v_k)$ onto the orthonormal basis, where the projection is defined by computing the expectation with respect to the distribution of $v_k$:
    \begin{align}
        b_{k,\alpha}^{(i)} &= \mathbb{E} \left[b_h^{(i)}(v_k)  T_{k,\alpha}(v_k)\right], \quad 1 \leq i \leq d, \label{b equation}\\
         \sigma_{k,\alpha}^{(i,j)} &= \mathbb{E} \left[\sigma_h^{(i,j)}(v_k)  T_{k,\alpha}(v_k)\right], \quad 1 \leq i \leq d, 1 \leq j \leq m.\label{sigma equation}
    \end{align}
    To ensure that $v_{k+1}$ admits a polynomial chaos representation, we also need to expand $v_k$ itself in the orthonormal basis. Since $v_k^{(i)} - \mathbb{E}[v_k^{(i)}]$ is linear in $v_k$, it lies exactly in the span of $\{T_{k,\alpha} : |\alpha|=1\}$, and the expansion coefficients are given explicitly by the covariance structure:
    \begin{equation}
    \label{v projection}
        v^{(i)}_k = \mathbb{E}[v^{(i)}_k]
        + \sum_{j=1}^{d} (\mathbf{L}_k)_{ij}  T_{k,e_j}(v_k), 
        \quad 1 \leq i \leq d,
    \end{equation}
    where $\mathbf{L}_k$ is the Cholesky factor of $\mathrm{Cov}(v_k) = \mathbf{L}_k \mathbf{L}_k^T$ and $e_j$ is the $j$-th standard unit multi-index. Substituting \eqref{b equation} \eqref {sigma equation} \eqref{v projection} into~\eqref{vector projection}, the numerical solution $v_{k+1}$ is thus entirely expressed as a polynomial chaos expansion in terms of $v_k$. 

    Note that for any two random variables $w_1$ and $w_2$ with polynomial chaos expansions
    \[
    w_1 = \sum_{|\alpha|<\infty} w_{1,\alpha} T_{k,\alpha}(v_k), \qquad 
    w_2 = \sum_{|\alpha|<\infty} w_{2,\alpha} T_{k,\alpha}(v_k),
    \]
    their expectation of the product can be expressed directly in terms of the expansion coefficients:
    \begin{equation}
        \mathbb{E}[w_1 w_2] = \sum_{|\alpha|<\infty} w_{1,\alpha}\circ w_{2,\alpha}.
    \end{equation} 
    Moreover, the polynomial chaos expansion of the product $w_1 w_2$ is obtained via the triple products $\Gamma_{k,\alpha\beta\gamma}$:
    \begin{equation}
        w_1 w_2 = \sum_{|\alpha|<\infty} \left( \sum_{|\beta|,|\gamma|<\infty} w_{1,\beta} \circ w_{2,\gamma} \Gamma_{k,\alpha\beta\gamma} \right) T_{k,\alpha}(v_k).
    \end{equation}
    where the products between $w_{1,\alpha}$ and $w_{1,\beta}$ denote the Hadamard product (\ref{Hadamard product}). Consequently, given $\{T_{k,\alpha}(x)\}_{|\alpha|<\infty}$, $\{\Gamma_{k,\alpha\beta\gamma}\}_{|\alpha|,|\beta|,|\gamma|<\infty}$, and the polynomial chaos expansion of $v_{k+1}$ in terms of $v_k$, we can compute any polynomial expectations of $v_{k+1}$. Since $z_k$ is independent of $v_k$, the terms involving $z_k$ can be handled separately, and the moments of $z_k$ integrated analytically. 
    This enables us to update the orthonormal polynomials $\{T_{k+1,\alpha}(x)\}_{|\alpha|<\infty}$ via the Gram-Schmidt orthogonalization procedure, and subsequently compute the corresponding triple products $\{\Gamma_{k+1,\alpha\beta\gamma}\}_{|\alpha|,|\beta|,|\gamma|<\infty}$. If $v_{k+1}$ is determinate, which we always assume, the orthonormal polynomials $\{T_{k+1,\alpha}(x)\}_{|\alpha|<\infty}$ uniquely determine this distribution. 

    Such a procedure can be applied recursively, starting from the known initial distribution of $v_0$. At each step, the pair $(\{T_{k,\alpha}\}, \{\Gamma_{k,\alpha\beta\gamma}\})$ enables us to compute the polynomial expectations of $v_{k+1}$, which in turn determine $(\{T_{k+1,\alpha}\}, \{\Gamma_{k+1,\alpha\beta\gamma}\})$. Under the assumption that all $v_k$ are determinate, this recursive construction propagates the exact distribution of the numerical solution at each time step. This recursive process is illustrated in Figure \ref{fig:projection}. 
    
    \begin{figure}[H]
        \centering
        \begin{tikzpicture}
            \draw[line width= 0.6 pt] (-0.5,0) -- (9.5,0);
            \foreach \x in {0,4.5,9} {
                \draw[line width=0.6pt] (\x,-0.15) -- (\x,0.15);
            }
            
            \node at (0, -1) {$v_{k-1}$};
            \node at (4.5, -1) {$v_{k}$};
            \node at (9, -1) {$v_{k+1}$};
            
           \node at (0, -0.5) {$(T_{k-1,\alpha},\; \Gamma_{k-1,\alpha\beta\gamma})$};
            \node at (4.5, -0.5) {$(T_{k,\alpha},\; \Gamma_{k,\alpha\beta\gamma})$};
            \node at (9, -0.5) {$(T_{k+1,\alpha},\; \Gamma_{k+1,\alpha\beta\gamma})$};
            
            \draw[->] (0,-1.3) to[out=-20,in=-160] (4.5,-1.3) ;
            \draw[->] (4.5,-1.3) to[out=-20,in=-160] (9,-1.3);
            
            \node at (2,-2) {\scriptsize$ \text{span}\{ T_{k-1,\alpha}(v_{k-1}),  z_{k-1}T_{k-1,\alpha}(v_{k-1})\} $};
            \node at (7,-2) {\scriptsize$ \text{span}\{ T_{k,\alpha}(v_k), z_kT_{k,\alpha}(v_k)\} $};
        \end{tikzpicture}
        \caption{Propagation of Euler-type projection procedure}
        \label{fig:projection}
    \end{figure}
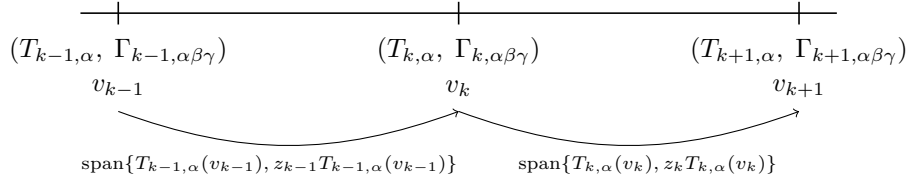

\subsection{Implementation}\label{sec:numericalmethods}
    In this section, we describe the numerical implementation of our RPC method. In practice, the evolved polynomials are truncated to a finite total degree $|\alpha| \leq L$. Assume the orthonormal polynomials at $t_k$ are $\{\hat{T}_{k,\alpha}(x)\}_{|\alpha|\leq L}$ and the corresponding triple products are $\{\hat{\Gamma}_{k,\alpha\beta\gamma}\}_{|\alpha|,|\beta|,|\gamma|\leq L}$. We approximate all expectations from time step $k$ to time step $k+1$ as follows: the product of two polynomial chaos expansions $\hat{w}_1 = \sum_{|\alpha|\leq L} \hat{w}_{1,\alpha} \hat{T}_{k,\alpha}$ and $\hat{w}_2 = \sum_{|\alpha|\leq L} \hat{w}_{2,\alpha} \hat{T}_{k,\alpha}$ can be approximated by projecting onto the basis $\{ \hat{T}_{k,\alpha}(x) \}_{|\alpha| \leq L}$:
        \begin{equation}
        \label{product calculation}
        \begin{aligned}
              w_1w_2\approx \Pi_{k,L}(w_1w_2)  = \sum_{|\alpha| \leq L}\left( \sum_{|\beta|, |\gamma|\leq L} \hat{w}_{1,\beta} \circ \hat{w}_{2,\gamma} \hat{\Gamma}_{k,\alpha\beta\gamma}\right)  \hat{T}_{k,\alpha},
        \end{aligned}
        \end{equation}
    and when computing the expectation of the product with another polynomial $\hat{w}_3 = \sum_{|\alpha|\leq L} \hat{w}_{3,\alpha} \hat{T}_{k,\alpha}$, the expectation can be approximated by
    \begin{equation}
    \label{approximated expectation}
    \mathcal{E}[\hat{w}_1\hat{w}_2\hat{w}_3] := \sum_{|\alpha|, |\beta|, |\gamma|\leq L} \hat{w}_{1,\alpha} \circ \hat{w}_{1,\beta}  \circ \hat{w}_{1,\gamma}  \hat{\Gamma}_{k,\alpha\beta\gamma}.
    \end{equation}
    Here, we use the notation $\mathcal{E}$ to represent the approximating expectation computed by projection coefficients and triple products. 

    In the original equations \eqref{b equation} and \eqref{sigma equation}, the polynomial chaos coefficients are defined via projection expectations with respect to the distribution of the current state $v_k$. Since this distribution is not explicitly available in the numerical implementation of the RPC method, an alternative strategy is required to obtain these coefficients. A natural approach is to seek a distribution that matches, at least approximately, the given orthonormal polynomials $\{\hat{T}_{k,\alpha}(x)\}_{|\alpha|\leq L}$ and the associated triple products $\{\hat{\Gamma}_{k,\alpha\beta\gamma}\}_{|\alpha|,|\beta|,|\gamma|\leq L}$, and then compute the Galerkin projections with respect to this distribution \cite{musco2025sharper,kong2017spectrum}. 
    
    In this paper, we approximate the numerical drift $b_h$ and diffusion $\sigma_h$ using polynomials $p_b$ and $p_\sigma$. This avoids explicit expectation evaluations, reducing the required projection coefficients to finite linear combinations of polynomial products, all of which are computable entirely within the polynomial chaos framework. 
    
    The solution at time step $t_{k+1}$ is then expressed as a linear combination of 
    $\{\hat{T}_{k,\alpha}, z_k^{(j)}\hat{T}_{k,\alpha}\}_{|\alpha|\leq L, 1\leq j\leq m}$. This representation enables us to apply \eqref{product calculation} and \eqref{approximated expectation} to approximate the required polynomial expectations at time $t_{k+1}$. From these, we can construct the approximate orthonormal polynomials $\{\hat{T}_{k+1,\alpha}\}_{|\alpha|\leq L}$ via the Gram-Schmidt procedure and then compute the associated triple products $\{\hat{\Gamma}_{k+1,\alpha\beta\gamma}\}_{|\alpha|,|\beta|,|\gamma|\leq L}$, thereby completing the iteration and advancing to the next time step. Different strategies for organizing and propagating these polynomial expectations lead to different implementations of the RPC method. See Algorithm \ref{algorithm3}, referred to as pRPC, implemented from the orthonormal polynomial perspective, and Algorithm \ref{algorithm}, referred to as mRPC, from the moment perspective, respectively.

\subsubsection{Implementation from the orthogonal polynomial  perspective (pRPC)}
   \label{sec:poly_implementation}

   We first describe the numerical implementation of RPC from the orthogonal polynomial perspective. This approach directly evolves the orthonormal basis and its associated algebraic structures. 

    Given the orthonormal polynomial basis functions $\{T_{k,\alpha}(v_k)\}$ and the triple products $\{\Gamma_{k,\alpha\beta\gamma}\}$, we compute the projection coefficients $b_{k,\alpha}$ and $\sigma_{k,\alpha}$ through equations (\ref{b equation}) and (\ref{sigma equation}).  The key challenge is then to evolve both the orthonormal basis and the triple products dynamically to the next time step. To achieve this, we apply an It\^o-Taylor expansion to the functions ${T}_{k,\alpha}(v_{k+1}) {T}_{k,\beta}(v_{k+1})$ and ${T}_{k,\alpha}(v_{k+1}) {T}_{k,\beta}(v_{k+1}) {T}_{k,\gamma}(v_{k+1})$ along the dynamics~\eqref{Euler type}. Specifically, if we truncate at $\mathcal{O}(h)$ to match the form of the Euler scheme, the inner products can be approximated by
\begin{equation}\label{grammatrix}
  \begin{aligned}
   \mathbb{E} [{T}_{k,\alpha}(v_{k+1}) 
      {T}_{k,\beta}(v_{k+1})]
    &\approx \delta_{\alpha\beta}
      + h \sum_{i=1}^{d}
        \mathbb{E} \left[
          \frac{\partial}{\partial v^{(i)}}
           ({T}_{k,\alpha} {T}_{k,\beta})
            b^{(i)}
        \right] 
      + \frac{h}{2}\sum_{i,l=1}^{d}\sum_{j=1}^{m}
        \mathbb{E} \left[
          \frac{\partial^2}{\partial v^{(i)}\partial v^{(l)}}
           ({T}_{k,\alpha} {T}_{k,\beta})
           \sigma^{(i,j)} \sigma^{(l,j)}
        \right],
  \end{aligned}
\end{equation}
where $\delta_{\alpha\beta}$ is the Kronecker delta arising from the orthonormality of ${T}_{k,\alpha}$, and all polynomial arguments are evaluated at $v_k$. Similarly, the triple products evolve according to
\begin{equation}\label{tripleproduct}
  \begin{aligned}
    &\mathbb{E} [{T}_{k,\alpha}(v_{k+1}) {T}_{k,\beta}(v_{k+1}) {T}_{k,\gamma}(v_{k+1})] \\
    &\approx \mathbb{E} [{T}_{k,\alpha}(v_k) {T}_{k,\beta}(v_k) {T}_{k,\gamma}(v_k)] + h\sum_{i=1}^{d}
        \mathbb{E} \left[
          \frac{\partial}{\partial v^{(i)}}
            ({T}_{k,\alpha} {T}_{k,\beta} 
             {T}_{k,\gamma}) b^{(i)}
        \right] \\
    &\quad + \frac{h}{2}\sum_{i,l=1}^{d}\sum_{j=1}^{m}
        \mathbb{E} \left[
          \frac{\partial^2}{\partial v^{(i)}\partial v^{(l)}}
            ({T}_{k,\alpha} {T}_{k,\beta} 
             {T}_{k,\gamma}) 
          \sigma^{(i,j)} \sigma^{(l,j)}
        \right].
  \end{aligned}
\end{equation}
Here, the symbol $\approx$ accounts exclusively for the It\^o--Taylor truncation at $\mathcal{O}(h)$, and the expectations $\mathbb{E}[\cdot]$ on the right-hand side are exact with respect to the distribution of $v_k$. 
  
In our pRPC algorithm, since the exact expectations and the orthonormal polynomials with respect to $v_k$ are not available, we denote by $\{T^{\pRPC}_{k,\alpha}(x)\}_{|\alpha|\leq L}$ and $\{\Gamma^{\pRPC}_{k,\alpha\beta\gamma}\}_{|\alpha|,|\beta|,|\gamma|\leq L}$ the corresponding orthonormal polynomials and triple products at time step $k$. We approximate the numerical drift $b_h$ and diffusion $\sigma_h$ by polynomials $p_b$ and $p_\sigma$, respectively. The resulting polynomial drift and diffusion terms can then be directly represented in the basis $\{T^{\pRPC}_{k,\alpha}(x)\}_{|\alpha|\leq L}$. Every $\mathbb{E}[\cdot]$ on the right-hand side of~\eqref{grammatrix} and~\eqref{tripleproduct} is replaced by the approximate expectation $\mathcal{E}[\cdot]$ defined in \eqref{approximated expectation}, evaluated via the polynomial chaos arithmetic at step $k$ using
$\{T^{\pRPC}_{k,\alpha}(x)\}_{|\alpha|\leq L}$ and $\{\Gamma^{\pRPC}_{k,\alpha\beta\gamma}\}_{|\alpha|,|\beta|,|\gamma|\leq L}$ and the multiplication rule~\eqref{product calculation}, leading to
\begin{equation}\label{grammatrix_RPC}
  \begin{aligned}
    \Hb^{\pRPC}_{k,\alpha \beta}
    &= \delta_{\alpha\beta} + h \sum_{i=1}^{d}
        \mathcal{E} \left[
          \frac{\partial}{\partial v^{(i)}}
           (T^{\text{pRPC}}_{k,\alpha} T^{\text{pRPC}}_{k,\beta})
           p_b^{(i)}
        \right] + \frac{h}{2}\sum_{i,l=1}^{d}\sum_{j=1}^{m}
        \mathcal{E} \left[
          \frac{\partial^2}{\partial v^{(i)}\partial v^{(l)}}
           (T^{\text{pRPC}}_{k,\alpha} T^{\text{pRPC}}_{k,\beta})
           p_\sigma^{(i,j)}p_\sigma^{(l,j)}
        \right],
  \end{aligned}
\end{equation}
and the approximated triple products evolve according to
\begin{equation}\label{tripleproduct_RPC}
  \begin{aligned}
  \widetilde{\Gamma}^{\pRPC}_{k,\alpha \beta \gamma} &= \Gamma^{\pRPC}
        _{k,\alpha \beta \gamma} 
    + h\sum_{i=1}^{d}
        \mathcal{E} \left[
          \frac{\partial}{\partial v^{(i)}}
           (T^{\text{pRPC}}_{k,\alpha} T^{\text{pRPC}}_{k,\beta} 
            T^{\text{pRPC}}_{k,\gamma}) p_b^{(i)}
        \right] \\&+ \frac{h}{2}\sum_{i,l=1}^{d}\sum_{j=1}^{m}
        \mathcal{E} \left[
          \frac{\partial^2}{\partial v^{(i)}\partial v^{(l)}}
           (T^{\text{pRPC}}_{k,\alpha} T^{\text{pRPC}}_{k,\beta} 
             T^{\text{pRPC}}_{k,\gamma}) 
           p_\sigma^{(i,j)} p_\sigma^{(l,j)}
        \right].
  \end{aligned}
\end{equation}

 After evolution, we apply Gram-Schmidt orthogonalization to $\{T^{\text{pRPC}}_{k,\alpha}(x)\}_{|\alpha| \leq L}$ with respect to the evolved inner product matrix $\Hb^{\pRPC} = (\Hb^{\pRPC}_{k,\alpha \beta})_{|\alpha|,|\beta| \leq L}$ to obtain the approximated orthonormal basis $\{T^{\text{pRPC}}_{k+1,\alpha}(x)\}_{|\alpha| \leq L}$ for $v_{k+1}$.  This procedure is equivalent to the Cholesky decomposition $\boldsymbol{L}_k\Hb^{\pRPC}_k\boldsymbol{L}_k^T = \Ib $. The new orthonormal basis $\{T^{\text{pRPC}}_{k+1,\alpha}\}_{|\alpha| \leq L}$ at step $k+1$ is then defined by
\begin{equation}
\label{new poly}
    T^{\text{pRPC}}_{k+1,\alpha}(x) = \sum_{|\beta| \leq L} \Lb_{k,\alpha\beta} T^{\text{pRPC}}_{k,\beta}(x).
\end{equation}
The evolved triple products are then transformed to the new basis via 
\begin{equation} \label{triple transform}
    \Gamma^{\pRPC}_{k+1,\alpha\beta\gamma} = \sum_{|\mu|,|\nu|,|\eta| \leq L} \Lb_{k,\alpha\mu} \Lb_{k,\beta\nu} \Lb_{k,\gamma\eta}      \widetilde{\Gamma}^{\pRPC}_{k,\mu \nu \eta}.
\end{equation} 
We summarize this implementation in Algorithm \ref{algorithm3}.

\begin{algorithm}[H]
    \caption{pRPC for d-dimensional SDEs using Euler-type scheme}
    \begin{algorithmic}[1]
    \STATE{\textbf{Input:} SDE parameters $(b, \sigma)$, time domain $[0,T] = \cup_{k=0}^{n-1}[t_{k},t_{k+1}]$ with time step $h$ and $t_k=kh$, initial condition $m_{0,\eta}=\mathbb{E}[v_0^\eta]$ for $|\eta| \leq 2L$, polynomial degree $L$, polynomials approximation $p_b$ and $p_\sigma$}
    \IF{$k = 0$}
    \STATE{Construct initial orthogonal polynomials $\{T^{\text{pRPC}}_{0,\alpha}(v_0\}_{|\alpha| \leq L}$ from moments $\{m_{0,\eta}\}_{|\eta| \leq 2L}$ using Gram-Schmidt}
    \STATE{Compute initial triple products $\{\Gamma^{\pRPC}_{0,\alpha\beta\gamma}\}_{|\alpha|,|\beta|,|\gamma|\leq L}$}
    \ENDIF
    \WHILE{$k = 0$ to $n$}
    \STATE{Compute the approximated projection coefficients $b_{k,\alpha}$ and $\sigma_{k,\alpha}$ by representing $p_b(x)$ and $p_\sigma(x)$ under the basis $\{T_{k,\alpha}^{\pRPC}(x)\}_{|\alpha|\leq L}$}
    
    \STATE{Evolve the approximated triple product $\Gamma^{\pRPC}_{k,\alpha \beta \gamma}$  via~\eqref{tripleproduct_RPC}}

    \STATE{Evolve the approximated inner products $\Hb^{\pRPC}_{k,\alpha\beta}$ via~\eqref{grammatrix_RPC}}

    \STATE{Construct new orthogonal polynomials $\{T^{\text{pRPC}}_{k+1,\alpha}(x)\}_{|\alpha| \leq L}$ from $\{T^{\text{pRPC}}_{k,\alpha}(x)\}_{|\alpha| \leq L}$ using Gram-Schmidt}
    
    \STATE{Transform $\widetilde{\Gamma}^\pRPC_{k,\alpha \beta \gamma}$ to $\Gamma^{\pRPC}_{k+1,\alpha\beta\gamma}$ for $|\alpha|, |\beta| , |\gamma| \leq L$ as \eqref{triple transform}}

    \STATE{$t_{k+1} \gets t_k + h$}
    \STATE{$k \gets k + 1$}
    \ENDWHILE
    
    \STATE{\textbf{Output:} Orthogonal polynomial representation of $u_n$ at final time $T=t_n$}
    \end{algorithmic}
    \label{algorithm3}
    \end{algorithm}

\subsubsection{Implementation from the moment perspective}
\label{sec:moment_implementation}

In Algorithm~\ref{algorithm3}, the computational cost of directly evolving $\mathcal{O}(\binom{d+L}{L}^3)$ triple products can be expensive in high dimensions. We now present an alternative implementation that directly evolves moments, which offers significant computational savings.

Unlike in Algorithm \ref{algorithm3}, where we directly track the orthonormal basis and triple products via an It\^o-Taylor expansion, here we reconstruct the orthonormal polynomials $\{T_{k,\alpha}(x)\}_{|\alpha|\leq L}$ and compute the triple products from moments. This reconstruction of $\{T_{k,\alpha}(x)\}_{|\alpha|\leq L}$  proceeds by applying Gram-Schmidt orthogonalization to the monomial basis $\{x^\alpha\}_{|\alpha| \leq L}$, for which we require moments up to degree $2L$. Since each polynomial in $\{T_{k,\alpha}(x)\}_{|\alpha|\leq L}$ has degree at most $L$ and is expressed as a linear combination of monomials, the triple product $\{\Gamma_{k,\alpha\beta\gamma}\}_{|\alpha|,|\beta|,|\gamma|< L}$ can involve monomials of degree up to $3L$. Computing the expectation of triple products thus requires moments $\mathbb{E}[v_k^\gamma]$ for $|\gamma| \leq 3L$. Therefore, we propagate moments up to degree $3L$ of each $v_k$ throughout the time stepping. 

 At each time step, since the Euler-type scheme (\ref{Euler type}) ensures the independence of $z_k$ and $v_k$, we can compute the mixed moments of $v_{k+1}$ using the multinomial expansion theorem. For any multi-index $\gamma = (\gamma_1, \cdots, \gamma_d)$ with $|\gamma| \leq 3L$, the mixed moments $\mathbb{E}[v_{k+1}^{\gamma}]$ are given by 
    \begin{equation}
    \label{moment_cal}
    \begin{aligned}
        \mathbb{E}[v_{k+1}^{\gamma}]
        &= \mathbb{E}\!\left[\left(
            v_k
            + h \sum_{|\alpha|\leq L} b_{k,\alpha}  T_{k,\alpha}(v_k)
            + \sqrt{h}\sum_{j=1}^{m} \sum_{|\alpha|\leq L}
              \sigma_{k,\alpha}^{(j)}  T_{k,\alpha}(v_k)  z_k^{(j)}
          \right)^{\!\gamma}\right] \\[4pt]
        &= \sum_{\substack{\mu+\nu+\sum_{j=1}^{m}\eta^{(j)}=\gamma}}
\frac{\gamma!}{\nu! \mu! \prod_{j=1}^m \eta^{(j)}!}
          h^{|\nu|+\frac{1}{2}\sum_{j=1}^{m}|\eta^{(j)}|}
          \prod_{j=1}^{m}\mathbb{E}\!\left[(z_k^{(j)})^{|\eta^{(j)}|}\right] \\[2pt]
        &\qquad\quad \cdot
          \mathbb{E}\!\left[
            v_k^{\mu}\!
            \left(\sum_{|\alpha|\leq L} b_{k,\alpha}  T_{k,\alpha}(v_k)\right)^{\!\nu}
            \prod_{j=1}^{m}
            \left(\sum_{|\alpha|\leq L}
              \sigma_{k,\alpha}^{(j)}  T_{k,\alpha}(v_k)\right)^{\!\eta^{(j)}}
          \right],
    \end{aligned}
\end{equation}
    where $\sigma_{k,\alpha }^{(j)}= (\sigma_{k,\alpha}^{(1,j)}, \cdots, \sigma_{k,\alpha}^{(d,j)})^T$ denotes the $j$-th column vector of the coefficient matrix $\sigma_{k,\alpha}$, $z_k = (z_k^{(1)},\cdots,z_k^{(m)})$, and $\mu, \nu, \eta^{(1)}, \cdots, \eta^{(m)} \in \mathbb{N}^d$ are multi-indices.

In the mRPC algorithm, we approximate the numerical drift $b_h$ and diffusion $\sigma_h$ by polynomials $p_b$ and $p_\sigma$, respectively. Let $\{m_{k,\gamma}^{\mRPC}\}_{|\gamma|\leq 3L}$ denote the set of propagated moments at time step $k$. From $\{m^{\mRPC}_{k,\gamma}\}_{|\gamma|\leq 3L}$, we reconstruct the orthonormal polynomials $\{T^{\text{mRPC}}_{k,\alpha}(x)\}_{|\alpha|\leq L}$ via the Gram-Schmidt procedure and then compute the corresponding triple products $\{\Gamma^{\mRPC}_{k,\alpha\beta\gamma}\}_{|\alpha|,|\beta|,|\gamma|\leq L}$. We then compute the moments $\{m^{\mRPC}_{k+1,\gamma}\}_{|\gamma|\leq 3L}$ using \eqref{moment_cal}, where $\mathbb{E}[\cdot]$ is replaced by $\mathcal{E}[\cdot]$. The polynomial expansions of $b_h$ and $\sigma_h$ are replaced, respectively, by the expansions of $p_b$ and $p_\sigma$ in the basis $\{T_{k,\alpha}^{\mRPC}\}_{|\alpha|\leq L}$. The evaluation of \eqref{moment_cal} at step $k$ then proceeds via polynomial chaos arithmetic, using $\{T^{\mRPC}_{k,\alpha}(x)\}_{|\alpha|\leq L}$, the triple products $\{\Gamma^{\mRPC}_{k,\alpha\beta\gamma}\}_{|\alpha|,|\beta|,|\gamma|\leq L}$, and the multiplication rule \eqref{product calculation}. This procedure can be applied recursively.

To reduce complexity and to remain consistent with the order of accuracy in \eqref{Euler type}, we neglect higher-order terms of $\mathcal{O}(h^\alpha)$ for $\alpha \geq 2$ in \eqref{moment_cal}. Retaining the terms with $|\nu| + \frac{1}{2}\sum_{j=1}^{m} |\eta^{(j)}| = 0, 1$, we obtain 
\begin{equation}
\label{eqfirst_order}
\begin{aligned}
m_{k+1,\gamma}^{\mRPC}
&= m_{k,\gamma}^{\mRPC} 
 + h \sum_{\substack{i=1 \\ \gamma_i \ge 1}}^{d} \gamma_i \, \mathcal{E}\!\left[ v_k^{\gamma - \mathbf{e}_i} p_b^{(i)} \right]   + h \sum_{j=1}^{m} \sum_{\substack{i,l=1 \\ \gamma_i \ge 1, \gamma_l \ge 1}}^{d} 
   \frac{\gamma_i (\gamma_l - \delta_{il})}{1 + \delta_{il}}\mathcal{E}\left[ v_k^{\gamma - \mathbf{e}_i-\mathbf{e}_l} p_{\sigma}^{(i,j)}p_{\sigma}^{(l,j)}\right] ,
\end{aligned}
\end{equation}
where $\mathbf{e}_i = (0,\dots,1,\dots,0)$ denotes the $i$-th standard basis vector, $\delta_{il}$ is the Kronecker delta satisfying $\delta_{il} = 1$ if $i = l$ and $\delta_{il} = 0$ otherwise, $\gamma - \mathbf{e}_i$ denotes the multi-index obtained by subtracting $1$ from the $i$-th component of $\gamma$,  and more generally $\gamma - e_i - e_l$  subtracts 
 $1$ from the $i$-th and $l$-th component of $\gamma$. This choice is consistent with the temporal order retained in our Euler-type numerical scheme. For extensions of the RPC framework to higher-order schemes, see Section~\ref{higher order}.

We summarize this implementation in Algorithm~\ref{algorithm}.

    \begin{algorithm}
    \caption{mRPC for d-dimensional SDEs using Euler-type discretizations}
    \begin{algorithmic}[1]
    \STATE{\textbf{Input:} SDE parameters $(b, \sigma)$, time domain $[0,T] = \cup_{k=0}^{n-1}[t_{k},t_{k+1}]$ with time step $h$ and $t_k=kh$, initial moments $m_{0,\eta}=\mathbb{E}[v_0^\eta]$ for $|\eta| \leq 3L$, polynomial degree $L$, polynomials approximation $p_b$ and $p_\sigma$}
    \WHILE{$k = 0$ to $n$}
    \STATE{Construct approximate orthonormal polynomials
      $\{T^{\text{mRPC}}_{k,\alpha}(v_k)\}_{|\alpha|\leq L}$ from
      $\{m^{\mRPC}_{k,\gamma}\}_{|\gamma|\leq 3L}$ via Gram--Schmidt}

    \STATE{Compute approximate triple products
      $\{\Gamma^{\mRPC}_{k,\alpha\beta\gamma}\}_{|\alpha|,|\beta|,|\gamma|\leq L}$ from
      $\{m^{\mRPC}_{k,\gamma}\}_{|\gamma|\leq 3L}$}
    \STATE{Compute the approximated projection coefficients $b_{k,\alpha}(x)$ and $\sigma_{k,\alpha}(x)$ by representing $p_b$ and $p_\sigma$ under the basis $\{T_{k,\alpha}^{\mRPC}(x)\}_{|\alpha|\leq L}$}

    \WHILE{$|\eta| \leq 3L$}
    \STATE{Compute approximate moments $m^{\mRPC}_{k+1,\gamma}$ for
      $|\gamma|\leq 3L$ via ~\eqref{eqfirst_order}}

    \STATE{Store moment $\{m^{\mRPC}_{k+1,\gamma}\}_{|\gamma|\leq 3L}$ for next iteration}
    \ENDWHILE
    
    \STATE{$t_{k+1} \gets t_k + h$}
    \STATE{$k \gets k + 1$}
    \ENDWHILE
    
    \STATE{\textbf{Output:} Approximate moments
      $\{m^{\mRPC}_{n,\gamma}\}_{|\gamma|\leq 3L}$ at final time~$T=t_n$,
      or full moment history
      $\{m^{\mRPC}_{k,\gamma}\}_{0\leq k\leq n, |\gamma|\leq 3L}$}
    \end{algorithmic}
    \label{algorithm}
    \end{algorithm}

\begin{Remark}
    Choosing between pRPC (Algorithm~\ref{algorithm3}) and mRPC (Algorithm~\ref{algorithm}) involves a trade-off between numerical stability and efficiency. pRPC is inherently well-conditioned, as it applies Gram-Schmidt to an inner product matrix $\hat{H}_{k,\alpha \beta} = \delta_{\alpha\beta} + \mathcal{O}(h)$ that represents only a small perturbation of the identity. Conversely, mRPC constructs a Hankel matrix from approximated moments, which may become ill-conditioned for large $L$. However, evolving $\mathcal{O}(\binom{d+3L}{3L})$ moments is much cheaper in computations and storage than evolving $\mathcal{O}(\binom{d+L}{L}^3)$ triple products, especially in high dimensions. Since a small $L$ typically ensures sufficient accuracy without severe ill-conditioning, this substantial computational advantage justifies our primary focus on the mRPC in the remainder of the paper.
\end{Remark}

\begin{Remark} \label{rem:nonMarkovian}
    While this paper focuses on the Markovian SDE~\eqref{original SDE}, the projection technique extends naturally to non-Markovian systems such as stochastic delay differential equations (SDDEs)~\cite{It1964ONSS}:
    \begin{equation}
        du(t) = b\big(u(t),  u(t-\tau)\big) dt
             + \sigma\big(u(t),  u(t-\tau)\big) dW(t),
    \end{equation}
    where $u(t)\in\mathbb{R}^d$, and the drift $b:\mathbb{R}^d\times\mathbb{R}^d\to\mathbb{R}^d$ and diffusion $\sigma:\mathbb{R}^d\times\mathbb{R}^d\to\mathbb{R}^{d\times m}$ 
    depend on both the current and delayed states. Euler-type schemes remain applicable to SDDEs~\cite{MAO2003215, Kuchler2000189}. For such systems, let $\ell=\lfloor\tau/h\rfloor$ denote the discretized delay and let $v(t_{k-\ell})$ be the numerical approximation of $u(t_k-\tau)$. The projection framework then involves projecting $v_{k+1}$ onto
    $\mathrm{span}\{\hat{T}_{k,\alpha}(v_k, v(t_{k-\ell})),
    z_k^{(j)}\hat{T}_{k,\alpha}(v_k, v(t_{k-\ell}))
    :|\alpha|\leq \infty,1\leq j\leq m\}$, where $\hat{T}_{k,\alpha}$ are orthonormal polynomials of total degree $|\alpha|$ with respect to the joint distribution of $(v_k, v(t_{k-\ell}))\in\mathbb{R}^{2d}$. This augmentation of the state space to recover a Markovian structure preserves the RPC framework while accounting for memory effects. A detailed analysis of this extension and its associated complexity analysis is left for future work.
\end{Remark}

    \begin{Remark}
        To simplify, the RPC algorithm and its analysis are presented for autonomous SDEs $(\ref{original SDE})$, as is standard in the literature $\cite{MAO2003215, mao2015truncated, higham2002strong}$. RPC naturally extends to non-autonomous SDE systems with time-inhomogeneous drift $b(t, u)$ and diffusion $\sigma(t, u)$ terms in \eqref{b equation} and \eqref{sigma equation}. In this setting, the projection step \eqref{vector projection} decouples temporal and stochastic components.  
    \end{Remark}
    
    \begin{Remark} 
        Standard choices for approximating polynomials $p_b$ and $p_\sigma$ include Chebyshev polynomials \cite{boyd2001chebyshev} for near-optimal approximation, Bernstein polynomials \cite{phillips2003bernstein} for shape preservation, and splines \cite{de1978practical} for local approximation. For high dimen\-sion\-al problems, tensor product constructions or sparse grid techniques can be used to alleviate the curse of dimensionality \cite{Adocksparse1}.
    \end{Remark}

\subsubsection{RPC framework for general higher-order schemes}
\label{higher order}
The proposed framework extends to general numerical schemes via the strong Itô-Taylor expansion \cite{Kloeden1992}. For a truncation order $\gamma \in \{0.5, 1.0, \dots\}$, let $\mathcal{M}_\gamma$ denote the standard hierarchical set of multi-indices $\beta = (j_1, \ldots, j_l)$. The corresponding coefficient functions $f_\beta$ are generated recursively via $f_\beta^{(i)} = \mathcal{L}^{j_1} \cdots \mathcal{L}^{j_l} u^{(i)}$, using the generalized Itô differential operators $\mathcal{L}^0 = \sum b^{(i)}\partial_{i} + \frac{1}{2}\sum \sigma^{(i,j)}\sigma^{(l,j)}\partial^2_{i,l}$ and $\mathcal{L}^j = \sum \sigma^{(i,j)}\partial_{i}$ for $j=1,\ldots,m$. Omitting the remainder term yields the order-$\gamma$ strong scheme:
\begin{equation}\label{Ito-taylor-expansion}
  v_{k+1} = v_k + \sum_{\beta\in\mathcal{M}_\gamma\setminus\{\emptyset\}} f_\beta(v_k) I_\beta^{(k)},
\end{equation}
where $I_\beta^{(k)}$ denotes the standard multiple stochastic integrals over $[t_k, t_{k+1}]$ \cite{Kloeden1992}.
 
To incorporate this scheme into our RPC framework \eqref{vector projection}, we project each coefficient function $f_\beta$ onto the orthonormal basis $\{{T}_{k,\alpha}(v_k)\}_{|\alpha|\leq \infty}$ at time $t_k$:
\begin{equation}
\label{coeff-projection}
  f_\beta^{(i)}(v_k) = \sum_{|\alpha|\leq \infty}
  c_{k,\alpha}^{(\beta,i)} T_{k,\alpha}(v_k),
  \quad i=1,\ldots,d,
\end{equation}
where the projection coefficients are computed by 
$c_{k,\alpha}^{(\beta,i)} = \mathbb{E}[f_\beta^{(i)}(v_k) T_{k,\alpha}(v_k)]$. Substituting \eqref{coeff-projection} into \eqref{Ito-taylor-expansion}, we obtain the discrete scheme
\begin{equation}
\label{discrete-scheme}
  v_{k+1} = v_k + \sum_{\beta\in\mathcal{M}_\gamma\setminus\{\emptyset\}}
  \left(\sum_{|\alpha|\leq \infty}
  c_{k,\alpha}^{(\beta)} T_{k,\alpha}(v_k)\right) I_\beta^{(k)},
\end{equation}
where $c_{k,\alpha}^{(\beta)} =
(c_{k,\alpha}^{(\beta,1)},\ldots,c_{k,\alpha}^{(\beta,d)})^T\in\mathbb{R}^d$.

This formulation shows that the higher-order Itô-Taylor scheme shares the same structure as the Euler-type scheme in equation (\ref{vector projection}), with the coefficients $\{c^{(\beta)}_{k,\alpha}\}_{\beta \in \mathcal{M}_\gamma, |\alpha| \leq \infty}$ generalizing the roles of $b_{k,\alpha}$ and $\sigma_{k,\alpha}$.  In practice, the orthonormal polynomials are truncated at total degree $L$ and replaced by ${\hat{T}_{k,\alpha}}$, and all expectations are replaced by the approximate expectation $\mathcal{E}[\cdot]$ evaluated using the truncated multiplication rule \eqref{product calculation}, as in the Euler-type setting. The two implementations from Section \ref{sec:poly_implementation} and Section \ref{sec:moment_implementation} extend directly. In pRPC, the Itô-Taylor expansion used to evolve the algebraic structures \eqref{grammatrix_RPC}--\eqref{tripleproduct_RPC} is performed with the generator incorporating all coefficient functions $\{f_\beta\}_{\beta \in \mathcal{M}_\gamma}$ at the appropriate truncation order; in mRPC, the multinomial expansion \eqref{moment_cal} is applied to \eqref{discrete-scheme} and truncated at the corresponding order, with the moments of the multiple stochastic integrals $\{I_\beta^{(k)}\}_{\beta \in \mathcal{M}_\gamma}$ computed analytically.

    As the truncation order $\gamma$ increases, the hierarchical structure of $\mathcal{M}_\gamma$ ensures that new terms are added systematically, which provides a natural framework for incorporating It\^{o}-Taylor schemes of arbitrary order into our RPC method.

    \subsubsection{Sparsity of RPC Algorithm}\label{sec:sparsity}

Let $S = \max\{\deg(p_{b}), 2\deg(p_{\sigma})\}$ denote the maximum degree involved in the evolution equation. In this section, we show that it suffices to evolve and store triple products whose indices belong to the restricted index set
\begin{equation}\label{eq:restricted_set}
  \mathcal{J}_{L,S}
  = \bigl\{(\alpha,\beta,\gamma)
      : |\alpha|,|\beta|,|\gamma|\leq L,
          |\alpha|+|\beta|+|\gamma|\leq J\bigr\},
  \qquad J := 2L+S,
\end{equation}
in the sense that all right-hand sides of \eqref{grammatrix_RPC} and \eqref{tripleproduct_RPC} can be evaluated using only triple products from~$\mathcal{J}_{L,S}$, so that the pRPC evolution is algebraically self-contained on this set. This also implies that in the mRPC algorithm, it suffices to evolve only the moments up to order $2L+S$.

In the pRPC, the right-hand sides of the approximated inner product evolution \eqref{grammatrix_RPC} and the approximated triple product evolution \eqref{tripleproduct_RPC} are computed from the triple products $\Gamma^{\pRPC}_{k,\alpha\beta\gamma}$ via the multiplication rule~\eqref{product calculation}. This requires that the collection of stored triple products is rich enough to evaluate all such expectations. If the full tensor is stored for all $|\alpha|,|\beta|,|\gamma|\leq L$, i.e., with total degree up to $3L$, the system is closed by \eqref{approximated expectation}, but the storage and computational cost scale as $\mathcal{O}\!\bigl(\binom{d+L}{L}^{3}\bigr)$. In the mRPC, the situation is analogous: computing $m^{\mRPC}_{k+1,\gamma}$ via the multinomial expansion \eqref{moment_cal} involves products and powers of polynomial chaos expansions that may exceed degree $L$. If moments up to degree $3L$ (and thereby the full triple product tensor up to total degree $3L$) are maintained, closure is guaranteed, but with the same storage cost.

A natural question is therefore whether both algorithms can be closed on a smaller index set while retaining high accuracy. We begin with the inner product evolution \eqref{grammatrix_RPC} in pRPC, which governs the Gram-Schmidt orthogonalization and is thus the most critical quantity for the accuracy of the orthonormal basis.  The most demanding term in \eqref{grammatrix_RPC}
is the diffusion contribution
\[
  \frac{h}{2}\sum_{i,l=1}^{d}\sum_{j=1}^{m}
  \mathcal{E}\!\left[
    \frac{\partial^{2}}{\partial v^{(i)}\partial v^{(l)}}
      \bigl(T^{\pRPC}_{k,\alpha} T^{\pRPC}_{k,\beta}\bigr)
    p_\sigma^{(i,j)} p_\sigma^{(l,j)}
  \right],
  \qquad |\alpha|,|\beta|\leq L.
\]
Introduce the diffusion covariance polynomial
$\Sigma^{(i,l)}:=\sum_{j=1}^{m}p_\sigma^{(i,j)}p_\sigma^{(l,j)}$, which has
degree at most $S$. The integrand
$\frac{\partial^{2}}{\partial v^{(i)}\partial v^{(l)}}
(T^{\pRPC}_{k,\alpha}T^{\pRPC}_{k,\beta})\cdot\Sigma^{(i,l)}$
has total degree at most
\[
  (|\alpha|+|\beta|-2)+S \leq 2L-2+S < 2L+S = J.
\]
 Therefore, the restricted set $\mathcal J_{L,S}$ in \eqref{eq:restricted_set} already contains all the information required to evaluate the inner product evolution exactly within the truncated polynomial arithmetic. This guarantees that the Gram-Schmidt orthogonalization, and hence the orthonormal basis at the next time step, is uniquely determined by the restricted set. 

The analysis for \eqref{tripleproduct_RPC} is more delicate because the target quantity $\Gamma^{\pRPC}_{k+1,\alpha\beta\gamma}$ itself carries three free indices with $|\alpha|+|\beta|+|\gamma|\leq J$. The following lemma shows that these fourth-order expectations can nevertheless be decomposed into bilinear expressions in triple products, all of which remain within~$\mathcal{J}_{L,S}$. See Appendix $\ref{proof of proposition 3.1}$ for a detailed proof.

\begin{Lemma}
\label{prop:closure}
Let $J = 2L+S$. For any $(\alpha,\beta,\gamma)\in\mathcal{J}_{L,S}$, the right-hand side of the triple product evolution~\eqref{tripleproduct_RPC} can be evaluated using only triple products $\{\Gamma^{\pRPC}_{k,\mu\nu\eta}\}_{(\mu,\nu,\eta)\in\mathcal{J}_{L,S}}$. Consequently, the restricted index set~$\mathcal{J}_{L,S}$ is closed under the pRPC evolution.
\end{Lemma}

To summarize, the restricted set $\mathcal{J}_{L,S}$ plays two distinct roles in the pRPC evolution. For the inner product evolution \eqref{grammatrix_RPC}, every triple product that arises in the computation has total degree strictly less than $J$, so the restricted set contains all the terms needed to evaluate the inner products exactly within the truncated polynomial arithmetic. This ensures the accuracy of the Gram-Schmidt orthogonalization, and hence the orthonormal basis at the next time step.  For the triple product evolution \eqref{tripleproduct_RPC}, the fourth-order expectations generated by the It\^o--Taylor expansion can be decomposed, via an adaptive projection whose level is chosen according to the index magnitudes, into bilinear expressions involving only triple products from $\mathcal{J}_{L,S}$, thereby allowing the triple products themselves to be propagated forward in time.

The same sparsity structure carries over to the mRPC algorithm by analogous arguments, where we only need to maintain the first $2L+S$ moments and compute the triple products $\{\Gamma^{\mRPC}_{k,\mu\nu\eta}\}_{(\mu,\nu,\eta)\in\mathcal{J}_{L,S}}$. The projection operator $\Pi_{k,L}$ in \eqref{product calculation} is reinterpreted as follows. For $w_1 = \sum_{|\alpha|\leq S} w_{1,\alpha} T^{\mRPC}_{k,\alpha}$ and $ w_2 = \sum_{|\alpha|\leq L} w_{2,\alpha} T^{\mRPC}_{k,\alpha}$, we define
\begin{equation}
    \label{eq:sparse projection operator}
    \Pi_{k,L}(w_1 w_2) = \sum_{\substack{|\alpha|\leq S, |\beta|\leq L, |\eta|\leq L}} w_{1,\alpha} \circ w_{2,\beta} \, \Gamma^{\mRPC}_{k,\alpha\beta\eta} \, T_{k,\eta}^{\mRPC}.
\end{equation}
In this setting, all indices remain within $\mathcal{J}_{L,S}$, and moments of arbitrary order can be computed recursively.

    \subsection{Computational Complexity}
    In this section, we analyze the computational cost of Algorithm~\ref{algorithm}, starting with the full index set $|\alpha|,|\beta|,|\gamma| \leq L$. For the $d$-dimensional case, the number of basis functions with total degree at most $L$ is $C_L = \binom{d+L}{L}$.

    \begin{table}[H]
    \centering
    \begin{tabular}{|l|c|}
    \hline
    \textbf{Operation} & \textbf{Complexity} \\
    \hline
    Orthogonal polynomial construction & $\mathcal{O}(C_L^3)$ \\
    Triple product computation & $\mathcal{O}(C_L^6)$ \\ 
    Moment updates & $\mathcal{O}(dm\cdot L \cdot C_L^3)$ \\
    \hline
    \textbf{Total per time step} & $\mathcal{O}(dm\cdot L \cdot C_L^3 + C_L^6)$ \\
    \hline 
    \end{tabular}
    \caption{Computational complexity of the method presented in Algorithm \ref{algorithm}.}
    \label{complexity}
    \end{table}
    
    The estimates in Table \ref{complexity} are obtained as follows. At each time step, the algorithm first constructs the orthonormal polynomials via the Gram-Schmidt procedure, which incurs a cubic cost relative to the basis size, yielding $\mathcal{O}(C_L^3)$. Next, evaluating the triple products $\Gamma^{\mRPC}_{k+1,\alpha\beta\gamma}$ over the full index set $\{|\alpha|, |\beta|, |\gamma| \leq L\}$ requires $\mathcal{O}(C_L^6)$ operations, since 
    computing the expectation for each of the $C_L^3$ index triples takes $\mathcal{O}(C_L^3)$ steps.  Because each polynomial multiplication costs $\mathcal{O}(C_L^3)$, summing over all indices in \eqref{eqfirst_order} for the moment updates leads to a dominant update complexity of $\mathcal{O}(d m L C_L^3)$. Combining these sequential stages yields the total per-step complexity shown in Table~\ref{complexity}.

In addition, parallel computation strategies can be optionally implemented for further acceleration when needed. By employing pre-computed linear index arrays to enable direct memory access for multi-dimensional moment indices, we eliminate synchronization overhead in parallel environments. This allows us to parallelize the computation of triple products $\widetilde{\Gamma}_{k,\alpha\beta\gamma}$ across different index combinations $(\alpha,\beta,\gamma)$. Moment updates for different orders can be evaluated simultaneously, since each moment evolution depends only on previously computed polynomial coefficients. Since these operations dominate the computational cost, parallelization delivers substantial performance gains scaling with the number of available CPU cores.

When the full index set of $C_L^3$ triples is replaced by the restricted set $\mathcal{J}_{L,S}$, with $N_{L,S}=|\mathcal{J}_{L,S}|$, a completely analogous analysis shows that the per-step complexity of mRPC reduces to  $\mathcal{O}(dm\cdot N_{L,S} \cdot L + N_{L,S}^2)$. If both $L$ and $d$ are large, such an index set leads to a significant reduction in time complexity.

\begin{Remark}
  While the worst-case complexity terms $C_L^3 dmL$ and $C_L^6$ appear significant for large values of $d$ and $m$, several mitigating factors render mRPC highly efficient in practice. When the drift coefficient $b$ and diffusion coefficient $\sigma$ are sufficiently smooth, a moderate truncation order $L$ typically suffices to achieve satisfactory accuracy, which makes the overall complexity tractable. In practice, sparse truncation techniques $\cite{luo2006wiener}$ can be employed to maintain accuracy while substantially reducing computational costs for both moment calculations and evolution processes. Furthermore, from an algorithmic perspective, the computation of triple products can be significantly accelerated via tensor contraction. By decomposing the full tensor summation into sequential matrix multiplications, the computational complexity can be reduced from $\mathcal{O}(C_L^6)$ to $\mathcal{O}(C_L^4)$.
    \end{Remark}

    \subsection{Main Convergence Results} \label{sec:main convergence results}

    We state the main convergence results in this section. We primarily focus on RPC combined with the Euler-type numerical scheme \eqref{Euler type}. To approximate the numerical drift and diffusion terms on a compact interval, we introduce the following Lipschitz truncation operator. For $R\in\mathbb{R}_+$, define a truncation function $ \chi_{R}\in C_{c}^{\infty}(\mathbb{R}^{d};\mathbb{R}^{d}) $ by 
    \begin{equation*}
        \chi_{R}(x) :=
        \left\lbrace 
        \begin{aligned}
            &x,\qquad x\in I_{R/3},\\
            &0,\qquad x\in\mathbb{R}^{d}\backslash I_{R},
        \end{aligned}
        \right.
    \end{equation*}
    where $ I_{R} = [-R,R]^d\subset \mathbb{R}^{d}$ such that $ \chi_{R}(x) $ decays smoothly on $ I_{R/3}^{c} $ and is Lipschitz continuous
    \begin{equation*}
        \vert \chi_{R}(x) - \chi_{R}(y)\vert \leq \vert x - y\vert. 
    \end{equation*}
    The existence of such a function $\chi_R$ is established by a standard mollification procedure applied to each coordinate $x^{(i)}$, $1\leq i\leq d$, using the continuous, piecewise smooth, and odd function $\chi(x)$ defined by $ \chi(x^{(i)}) = R/3 - |x^{(i)}-R/3| $ on $ [0, 2R/3] $ and extended by $ 0 $ outside $ [-2R/3,2R/3] $.

    Assume that we approximate $b_h$ and $\sigma_h$ on the compact set $I_R$ by polynomials $p_b^{R,S}$ and $p_\sigma^{R,S}$ with degrees $\deg(p_b^{R,S}) \le S$ and $2\deg(p_\sigma^{R,S}) \le S$, given by
    \begin{equation}\label{eq:approximate polynomials}
        p^{R,S}_b(x) = \sum_{|\alpha|\leq S} b_\alpha x^\alpha,\quad\big[p^{R,S}_\sigma(x)\big]^2 = \sum_{|\alpha|\leq S} \sigma_\alpha x^\alpha.
    \end{equation}
    We introduce the following numerical scheme \eqref{projection scheme} as a reference solution. For $0\leq k\leq n-1$, let $z_k$ be the same Brownian motion increments from time step $t_k$ to $t_{k+1}$ as in \eqref{Euler type}. Define
    \begin{equation}\label{projection scheme}
        \begin{aligned}
            v^{R,S}_{k+1} = \chi_{R}\left(v^{R,S}_{k} + p_{b}^{R,S}(v^{R,S}_{k})h +  p_{\sigma}^{R,S}(v^{R,S}_{k})z_{k}\sqrt{h}\right),  \quad     v^{R,S}(0) &= \chi_{R}(u_{0}).
        \end{aligned}
    \end{equation}
     
    The main assumptions we impose are summarized as follows.
    
    \begin{Asm}\label{asm:drift diffusion terms}
        (i) Assume that the numerical drift term $b_h$ and the numerical diffusion term $\sigma_h$ in~\eqref{Euler type} satisfy the following conditions: there exists a constant $K > 0$ such that for all $x, y \in \mathbb{R}^d$,
            \begin{equation}
                \begin{aligned}
                    |b_{h}(x)| &\leq K(|x| + 1),\quad \sup_{x \in \mathbb{R}^d} |\sigma_h(x)| < K,\\
                    |b_h(x) - b_h(y)| &\leq K|x - y|, \quad
                    |\sigma_h(x) - \sigma_h(y)| \leq K|x - y|. 
                \end{aligned}
            \end{equation}
            That is, $b_h$ has linear growth, $\sigma_h$ is uniformly bounded, and both $b_h$ and $\sigma_h$ are globally Lipschitz continuous.

            (ii) Assume that there exist a positive number $r$ and a constant $C>0$ such that for every $R>0$ and every integer $S\ge 0$, there exist polynomials $p_b^S$ and $p_\sigma^S$ satisfying $\max\{\deg(p_b^S),\,2\deg(p_\sigma^S)\}\le S$ and
            \begin{equation}
                \max\left\{ \sup_{x\in I_R} \big|b_h(x)-p_b^S(x)\big|,\; \sup_{x\in I_R} \big|\sigma_h(x)-p_\sigma^S(x)\big| \right\} \le C R^{r} S^{-r}.
            \end{equation}

            (iii) For each truncation level $R>R^*$, suppose 
         $L=\lfloor \lambda_1 R^\eta\rfloor, S=\lfloor \lambda_2 R^\eta\rfloor$ for some constant $\lambda_1\geq \lambda_2 >0$ and $ \eta \in (1,2)$ independent of $R$. Let
            \begin{equation}
            \label{Mb}
            M_{R,b} := \sum_{|\alpha|\le S} |b_\alpha|(2L+2S)^{|\alpha|},
            \qquad
            M_{R,\sigma} := \sum_{|\alpha|\le S } |\sigma_\alpha|(2L+2S)^{|\alpha|},
            \end{equation}
            and $M_{R,b,\sigma}:=M_{R,b}+\frac12 M_{R,\sigma}$, where the coefficients $b_\alpha,\sigma_\alpha$ are defined in \eqref{eq:approximate polynomials}.
            We assume that there exist constants $C_{b,\sigma}>0$ and $C_\sigma > 0$, independent of $R$, such that 
            \begin{equation}
                M_{R,b,\sigma}\le C_{b,\sigma} R^ \eta. 
            \end{equation}
    \end{Asm}

    \begin{Asm}\label{asm:initial condition}
        There exists a constant $\delta_0 > 0$ such that the initial condition $u_0$ in \eqref{original SDE} satisfies
        \begin{equation}
            \mathbb{E}\left[\exp\left(\delta_0 |u_0|^2\right)\right] < \infty.
        \end{equation}
    \end{Asm}
    
    \begin{Asm}\label{asm:stability}
        (i) For every $0 \le k \le n$, the  Hankel matrix of the mRPC algorithm $\mathbf{H}^{\RPC}_k \in \mathbb{R}^{C_L \times C_L}$ defined by $        (\mathbf{H}^{\RPC}_k)_{\alpha,\beta} = m_{k,\alpha+\beta}^{\mathrm{RPC}},  |\alpha|, |\beta| \le L$ with $C_L = \binom{L+d}{d}$, is positive definite.

        (ii) There exists a constant $C_\delta$, depending only on $K$ and $T$, such that the moments computed by the mRPC method  satisfy
        \begin{equation}
        \sup_{0 \le k \le n} |m_{k,\gamma}^{\mathrm{RPC}}| \le C_{\delta} \sqrt{|\gamma|!}
        \end{equation}
        for all multi-indices $\gamma$ with $|\gamma| \le 2L + 2S$.
    \end{Asm}

    \begin{Asm}\label{asm:positive measure}
        There exists a constant $h_0>0$ such that for all $h\in(0,h_0)$ and for each $t_k$, the distribution $\mu_k$ of the solution $v_k$ in \eqref{Euler type} is continuous and satisfies the following condition: there exists a constant $R^* > 0$ such that for every polynomial degree $q$, there exists a constant $\lambda(q) > 0$ such that for all $R > R^*$ and for every real coefficient vector $\{c_\alpha\}_{|\alpha|\le q}$ satisfying
        \[
        \sum_{|\alpha|\le q} c_\alpha^2 = 1 \quad \text{and} \quad \sum_{|\alpha|\le q} c_\alpha x^\alpha \ge 0 \ \text{ for all } x\in I_R,
        \]
        we have
        \begin{equation}
            \inf_{\substack{0 \leq k \leq n \\ \sum_{|\alpha| \leq q} c_\alpha^2 = 1}}  \int_{I_R}\sum_{|\alpha| \leq q} c_{\alpha} x^{\alpha}   \mu_{k}(dx) \geq \lambda(q).
        \end{equation}
    \end{Asm}

    \begin{Remark} \rm
        Assumption \ref{asm:drift diffusion terms} imposes regularity conditions on $b$ and $\sigma$. 
        For results on polynomial approximation, see, e.g., \cite{bagby2002multivariate,gautschi2004orthogonal,de2012mathematics}. 
        Whenever $b$ and $\sigma$ extend holomorphically to a complex neighborhood of $[-\lambda R^\eta,\lambda R^\eta]^d$ and have at most linear growth there, the coefficients of the corresponding polynomials decay sufficiently fast that the weighted sums $M_{R,b}$ and $M_{R,\sigma}$ remain of order $\mathcal O(R^\eta)$. 
        
        Assumption~\ref{asm:stability} is a stability assumption concerning the mRPC scheme. (i) guarantees that the Gram--Schmidt procedure on the monomial basis $\{x^\alpha\}_{|\alpha| \le L}$ with respect to the bilinear form $\langle x^\alpha, x^\beta \rangle_k := m_{k,\alpha+\beta}^{\mathrm{RPC}}$ yields a well-defined family of orthonormal polynomials. (ii) is the analogue, at the level of moment sequences, of the exponential square-integrability bound $\mathbb{E}[\exp(\delta |v_k^{R,S}|^2)] \le C_2$ established in Lemma~\ref{determinate lemma2}. While obtaining such estimates directly for some classes of equations \eqref{original SDE} would be desirable, these verifications are left for future work. Such estimates are directly verifiable numerically.       

    \end{Remark}
    \begin{Remark} \rm
         Assumption \ref{asm:positive measure} excludes cases where the distribution $\mu_k$ degenerates, for instance to a Dirac measure, or concentrates on a lower-dimensional manifold on which certain non-trivial polynomials vanish identically. A typical example is the small-noise SDE
        \begin{equation*}
            du(t) = b(u(t))   dt + \varepsilon \sigma(u(t))   dW(t)
        \end{equation*}
        with $\varepsilon \to 0$ and $b(u)=-c u$ for some $c>0$, whose distribution converges to a Dirac measure concentrated at the origin in that limit. The requirement of a uniform positive lower bound for $\lambda(q)$ over all indices $0 \leq k \leq n$ ensures that the sequence of distributions remains sufficiently non-degenerate throughout the entire evolution. 
    \end{Remark}

    Under Assumptions~\ref{asm:drift diffusion terms} - \ref{asm:initial condition}, we prove that the distribution of $v^{R,S}_k$ satisfies the following error estimate.  
    
    \begin{Lemma}\label{convergent thm}
        Under Assumptions \ref{asm:drift diffusion terms} and \ref{asm:initial condition}, there exist constants $ \delta> 0 $, $ C > 0 $, and $h^*>0$, depending on $ K $, $ T $, and $\delta_0$, such that for all $h\in(0,h^*)$,
        \begin{equation}
            \max_{0\leq k\leq n}\left(\mathbb{E}|v_k - v^{R,S}_k|^2\right)^{\frac{1}{2}} \leq C e^{-\delta  R^2}  + C R^r S^{-r},
        \end{equation}
        where $v_k$ is the solution of Euler-type scheme \eqref{Euler type} and $v_k^{R,S}$ is the solution of \eqref{projection scheme}.
    \end{Lemma}

    See Section~\ref{convergenceresults} for the proof.  Our main convergence result compares the moments obtained by the mRPC method (Algorithm \ref{algorithm}) with the corresponding moments of $v_k^{R,S}$. At time step $k$, let $\{ m_{k,\gamma}^{R,S} \}_{|\gamma|\leq J}$ denote the moments of $v_k^{R,S}$ up to order $2L+S$. Let $\{ m_{k,\gamma}^{\mathrm{RPC}} \}_{|\gamma|\leq J}$ be the moments computed via the mRPC method, where $J=2L+S$ is defined in \eqref{eq:restricted_set}. 
    
    We will prove in Lemma~\ref{determinate lemma2} below that $v_k^{R,S}$ is exponentially integrable; that is, there exist constants $\delta>0$ and $C>0$ independent of $k$ such that $\mathbb{E}[e^{\delta|v_k^{R,S}|^2}] \le C$.  To compare two such determinate distributions via their moments, we introduce the following metric. Let $\mu$ and $\nu$ be two continuous probability measures on $\mathbb{R}^d$ for which there exist $\delta_\mu > 0$ and $\delta_\nu > 0$, such that $    
    \int  e^{\delta_{\mu}|x|^2} d\mu(x) < \infty$ and $\int e^{\delta_{\nu}|x|^2} d\nu(x) < \infty$. 
    By Lemma~\ref{dense thm} below, $\mu$ and $\nu$ are determinate (i.e., uniquely determined by their moments). We define the distance between $\mu$ and $\nu$ by 
    \[
    d(\mu,\nu) = \sum_{|\gamma| \ge 0} \frac{|m_\gamma(\mu) - m_\gamma(\nu)|}{|\gamma|!}.
    \]
    The above bounds ensure that the distance is bounded. We now introduce a truncated version of the above metric for any two moment sequences $\{m_{1,\gamma}\}_{|\gamma|\leq J}$ and $\{m_{2,\gamma}\}_{|\gamma|\leq J}$:  
    \begin{equation}\label{metric}
        d\bigl(\{m_{1,\gamma}\}_{|\gamma|\leq J},\{m_{2,\gamma}\}_{|\gamma|\leq J}\bigr)  = \sum_{|\gamma|\leq J} \frac{|m_{1,\gamma}-m_{2,\gamma}|}{|\gamma|!}.
    \end{equation}
    We have the following error estimate for Algorithm~\ref{algorithm}. 
    
        \begin{Thm}\label{thm: moment error}
        Under Assumptions \ref{asm:drift diffusion terms} --
        \ref{asm:stability}, if we set $R  =  \delta^{-1/2}(-2\log h)^{1/2}$, then for every $N\ge0$, there exist constants $C_N>0$ and $h_N>0$, depending only on $N$, $K$, $T$, $d$, $C_{b,\sigma}$, $\lambda_1$, $\lambda_2$, and $\delta_0$, such that for all $h\in(0,h_N)$, we have 
\begin{equation}\label{spectral truncation}
\sup_{0\le k\le n}
d\left( \{ m_{k,\gamma}^{R,S} \}_{|\gamma|\le J},
       \{ m_{k,\gamma}^{\RPC} \}_{|\gamma|\le J} \right)
\le C_N(-\log h)^{-N},
\end{equation}
    where $\{m^{R,S}_{k,\gamma}\}_{|\gamma|\leq J}$ are the moments corresponding to the distribution of the solution $v_{k}^{R,S}$ of scheme \eqref{projection scheme}, and $\{m^{\RPC}_{k,\gamma}\}_{|\gamma|\leq J}$ are the moments computed by the mRPC algorithm.
    \end{Thm}

    Fix a moment degree $\ell$. By Theorem~\ref{thm: moment error}, the moments of order $|\gamma|\le \ell$ computed by the RPC method can be made sufficiently close to the corresponding moments of $v_k^{R,S}$.

    Under conditions described in detail in Section~\ref{sec:moment problem}, we can recover a probability measure from finitely many moments. For a given moment degree $\ell$, Lemma~\ref{lemma:distribution matching} below shows that there exists $h_{\tm}>0$ such that for all $h\le h_{\tm}$ and for $R = \delta^{-1/2}(-2\log h)^{1/2}$, there exists a probability measure $\mu_{k,l}^{\RPC}$ supported on $I_R = [-R,R]^d$ whose moments of order $|\gamma|\le \ell$ coincide with $\{m_{k,\gamma}^{\RPC}\}_{|\gamma|\le \ell}$. To quantify the difference between $\mu_k$ and $\mu_{k,\ell}^{\RPC}$, we employ the $1$-Wasserstein distance defined as follows:
\begin{equation*}
    W_{1}(\mu, \nu) = \inf_{\pi \in \mathscr{C}(\mu, \nu)} \int_{\mathbb{R}^{d} \times \mathbb{R}^{d}} |x - y| \, \pi(dx, dy),
\end{equation*}
    where $\mathscr{C}(\mu, \nu)$ denotes the set of all couplings of $\mu$ and $\nu$; that is, the set of probability distributions on $\mathbb{R}^d \times \mathbb{R}^d$ with marginals $\mu$ and $\nu$, respectively. We then obtain the following estimate:
    \begin{Thm}\label{W1distance}
    Let $\ell\in\mathbb N$ be a fixed moment-matching order and let $s\in\bigl(0,\tfrac{1}{2}\bigr)$.
    Under Assumptions~\ref{asm:drift diffusion terms}--\ref{asm:positive measure}, if we set $R  =  \delta^{-1/2}(-2\log h)^{1/2}$, then there exist constants $C > 0$ and $h_{\ell}$ that depend on $d$, $K$, $T$, $\lambda_1$, $\lambda_2$, $\delta$, $\eta$, $s$, $r$, such that, for all $h\in (0,h_{\ell})$, we have
        \begin{equation}
    \begin{aligned}
W_1 \bigl(\mu_k, \mu_{k,\ell}^{\mathrm{RPC}}\bigr)
\le C \ell^{-s} + C(-\log h)^{-\frac{r(\eta-1)}{2}},
    \end{aligned}
    \end{equation}
    where $\mu_k$ is the distribution of $v_k$ in \eqref{Euler type} and $\mu_{k,\ell}^{\RPC}$ is the distribution that matches the moments $\{m_{k,\gamma}^{\RPC}\}_{|\gamma|\leq \ell}$.
    \end{Thm}

       In particular, for a prescribed error tolerance $\varepsilon>0$, we set $\ell = \left\lceil \left(\frac{1}{\varepsilon}\right)^{1/s} \right\rceil$, $h \le \exp\left(-\varepsilon^{-2/(r(\eta-1))}\right)$, $R = \delta^{-1/2}\bigl(-2\log h\bigr)^{1/2}$,  
        and take $L = \lfloor \lambda_1 R^\eta \rfloor$ and $S = \lfloor \lambda_2 R^\eta \rfloor$ as in Assumption~\ref{asm:drift diffusion terms}. Then, the $W_1$ distance between the distribution that matches the moments $\{m_{k,\gamma}^{\RPC}\}_{|\gamma|\leq \ell}$ and the distribution of the numerical solution $v_k$ of \eqref{Euler type} can be bounded by the tolerance $\varepsilon$.
        
        Under the proposed assumptions and for the above choice of parameters $(h,R,S,L)$, Theorem~\ref{thm: moment error} shows that the mRPC scheme can accurately compute the moments of $v^{R,S}$. Theorem~\ref{W1distance} then shows that a measure reconstructed from the mRPC moments remains close to the reference projected distribution of $v^{R,S}$ in the Wasserstein distance. 

    The proofs of Theorem~\ref{thm: moment error} and Theorem~\ref{W1distance}, together with the additional technical lemmas and results used in the analysis, are presented in Section \ref{sec:proof of main result}.

    \section{Convergence Analysis}
    \label{sec:proof of main result}
    In this section, we provide a more detailed discussion of the results presented in Section~\ref{sec:main convergence results}. In Section~\ref{sec:moment problem}, we review the moment problem, provide conditions for distribution determinacy, and give estimates for controlling the $W_1$ distance from moment estimates. In Section~\ref{convergenceresults}, we establish the exponential integrability of the solution $v_k^{R,S}$ defined in \eqref{projection scheme} and prove Lemma~\ref{convergent thm}. In Section~\ref{sec:errors from expectation approximation}, we present the details and proofs of Theorem~\ref{thm: moment error} and Theorem~\ref{W1distance}.
     We use $ \vert\cdot\vert $ to denote the Euclidean distance in $ \mathbb{R}^{d} $ and for a matrix $ A\in\mathbb{R}^{d\times m} $, $ |A| = \sqrt{\rm{tr}(AA^{T})} $.

    \subsection{Moment Problem and Orthonormal Polynomials}
    \label{sec:moment problem}
    Our moment-based method requires, as in any dynamical PCE-based approach \cite{Bal2,bal2017300}, the completeness of the orthogonal polynomials constructed at each time step. This is guaranteed under appropriate assumptions on the measures generated by the random variables $v_k$. See, e.g., \cite{akhiezer2020classical,ernst2012convergence,petersen1982relation} for references on the moment problem. We need the following definition.
    \begin{Def}[Hamburger moment problem]\label{moment problem}
        Given an infinite sequence of real numbers $\{m_\alpha\}_{\alpha\in\mathbb{N}_{0}^{d}}$, the moment problem is said to be uniquely solvable for a positive Radon measure $ \mu(dx) $ on $ (\mathbb{R}^{d}, \mathscr{B}(\mathbb{R}^{d})) $, or equivalently, the distribution is called determinate if it is uniquely determined by its moments $ m_{\alpha} = \int_{\mathbb{R}^{d}}x^\alpha\mu(dx),  \alpha \in \mathbb{N}^{d}_0 $.
    \end{Def}

    Many conditions ensure that a distribution $\mu$ is determinate, such as the distribution being compactly supported, or the moments satisfying certain growth conditions \cite{ernst2012convergence, akhiezer2020classical}. For instance, Gaussian and uniform distributions are determinate, whereas the log-normal distribution is a classic example of an indeterminate distribution. The moment problem is intrinsically linked to the density of polynomial spaces. Here, we collect two main results that are sufficient for our purpose. Denote by $L^p(\mathbb{R}^{d}, \mu)$ for $p\in [1,\infty)$ the standard space of $\mu$-measurable functions such that $ \int_{-\infty}^\infty |f(x)|^p  d\mu(x) < \infty $, and by $\mathcal{P}(\mathbb{R}^{d}) $ the linear space of polynomials defined on $ \mathbb{R}^{d} $. We have the following proposition.
    \begin{Prop}\label{dense thm}
        Let $\mu$ be a continuous distribution with finite moments of all orders on $(\mathbb{R}^{d}, \mathscr{B}(\mathbb{R}^{d}))$. Then $\mu$ is determinate if either of the following conditions holds:
        
        \text{(i)} $\mu$ has compact support;
        
        \text{(ii)} There exists $\alpha > 0$ such that 
        \begin{equation}
            \int_{\mathbb{R}^d} e^{\alpha|x|} d\mu(x) < \infty.
        \end{equation}
    \end{Prop}
    
    We then consider the truncated moment problem, which concerns the correspondence between a given truncated moment set $\{m_{\gamma}\}_{|\gamma|\le q}$ of order up to $q$ and a distribution. Specifically, the problem asks whether there exists a distribution whose moments of order $\gamma$ coincide with $m_{\gamma}$ for all $|\gamma|\leq q$. For a given $\hat{R}>0$, denote $I_{\hat{R}} = [-\hat R,\hat R]^{d}$. Let $\mathcal{P}_L(I_{\hat R}) =\mathrm{span}\{ x^\gamma : |\gamma|\le L,\; x\in I_{\hat R}\}$ denote the space of real polynomials in $d$ variables of degree at most $L$ defined on $I_{\hat R}$. Define a linear functional $\mathcal{T}_L : \mathcal{P}_L(I_{\hat{R}}) \to \mathbb{R}$ by its action on the basis monomials:
    \begin{equation}
    \label{eq32}
        \mathcal{T}_L(x^\gamma) = m_{\gamma} \quad\text{for all  } |\gamma|\le L.
    \end{equation}
    We have the following proposition; see \cite{petersen1982relation,schmudgen2017moment} for details.
    \begin{Prop}\label{truncated determinant}
        If $\mathcal{T}_L$ satisfies $\mathcal{T}_L(p) \ge 0$ for every polynomial $p \in \mathcal{P}_L(I_{\hat R})$ that is nonnegative on $I_{\hat R}$, then there exists a positive Radon measure $\mu_L$ supported on $I_R$ such that for all $|\gamma| \le L$,
        \begin{equation*}
            m_{\gamma} = \int_{\mathbb{R}^d} x^\gamma  d\mu_L.
        \end{equation*}
    \end{Prop}

    Now, we analyze the error between two probability distributions given the error in their truncated moments.  We introduce the following lemma, which can be viewed as a d-dimensional generalization of a result in \cite{kong2017spectrum}; the proof, which relies mainly on techniques from \cite{musco2025sharper}, is deferred to Section~\ref{proof of Lemma 4.6}. 
    \begin{Lemma}\label{moment matching}
        Let $\mu$ and $\nu$ be two probability measures supported on $I_{\hat R} = [-\hat R, \hat R]^d$. For a multi-index $\gamma$, the $\gamma$-th moments of $\mu$ and $\nu$ are denoted by $m^\mu_{\gamma}$ and $m^\nu_{\gamma}$, respectively. If all moments up to total degree $q$ are given, then the Wasserstein distance $W_1(\mu, \nu)$ is bounded by:
        \begin{equation}\label{eq:boundW1}
            W_1(\mu, \nu) \leq \frac{C_d  \hat R}{q}
            + g(q)\hat R
              \sqrt{\sum_{|\gamma| \leq q}
              \left( \frac{m^\mu_{\gamma} - m^\nu_{\gamma}}{\hat R^{|\gamma|}} \right)^{2}},
        \end{equation}
        where $g(q) = C_d 3^q$ and $C_d$ depends only on $d$. 
    \end{Lemma}
    
    \subsection{Errors from polynomial approximation}
    \label{convergenceresults}
    This section establishes estimates for the reference solution $v_k^{R,S}$ in \eqref{projection scheme}. We first present an exponential convergence result for the numerical solutions obtained via \eqref{Euler type} and \eqref{projection scheme}. The detailed proof is provided in Appendix $\ref{proof of Lemma 3.3}$.
    \begin{Lemma}\label{determinate lemma2}
        Under Assumptions \ref{asm:drift diffusion terms} and \ref{asm:initial condition}, there exist constants $\hat \delta > 0$, $C> 0$, and $h^* > 0$, depending on $ K $, $ T $, and $\delta_0$, such that for all $h\in(0,h^*)$, we have
        \begin{equation} 
            \sup_{0 \leq k \leq n} \mathbb{E}\left[\exp\left(\hat\delta |v_k|^2\right)\right] \leq C.
        \end{equation}
        Moreover, if $R$ and $S$ satisfy $RS^{-1} \le 1$, then
        \begin{equation} 
            \sup_{0 \leq k \leq n} \mathbb{E}\left[\exp\left(\hat\delta |v^{R,S}_k|^2\right)\right] \leq C.
        \end{equation}
    \end{Lemma}

    We introduce the following truncated Euler-type numerical scheme:
    \begin{equation}\label{turncation scheme}
        \begin{aligned}
            v^{R}_{k+1} =\chi_R\left(v^{R}_{k} + b_{h}(v^{R}_k)h + \sigma_{h}(v^{R}_{k}) z_k\sqrt{h}\right), \quad v^{R}_0 &= \chi_{R}(u_{0}).
        \end{aligned}
    \end{equation}
    We establish the following error estimates for the truncated scheme \eqref{turncation scheme}, whose detailed proof is presented in Appendix $\ref{proof of Lemma 3.4}$.    
    \begin{Lemma}\label{Lemma 1}
        Under Assumptions \ref{asm:drift diffusion terms} and \ref{asm:initial condition}, there exist positive constants $ \delta > 0 $, $ C > 0 $, and $h^*>0$, depending on $ K $, $ T $, and $\delta_0$, such that for all $h\in(0,h^*)$, the solution of the numerical scheme (\ref{turncation scheme}) satisfies
        \begin{equation}
            \max_{0\leq k\leq n}\left(\mathbb{E}|v_k- v^{R}(t_{k})|^{2}\right)^{\frac{1}{2}} \leq Ce^{-\delta R^{2}}.
        \end{equation}
    \end{Lemma}
    From Lemma~\ref{Lemma 1}, we establish the convergence result, namely that the numerical scheme (\ref{projection scheme}) can effectively approximate the solution of the originally proposed numerical scheme (\ref{Euler type}).
    \begin{proof}[Proof of Lemma~\ref{convergent thm}]
        Let $\hat{v}^{R}_{k} =v^{R}_{k} + b_{h}(v^{R}_{k})h$ and $\hat{v}^{R,S}_{k} =v^{R,S}_{k} + b_{h}(v^{R,S}_{k})h$. From the formulation of the numerical schemes (\ref{turncation scheme}) and (\ref{projection scheme}), and using the Lipschitz property of $\chi_R$, we obtain
        \begin{equation}\label{thm2 eq2}
            \begin{aligned}
                &\mathbb{E}|v_{k+1}^{R} - v^{R,S}_{k+1}|^{2} = \mathbb{E}|\chi_{R}\left( \hat{v}^{R}_{k} + \sigma_{h}(v_{k}^{R}) z_{k}\sqrt{h} \right) - \chi_{R}\left( \hat{v}^{R,S}_{k} + p_{\sigma}^{R,S}(v_{k}^{R,S}) z_{k}\sqrt{h} \right) |^{2}\\
                \leq & \mathbb{E}|\hat{v}^{R}_{k} - \hat{v}^{R,S}_{k} + (\sigma_{h}(v_{k}^{R})- p_{\sigma}^{R,S}(v_{k}^{R,S}) z_{k}\sqrt{h}|^{2}\\
                =  & \mathbb{E}|v^{R}_{k} - v^{R,S}_{k} + b_{h}(v_{k}^{R})h - p_{b}^{R,S}(v_{k}^{R,S})h|^{2} + h\mathbb{E}|\sigma_{h}(v_{k}^{R})- p_{\sigma}^{R,S}(v_{k}^{R,S}) |^{2}.
            \end{aligned}
        \end{equation} 
        By the triangle inequality, the Lipschitz properties, and Assumption $\ref{asm:drift diffusion terms}$, the first expectation on the last line of equation \eqref{thm2 eq2} can be bounded by
                \begin{equation}\label{thm2 eq1}
            \begin{aligned}
                &\left[\left(\mathbb{E}|v^{R}_{k} - v^{R,S}_{k} + b_{h}(v_{k}^{R})h - b_{h}(v_{k}^{R,S})h|^{2}\right)^{\frac{1}{2}} + \left(\mathbb{E}|b_{h}(v_{k}^{R,S}) - p_{b}^{R,S}(v_{k}^{R,S})h|^{2}\right)^{\frac{1}{2}}\right]^{2}\\
                \leq &\left[(1+Kh)\left(\mathbb{E}|v_{k}^{R} - v_{k}^{R,S}|^{2}\right)^{\frac{1}{2}} + hCRS^{-1}\right]^{2}\leq (1+h)(1+Kh)^{2}\mathbb{E}|v_{k}^{R} - v_{k}^{R,S}|^{2} + Ch(1+h)R^{r}S^{-r}.
            \end{aligned} 
        \end{equation}
        Here, we use the following type of inequality $ (\sqrt{a} + h\sqrt{b} )^{2} \leq a + 2h\sqrt{ab} + h^{2}b\leq a + h(a+b) + h^{2}b =  (1+h)a + h(1+h)b $ for the last estimate in \eqref{thm2 eq1}. The second expectation on the last line of equation (\ref{thm2 eq2}) can be bounded by
        \begin{equation}\label{thm2 eq4}
            \begin{aligned}
                &2\big[\mathbb{E}|\sigma_{h}(v_{k}^{R})- \sigma_{h}(v_{k}^{R,S})|^{2}+\mathbb{E}|\sigma_{h}(v_{k}^{R,S})-p_{\sigma}^{R,S}(v_{k}^{R,S})|^{2}\big]
            \leq  2K^{2}\mathbb{E}|v_{k}^{R} - v_{k}^{R,S}|^{2} + 2CR^{r}S^{-r}.
            \end{aligned}
        \end{equation}
        Combining equations (\ref{thm2 eq2}), (\ref{thm2 eq1}), (\ref{thm2 eq4}), there exists a constant $ C $, depending on $ K $ and $ T $, such that
        \begin{equation*}
            \mathbb{E}|v_{k+1}^{R} - v^{R,S}_{k+1}|^{2}\leq (1+Ch)\mathbb{E}|v_{k}^{R} - v_{k}^{R,S}|^{2} + ChR^{r}S^{-r}.
        \end{equation*}
        By iteration, we obtain
        \begin{equation*}
            \begin{aligned}
                \mathbb{E}|v_{k}^{R} - v^{R,S}_{k}|^{2}< e^{CT}CR^{r}S^{-r} \cdot\frac{1+CT}{C}.
            \end{aligned}
        \end{equation*}
        Then, from Lemma \ref{Lemma 1}, we obtain the final estimate via the triangle inequality.
    \end{proof}

    \begin{Remark} \rm
        With an appropriate contraction-type Lipschitz condition on the drift and diffusion terms, Lemma~\ref{convergent thm} can be generalized to the long-time regime, so that the error bound holds uniformly for all $k \geq 0$. We refer to \cite{yuan2004stability,weng2019invariant,bal2017300} for the corresponding assumptions and techniques.
    \end{Remark}

    \subsection{Errors from expectation approximation}\label{sec:errors from expectation approximation}
    In this section, we present an error analysis of the mRPC method implemented in Algorithm~\ref{algorithm} and prove our main results, Theorem $\ref{thm: moment error}$ and Theorem $\ref{W1distance}$.
    
    At time step $k$, let $\{ m_{k,\gamma}^{R,S} \}_{|\gamma|\leq J}$ denote the moments of $v_k^{R,S}$ up to order $2L+S$. Let $\{ m_{k,\gamma}^{\mathrm{RPC}} \}_{|\gamma|\leq J}$ be the moments computed via the mRPC method. We denote by $\{T^{R,S}_\alpha(x)\}_{|\alpha| \leq L}$ the resulting family of orthonormal polynomials and $\{ \Gamma^{R,S}_{\alpha\beta\gamma} \}_{\multiset}$ the triple products computed from $\{ m_{k,\gamma}^{R,S} \}_{|\gamma|\leq J}$. Denote by $\{T^{\mathrm{RPC}}_{k,\alpha}\}_{|\alpha|\leq L}$ and $\{\Gamma^{\mathrm{RPC}}_{k,\alpha\beta\gamma}\}_{(\alpha,\beta,\gamma)\in\ILS}$ the corresponding orthonormal polynomials and triple products computed by mRPC moments.
    
    We introduce an intermediate mRPC procedure that starts from the moments $\{ m_{\gamma}^{R,S} \}_{|\gamma|\leq J}$ and applies one step of the mRPC evolution with all expectations approximated by the $\mathcal{E}$ operator using $\{T^{R,S}_{k,\alpha}\}_{|\alpha|\leq L}$ and $\{\Gamma^{R,S}_{k,\alpha\beta\gamma}\}_{(\alpha,\beta,\gamma)\in\ILS}$.  This yields the intermediate moments $\{ m^{\m}_{k+1,\gamma} \}_{|\gamma|\leq J}$. At time step $t_{k+1}$, we thus have three distinct sets of moments. Their relationships are summarized as follows:
    \begin{equation*}\label{evolution relation}
        \begin{aligned}
            \{ m_{k,\gamma}^{R,S} \}_{|\gamma|\le J} 
            &\xrightarrow{\eqref{projection scheme}} 
            \{ m_{k+1,\gamma}^{R,S} \}_{|\gamma|\le J}, \\
            \{ m_{k,\gamma}^{\RPC} \}_{|\gamma|\le J} 
            &\xrightarrow{\mRPC} 
            \{ m_{k+1,\gamma}^{\RPC} \}_{|\gamma|\le J}, \\
            \{ m_{k,\gamma}^{R,S} \}_{|\gamma|\le J} 
            &\xrightarrow{\mRPC} 
            \{ m_{k+1,\gamma}^{\m} \}_{|\gamma|\le J}.
        \end{aligned}
    \end{equation*}   
    Using the distance $d(\cdot,\cdot)$ defined in \eqref{metric}, we quantify the approximation errors as follows:
    \begin{equation}\label{defined errors}
        \begin{aligned}
            \Delta_{k} &= d\left( \{ m_{k,\gamma}^{R,S} \}_{|\gamma|\le J}, \{ m_{k,\gamma}^{\RPC} \}_{|\gamma|\le J} \right), \\
            \Delta^{(1)}_{k+1} &= d\left( \{m_{k+1,\gamma}^{R,S} \}_{|\gamma|\le J}, \{ m_{k+1,\gamma}^{\m} \}_{|\gamma|\le J} \right), \\
            \Delta^{(2)}_{k+1} &= d\left( \{ m_{k+1,\gamma}^{\RPC} \}_{|\gamma|\le J}, \{ m_{k+1,\gamma}^{\m} \}_{|\gamma|\le J} \right).
        \end{aligned}
    \end{equation}
    Here, $\Delta_k$ measures the total error accumulated up to step $k$. $\Delta^{(1)}_{k+1}$ is the one-step RPC error, which bounds how the error at step $k$ (the difference between $\{ m_{k,\gamma}^{R,S} \}_{|\gamma|\le J}$ and $\{ m_{k,\gamma}^{\RPC} \}_{|\gamma|\le J}$) propagates under the same mRPC update rule. $\Delta^{(2)}_{k+1}$ is the accumulation error, which bounds the discrepancy introduced by performing one mRPC evolution instead of the exact numerical scheme \eqref{projection scheme}, starting from the same moment set $\{ m_{k,\gamma}^{R,S} \}_{|\gamma|\le J}$. By the triangle inequality, $\Delta_{k+1}\leq \Delta_{k+1}^{(1)} + \Delta_{k+1}^{(2)}$.

    In the following proof, we assume that $R$ and $S$ are chosen such that $R^rS^{-r}$ is sufficiently small in \eqref{asm:drift diffusion terms} (in particular, $R^rS^{-r}\le 1$ and Lemma~\ref{determinate lemma2} holds), so that the solution $v_k^{R,S}$ of the projected scheme \eqref{projection scheme} remains close to the exact solution $v_k$ of the Euler-type scheme \eqref{Euler type} by Theorem~\ref{convergent thm}.

    We first prove the following two lemmas on moment bounds.
    \begin{Lemma}\label{scaled_moment_bound}
        Under Assumptions \ref{asm:drift diffusion terms} and \ref{asm:initial condition}, there exist constants $C_{\delta}$ and  $h^*>0 $,  depending only on $K$, $T$, and $\delta_0$, such that for all multi-indices $\gamma$ and $h\in(0,h^*)$,
        \begin{equation}
        \label{lemma310}
            \sup_{0\leq k\leq n}|\mathbb{E}[(v_k^{R,S})^\gamma]| \le \sup_{0\leq k\leq n}\left(|\mathbb{E}[|v_k^{R,S}|^{2|\gamma|}]|\right)^{\frac{1}{2}}\leq C_{\delta} \sqrt{|\gamma|!}.
        \end{equation}
    \end{Lemma}
    \begin{proof}
        By Lemma~\ref{determinate lemma2}, there exist constants $\delta>0$ and $C>0$ such that
        \begin{equation}
            \sup_{0\le k\le n} \mathbb{E}\left[\exp\left(\delta |v_k^{R,S}|^2\right)\right] \le C.
        \end{equation}
        For any multi-index $\gamma$, using the inequality $\exp(\delta |v|^2) \ge \frac{(\delta |v|^2)^{|\gamma|}}{|\gamma|!}$ and the Cauchy-Schwarz inequality, we obtain
        \begin{equation}
            \label{eq:moment bdd est}
            \big|\mathbb{E}[(v_k^{R,S})^\gamma]\big| \le \big(\mathbb{E}[|v_k^{R,S}|^{2|\gamma|}]\big)^{1/2} \le \big(|\gamma|! \delta^{-|\gamma|} \mathbb{E}[e^{\delta |v_k^{R,S}|^2}]\big)^{1/2} \le \sqrt{C |\gamma|!} \delta^{-|\gamma|/2}.
        \end{equation}
        Since $\sqrt{|\gamma|!}$ grows faster than any exponential function of the form $\delta^{-|\gamma|}$ for any fixed $\delta>0$, estimate \eqref{eq:moment bdd est} implies \eqref{lemma310}.
    \end{proof}
    
     \begin{Lemma}\label{mRPC bound}
        For all $0\le k\le n-1$, the approximate expectation $\mathcal{E}$ satisfies the following. If $|\gamma| \le 2L+S$, then
        \begin{equation*}
            \Ecal[(v_k^{R,S})^\gamma] = \mathbb{E}[(v_k^{R,S})^\gamma];
        \end{equation*}
        if $2L+S < |\gamma| \le 2L+2S$, then there exists a decomposition $\gamma = \gamma_1 + \gamma_2$ with $|\gamma_1|, |\gamma_2| \le 2L+S$ such that
        \begin{equation*}
            \big|\Ecal[(v_k^{R,S})^\gamma]\big| \le \bigl(\mathbb{E}[(v_k^{R,S})^{2\gamma_1}]\bigr)^{1/2} \bigl(\mathbb{E}[(v_k^{R,S})^{2\gamma_2}]\bigr)^{1/2}.
        \end{equation*}
        Here, the approximation $\mathcal{E}$ of the expectation is computed using $\{T^{R,S}_\alpha(x)\}_{|\alpha| \leq L}$ and $\{ \Gamma^{R,S}_{\alpha\beta\gamma} \}_{\multiset}$.
    \end{Lemma}
    \begin{proof}
        First, if $|\gamma| \leq 2L + S$, let $\gamma = \gamma_1 + \gamma_2 + \gamma_3$ with $|\gamma_1|, |\gamma_2| \leq L$ and $ |\gamma_3| \leq S$. Assume that the expansions of $(v_k^{R,S})^{\gamma_1}$ and $(v_k^{R,S})^{\gamma_2}$ in the orthonormal polynomials are given respectively for $ i = 1, 2 ,3 $ by $(v_k^{R,S})^{\gamma_i} = \sum_{|\alpha| \leq L} p_{\gamma_i\alpha} T^{R,S}_\alpha(v_k^{R,S})$.
        
        In the computation of $\mathcal{E}^{R,S}$, only the triple products $\{ \Gamma^{R,S}_{\alpha\beta\gamma} \}_{\multiset}$ are used. Denote the projection of the multiplication $(v_k^{R,S})^{\gamma_1}(v_k^{R,S})^{\gamma_2}$ into the orthonormal basis by
        \begin{equation}
            (v_k^{R,S})^{\gamma_1}(v_k^{R,S})^{\gamma_2}=\sum_{|\eta|\leq 2L}q_{\eta}T_{\eta}^{R,S}(v_k^{R,S}).
            \end{equation}
        Then, since $ p_{\gamma_{3}\eta} = 0 $ for $ |\eta| > S $
        \begin{equation*}
            \begin{aligned}
                &\mathcal{E}\big[(v_k^{R,S})^\gamma\big] 
                = \mathcal{E}\left[(v_k^{R,S})^{\gamma_1}(v_k^{R,S})^{\gamma_2}(v_k^{R,S})^{\gamma_3}\right]\\
                = &\mathcal{E}\big[\sum_{|\eta|\leq S}q_{\eta}T_{\eta}^{R,S}(v^{R,S}_{k})\sum_{|\eta|\leq S}p_{\gamma_3\eta}T_{\eta}^{R,S}(v^{R,S}_{k})\big]
                =\sum_{\substack{|\eta|\leq S}} q_{\eta}p_{\gamma_{3}\eta} =\sum_{\substack{|\eta|\leq 2L}} q_{\eta}p_{\gamma_{3}\eta} \\=& \E\big[\sum_{|\eta|\leq 2L}q_\eta T_{\eta}^{R,S}(v^{R,S}_{k})\sum_{|\eta|\leq S}p_{\gamma_3\eta}T_{\eta}^{R,S}(v^{R,S}_{k})\big]                 =\E\big[(v_k^{R,S})^{\gamma_1}(v_k^{R,S})^{\gamma_2}(v_k^{R,S})^{\gamma_3}\big] = \E\big[(v_k^{R,S})^\gamma\big].
            \end{aligned}
        \end{equation*}
    
        When $ 2L+S +1 \leq |\gamma| \leq 2L+2S $, let $\gamma = \gamma_1 + \gamma_2 + \gamma_3 + \gamma_{4}$ with $|\gamma_1|, |\gamma_2| \leq L$ and $|\gamma_3|,|\gamma_{4}| \leq S$. Denote 
$(v_k^{R,S})^{\gamma_1}(v_k^{R,S})^{\gamma_3}=\sum_{|\eta|\leq L+S}q^{(1)}_{\eta}T_{\eta}^{R,S}(v_k^{R,S})$ and $(v_k^{R,S})^{\gamma_2}(v_k^{R,S})^{\gamma_4}=\sum_{|\eta|\leq L+S}q^{(2)}_{\eta}T_{\eta}^{R,S}(v_k^{R,S})$. We have  
        \begin{equation*}
            \begin{aligned}
                &|\mathcal{E}\big[(v_k^{R,S})^\gamma\big]| 
                =|\mathcal{E}\left[(v_k^{R,S})^{\gamma_1}(v_k^{R,S})^{\gamma_2}(v_k^{R,S})^{\gamma_3}(v_k^{R,S})^{\gamma_4}\right]| \\=&|\mathcal{E}\big[\sum_{|\eta|\leq L}q^{(1)}_\eta T_{\eta}^{R,S}(v^{R,S}_{k})\sum_{|\eta|\leq L}q^{(2)}_\eta T_{\eta}^{R,S}(v^{R,S}_{k})\big]|
                =\sum_{|\eta|\leq L}q^{(1)}_\eta q^{(2)}_\eta 
                \leq \big[\sum_{|\eta|\leq L}\big(q^{(1)}_\eta\big)^{2}\big]^{\frac12}\big[\sum_{|\eta|\leq L}\big(q^{(2)}_\eta\big)^{2}\big]^{\frac12} 
                \\\leq& \big[\E\big[\big(\sum_{|\eta|\leq L+S}q^{(1)}_\eta T_{\eta}^{R,S}(v^{R,S}_{k})\big)^{2}\big] \big]^{\frac{1}{2}}\big[\E\big[\big(\sum_{|\eta|\leq L+S}q^{(2)}_\eta T_{\eta}^{R,S}(v^{R,S}_{k})\big)^{2}\big] \big]^{\frac{1}{2}}\\
                =& \big[\E\big[(v_k^{R,S})^{2\gamma_1+2\gamma_{3}}\big]\big]^{1/2}\big[\E\big[(v_k^{R,S})^{2\gamma_2+2\gamma_{4}}\big]\big]^{1/2},
            \end{aligned}
        \end{equation*}
        which yields the desired estimate.
        \end{proof}

    A direct corollary of Lemma~\ref{scaled_moment_bound} and Lemma~\ref{mRPC bound} is that, for all multi-indices with $|\gamma|\leq 2L+2S$,
    \begin{equation*}
        \begin{aligned}
            |\Ecal[(v_k^{R,S})^\gamma]| &\le \bigl(\mathbb{E}^{R,S}[|v_k^{R,S}|^{2|\gamma_1|}]\bigr)^{1/2} \bigl(\mathbb{E}[|v_k^{R,S}|^{2|\gamma_2|}]\bigr)^{1/2}\\& \leq C\left( |\gamma_1|!\delta^{-|\gamma_1|}\right)^{1/2}\left(|\gamma_2|!\delta^{-|\gamma_2|}\right)^{1/2}
            \leq C\delta^{-(|\gamma_1| +|\gamma_2|)/2}\sqrt{|\gamma|!}.
        \end{aligned}
    \end{equation*}
    We still write this as $\sup_{0\leq k\leq n}|\Ecal[(v_k^{R,S})^\gamma]|\le C_{\delta} \sqrt{|\gamma|!}$.

    In accordance with \eqref{projection scheme}, let $v_{k+1}^{S} = v_k^{R,S} + h  p_b(v_k^{R,S}) + \sqrt{h}  p_\sigma(v_k^{R,S})  z_k$ denote one step of the exact numerical scheme without truncation operator $\chi_R$. Assume that the corresponding moment set of $v_k^S$ is $\{m_{k,\gamma}^{S}\}_{|\gamma|\leq J}$. We prove the following bound.

    \begin{Lemma}\label{moment turncation error}
        Under Assumptions \ref{asm:drift diffusion terms} and \ref{asm:initial condition}, there exist constants $C,\delta>0$, and $h^*>0$,  depending on $K$ and $T$, such that for all $h\in(0,h^*)$,
        \begin{equation}\label{polynomial truncation error}
           \sup_{0\leq k\leq n-1} d\bigl(\{ m_{k+1,\gamma}^{R,S} \}_{|\gamma|\le J},  \{ m_{k+1,\gamma}^{S} \}_{|\gamma|\le J}\bigr)\leq Ce^{-\delta R^2}. 
        \end{equation}
    \end{Lemma}
    \begin{proof}
        We first show that the distribution of $v_{k+1}^{S}$ remains exponentially integrable. Since $v_k^{R,S}$ is compactly supported on $I_R$, by a similar argument as in \eqref{exp bdd eq1}, we have
        \begin{equation*}
            \E\left[ e^{\delta |v_{k+1}^{S}|^2} \right] \leq \exp\left( \delta(1+ Ch) |v^{R,S}_k|^2 + Ch \right),
        \end{equation*}
        where the constant $C$ depends on $K$ and $T$ and may vary from line to line. Using the exponential integrability of $v_k^{R,S}$ from Lemma~\ref{determinate lemma2} yields $\E\left[ e^{\delta |v_{k+1}^{S}|^2} \right] \leq C$.

        We then estimate the difference between the moments. Note that $v_{k+1}^{R,S}=v_{k+1}^{S}$ on the set $\{ |v_{k+1}^{S}| \le R/3 \}$. By Lemma~\ref{scaled_moment_bound}, we have
        \begin{equation*}
            \begin{aligned}
                |m^{S}_{k+1,\gamma} - m^{R,S}_{k+1,\gamma}| &= | \E[ (v_{k+1}^{S})^\gamma ] - \E[(v_{k+1}^{R,S})^\gamma]| \leq | \E[|v_{k+1}^{S}|^{|\gamma|}\mathds{1}_{\{v^S_{k+1}\in I^c_{R/3}\}}]|\\
                &\leq \left(\mathbb{E}|\mathds{1}_{\{v^{S}_{k+1}\in I^{c}_{R/3}\}}|^{2}\right)^{\frac{1}{2}}\left(\mathbb{E}|v_{k+1}^{S}|^{2|\gamma|}\right)^{\frac{1}{2}}\\
                &\leq \left[\mathbb{E}\left(\frac{\exp(\delta|v^{S}_{k+1}|^{2})}{\exp(\delta R^{2}/9)}\cdot \mathds{1}_{\{v^{S}_{k+1}\in I^{c}_{R/3}\}}\right)\right]^{\frac{1}{2}}\cdot \left(C\delta^{-|\gamma|}|\gamma|!\right)^{\frac{1}{2}} \\
                &\leq C^{\frac{3}{2}}e^{-\delta R^2 / 9} \delta^{-|\gamma|/2}\sqrt{|\gamma|!}.
            \end{aligned}
        \end{equation*}
        Thus, $d\bigl(\{ m_{k+1,\gamma}^{R,S} \}_{|\gamma|\le J},  \{ m_{k+1,\gamma}^{S} \}_{|\gamma|\le J}\bigr)$ is bounded by 
        $Ce^{-\delta R^{2}}\sum_{|\gamma|\leq J} \frac{\delta^{-|\gamma|/2}}{\sqrt{|\gamma|!}}\leq Ce^{-\delta R^{2}}$.
    \end{proof}
    
    Next, we obtain the following lemma for the RPC error $ \Delta_{k}^{(1)} $.
    
    \begin{Lemma} \label{lem:local_error1}
        For all $0 \le k \le n-1$, let $\Delta_k^{(1)}$ be defined in \eqref{defined errors}. Under Assumptions \ref{asm:drift diffusion terms} and \ref{asm:initial condition}, there exist constants $C,\delta>0$ and $h^*>0$, depending on $K$ and $T$, such that for all $h\in(0,h^*)$, 
        \begin{equation}\label{eq:local_error1}
           \Delta_{k+1}^{(1)}\leq \frac{2C_{\delta} C_{S}C_{b,\sigma}R^\eta h}{\sqrt{(2L+1)!}}  + Ce^{-\delta R^2} + Ch^2,
        \end{equation}
        where $C_{S} =\binom{S+d}{S}$.
    \end{Lemma}
    
    \begin{proof}
        For clarity, we present the proof in the one-dimensional case; the extension to higher dimensions follows by analogous arguments using standard multi-index notation. We first expand $(v_{k+1}^{S})^\gamma = (v_k^{R,S} + h p_b(v_k^{R,S}) + \sqrt{h} p_\sigma(v_k^{R,S}) z_k)^\gamma$ via the binomial theorem up to order $\mathcal{O}(h)$. Taking expectation and summing over $\gamma \le 2L+S$ with weights $\gamma!$ gives
        \begin{equation}\label{delta1 eq1}
            \begin{aligned}
                \sum_{\gamma\leq 2L+S}\dfrac{\E[(v_{k+1}^{S})^\gamma]}{\gamma !} &= \sum_{\gamma\leq 2L+S}\dfrac{\E[(v_{k}^{R,S})^\gamma]}{\gamma !} + h \sum_{\substack{\gamma\leq 2L+S\\ \alpha\leq S}} b_\alpha \frac{\E[(v_k^{R,S})^{\gamma+\alpha-1}]}{(\gamma-1)!} +\\
                &\quad+ \dfrac{h}{2}\sum_{\substack{\gamma\leq 2L+S\\ \alpha\leq S}} \sigma_\alpha   \dfrac{\E[(v_k^{R,S})^{\gamma+\alpha-2}]}{(\gamma-2 )!}
                + \mathcal{O}(h^{2}).
            \end{aligned}
        \end{equation}

        The intermediate mRPC moments $\sum_{\gamma\leq 2L+S}({m_{k+1,\gamma}^{\m}}/{\gamma !})$ satisfy the same expansion as \eqref{delta1 eq1} with $\E$ replaced by $\Ecal^{R,S}$. Subtracting these two expressions and using Lemma~\ref{mRPC bound}, which ensures that $\mathcal{E}$ agrees with the exact expectation for moments of order no larger than $2L+S$, we only need to consider terms for which the moment order exceeds $2L+S$, i.e., those with $\gamma+\alpha-1 > 2L+S$ or $\gamma+\alpha-2 > 2L+S$. In this case, by Lemma~\ref{mRPC bound}, the approximate moments are bounded by the exact ones, and by Lemma~\ref{scaled_moment_bound}, we obtain
        \begin{equation*}
            \begin{aligned}
                &\left|\sum_{\gamma\leq 2L+S}\dfrac{\E[(v_{k+1}^{S})^\gamma] - m_{k+1,\gamma}^{\m}}{\gamma!} \right| \\\leq  & h\;\sum_{\mathclap{\substack{\gamma\leq 2L+S,\alpha\leq S\\ \gamma + \alpha - 1 > 2L+S }}} |b_\alpha| \dfrac{|\E-\Ecal^{R,S}|[(v_k^{R,S})^{\gamma+\alpha-1}]}{(\gamma -1)!} + \dfrac{h}{2}\quad\sum_{\mathclap{\substack{\gamma\leq 2L+S,\alpha\leq S\\ \gamma + \alpha - 1 > 2L+S }} }|\sigma_\alpha|   \dfrac{|\E-\Ecal^{R,S}|[(v_k^{R,S})^{\gamma+\alpha-2}]}{(\gamma-2)!}
                + \mathcal{O}(h^{2})\\
                \leq & 2C_{\delta}h \sum_{\mathclap{\substack{\alpha\leq S\\ 2L+2\leq \gamma\leq 2L+S }}} |b_\alpha| \dfrac{\sqrt{(\gamma + \alpha - 1)!}}{(\gamma -1)!} + C_{\delta} h\sum_{\mathclap{\substack{\alpha\leq S\\ 2L+3\leq \gamma\leq 2L+S }}} |\sigma_\alpha|   \dfrac{\sqrt{(\gamma + \alpha - 2)!}}{(\gamma-2)!} + \mathcal{O}(h^{2})\\
                \leq & \dfrac{2C_{\delta} SM_{R,b}h}{\sqrt{\left[(2L+1)!\right]}} + \dfrac{C_{\delta} SM_{R,\sigma}h}{\sqrt{\left[(2L+1)!\right]}}  +\mathcal{O}(h^{2}).
            \end{aligned}
        \end{equation*}
        The last inequality follows from the bounds $\sqrt{\frac{(\gamma+\alpha-1)!}{(\gamma-1)!}}, \sqrt{\frac{(\gamma+\alpha-2)!}{(\gamma-2)!}} \le (2L+2S)^\alpha$ for $\gamma\le 2L+S$, $\alpha\le S$, and the definitions of $M_{R,b}$ and $M_{R,\sigma}$ in \eqref{Mb}. The factor $S$ in the last bound arises from summing over the indices $\gamma$ satisfying $2L+2 \le \gamma \le 2L+S$ (or $2L+3 \le \gamma \le 2L+S$), as there are at most $S$ such values in one dimension. In higher dimensions, this counting factor becomes the number of multi-indices in the corresponding range, which is bounded by $\binom{S+d}{S}$.

        We now consider high-order terms. If we neglect terms of order higher than $\mathcal{O}(h)$ in the one-step mRPC evolution, then the higher-order terms arise only from the expectation in the expansion of $(v_k^{R,S} + h p_b(v_k^{R,S}) + \sqrt{h} p_\sigma(v_k^{R,S}) z_k)^\gamma$. The absolute values of these terms are bounded by:
        \begin{equation*}
            h^2\E\left(|v^{R,S}_k| + |p_b^{R,S}(v_k^{R,S})| + |p^{R,S}_\sigma(v_k^{R,S})|z_k\right)^{\gamma}.
        \end{equation*}
        By the Lipschitz condition of the drift term and the uniform boundedness of the diffusion term, there exists a constant $C>0$ depending only on $K$ such that the expectation is further bounded by
        \begin{equation*}
            C^{\gamma}\E\left(|v^{R,S}_k| + 1 +z_k\right)^{\gamma}\leq (3C)^{\gamma}\max\big\{  \mathbb{E}|v_k^{R,S}|^{\gamma},1,\mathbb{E}(z_k)^\gamma  \big\}.
        \end{equation*}
        Therefore, the higher-order term in $h$ is bounded by

        \begin{equation}\label{high order bound}
    \begin{aligned}
        &h^2\sum_{\gamma \leq 2L+S}\frac{(3C)^{\gamma}\max\{ \mathbb{E}|v_k^{R,S}|^{\gamma},1,\mathbb{E}|z_k|^\gamma \}}{\gamma !} \\
        &\leq h^2 (\sum_{\gamma \leq 2L+S} \frac{(3C)^\gamma}{\gamma!} \mathbb{E}|v_k^{R,S}|^{\gamma} +  \sum_{\gamma \leq 2L+S} \frac{(3C)^\gamma}{\gamma!} + \sum_{\gamma \leq 2L+S} \frac{(3C)^\gamma}{\gamma!} \mathbb{E}|z_k|^\gamma).
    \end{aligned}
\end{equation} 
        The summation in \eqref{high order bound} is uniformly bounded by a constant depending on $K$ and $T$. Specifically, the first sum converges by Lemma~\ref{scaled_moment_bound}, since $\mathbb{E}|v_k^{R,S}|^{\gamma} \leq C_{\delta} \delta^{-\gamma/2} \sqrt{\gamma!}$, which yields a convergent series of the form $\sum x^\gamma / \sqrt{\gamma!}$. The second sum is easily bounded by $h^2 e^{3C}$. The third sum converges similarly to the first one, since the Gaussian moments satisfy $\mathbb{E}|z_k|^\gamma \leq C_z^\gamma \sqrt{\gamma!}$. Combining all these estimates and Assumption $\ref{asm:drift diffusion terms}$ yields \eqref{eq:local_error1}.
    \end{proof}

    The following lemma estimates the RPC error $\Delta_k^{(2)}$.
    
    \begin{Lemma}\label{lem:local_error2}
    For all $0 \le k \le n-1$, let $\Delta_{k}^{(2)}$ be defined in
    \eqref{defined errors}. Under Assumptions ~\ref{asm:drift diffusion terms}, \ref{asm:stability}, and \ref{asm:initial condition}, there exists $h^{*} >0$, such that for all $h\in(0,h^*)$, we have
    \begin{equation}\label{eq:local_error2_new}
        \Delta_{k+1}^{(2)}
        \le \left( 1 + C_{b,\sigma}R^\eta h \right)\Delta_{k}
        + \frac{3C_{S}C_{\delta}C_{b,\sigma}R^\eta h}{\sqrt{(2L+1)!}},
    \end{equation}
    where $C_{S} = \binom{S+d}{S}$.
\end{Lemma}
    
    \begin{proof} 
         We again present the proof in the one-dimensional case. Using a similar expansion as in the proof of Lemma \ref{lem:local_error1}, we have
        \begin{equation*}
            \begin{aligned}
                &\Delta_{k+1}^{(2)} =   \sum_{\gamma\leq 2L+S}\dfrac{|m_{k+1,\gamma}^{\RPC}- m_{k+1,\gamma}^{\m}|}{\gamma!} \leq \sum_{\gamma\leq 2L+S}\dfrac{\left|\Ecal[(v_{k}^{R,S})^\gamma] - \Ecal[(v_{k}^{\RPC})^\gamma]\right|}{\gamma!} \\&+ h \sum_{\mathclap{\substack{\gamma\leq 2L+S\\ \alpha\leq S}}}|b_\alpha| \dfrac{\left|\Ecal[(v_k^{\RPC})^{\gamma+\alpha-1}] - \Ecal[(v_k^{R,S})^{\gamma+\alpha-1}]\right|}{(\gamma-1)!} + \dfrac{h}{2}\sum_{\substack{\gamma\leq 2L+S\\ \alpha\leq S}} |\sigma_\alpha|   \dfrac{\left|\Ecal[(v_k^{\RPC})^{\gamma+\alpha-2}] - \Ecal[(v_k^{R,S})^{\gamma+\alpha-2}]\right|}{(\gamma-2)!}.
            \end{aligned}
        \end{equation*}
        We first estimate the term involving $b_\alpha$. Splitting the inner sum over $\gamma$ into two parts according to whether $\gamma\le 2L+1$ or $\gamma\ge 2L+2$, we obtain 
        \begin{equation*}
            \begin{aligned}
                &\sum_{\substack{\gamma\leq 2L+S\\ \alpha\leq S}} |b_\alpha| \dfrac{\left|\Ecal[(v_k^{\RPC})^{\gamma+\alpha-1}] - \Ecal[(v_k^{R,S})^{\gamma+\alpha-1}]\right|}{(\gamma-1)!}\\
           \leq & \sum_{\substack{\alpha\leq S}} b_\alpha\left( \sum_{\gamma\leq 2L+1} + \qquad\sum_{\mathclap{2L+2\leq\gamma\leq 2L+S}}\qquad\right) \dfrac{(\gamma+\alpha-1)!}{(\gamma - 1)!}\dfrac{\left|\Ecal[(v_k^{\RPC})^{\gamma+\alpha-1}] - \Ecal[(v_k^{R,S})^{\gamma+\alpha-1}]\right|}{(\gamma+\alpha-1)!}.
            \end{aligned}
        \end{equation*}
        For the low-order part ($\gamma\le 2L+1$), we use the bound $(\gamma+\alpha-1)!/(\gamma-1)! \le (2L+S)^\alpha$ (since $\gamma+\alpha-1 \le 2L+S$). This yields  
        \begin{equation*}
            \begin{aligned}
                \sum_{\alpha\le S} |b_\alpha| (2L+S)^\alpha \sum_{\gamma\le 2L+1} \frac{|\mathcal{E}[(v_k^{\RPC})^{\gamma+\alpha-1}] - \mathcal{E}[(v_k^{R,S})^{\gamma+\alpha-1}]|}{(\gamma+\alpha-1)!}\leq M_{R,b}\Delta_{k}.
            \end{aligned}
        \end{equation*}
        For the high-order part, by Lemma~\ref{mRPC bound}, Lemma~\ref{scaled_moment_bound}, and Assumption~\ref{asm:stability}, it is bounded by
\begin{equation*}
    \begin{aligned}
        &\sum_{\substack{2L+2\le \gamma\le 2L+S\\ \alpha\le S}}
        |b_\alpha|(2L+2S)^\alpha
        \frac{2\max\left\{
        |\Ecal^{R,S}[(v_k^{R,S})^{\gamma+\alpha-1}]|,
        |\Ecal[(v_k^{\RPC})^{\gamma+\alpha-1}]|
        \right\}}{(\gamma+\alpha-1)!}\\
        &\qquad\le
        2C_{\delta}\sum_{\substack{2L+2\le \gamma\le 2L+S\\ \alpha\le S}}
        |b_\alpha|(2L+2S)^\alpha \frac{1}{\sqrt{(\gamma+\alpha-1)!}}
        \le
        \frac{2SC_{\delta}M_{R,b}}{\sqrt{(2L+1)!}}.
    \end{aligned}
\end{equation*}
        The factor $S$ again arises from summing over the indices. A completely analogous estimate for the $\sigma^{(2)}$ term yields
        \begin{equation*}\label{eq:sigmaterm}
\begin{aligned}
    \sum_{\substack{\gamma\le 2L+S\\ \alpha\le S}}
    |\sigma_\alpha|
    \frac{\left|\Ecal[(v_k^{\RPC})^{\gamma+\alpha-2}]
    - \Ecal^{R,S}[(v_k^{R,S})^{\gamma+\alpha-2}]\right|}{(\gamma-2)!} 
    \le
    M_{R,\sigma}\Delta_k
    +\frac{2SC_{\delta}M_{R,\sigma}}{\sqrt{(2L+1)!}}.
\end{aligned}
\end{equation*}
        Combining the above estimates and Assumption $\ref{asm:drift diffusion terms}$ yields the desired bound. 
    \end{proof}

    Now, we give the following error estimate of $\Delta_k$.
    
    \begin{proof}[Proof of Theorem $\ref{thm: moment error}$]
    Throughout this proof, $C,c>0$ denote generic constants that may vary from line to line and depend only on $K$, $T$, $d$, $C_{b,\sigma}$, $\lambda_1$, $\lambda_2$, and $\delta$. Combining Lemmas~\ref{lem:local_error1} and~\ref{lem:local_error2}, we obtain
\begin{equation}\label{eq:delta_recursion}
    \Delta_{k+1}
    \leq \Delta_{k+1}^{(1)}+\Delta_{k+1}^{(2)}
    \leq \bigl(1+C_{b,\sigma}R^\eta h\bigr)\,\Delta_k
    + \frac{C\,R^{\eta(d+1)}}{\sqrt{(2L+1)!}}
    + Ce^{-\delta R^2}
    + Ch^2,
\end{equation}
where we have used $C_S=\binom{S+d}{d}\le C R^{\eta d}$ (since $S=[\lambda_2 R^\eta]$)
to absorb $C_\delta C_S$ into the constant. Iterating \eqref{eq:delta_recursion} over $k=0,\dots,n-1$ and using $(1+C_{b,\sigma}R^\eta h)^n\le e^{C_{b,\sigma}R^\eta T}\le e^{cR^\eta}$, we obtain 
\begin{equation}\label{eq:delta_iterated}
    \Delta_k
    \leq \frac{C\,R^{\eta (d+1)}\,e^{cR^\eta}}{\sqrt{(2L+1)!}}
    + C(R^\eta h)^{-1}e^{cR^\eta-\delta R^2}
    + CR^{-\eta}e^{cR^\eta}\,h.
\end{equation}
It remains to estimate the first term. Since $L=[\lambda_1 R^\eta]$, Stirling's formula gives, for sufficiently large $R$, we have $\sqrt{(2L+1)!}
    \ge C\exp (cR^\eta \log R - cR^\eta \bigr)$. 
Hence
\[
    \frac{R^{\eta(d+1)}\,e^{cR^\eta}}{\sqrt{(2L+1)!}}
    \le CR^{\eta d}\,\exp (cR^\eta - cR^\eta \log R\bigr)
    = C\,R^{\eta(d+1)}\,e^{cR^{\eta}-c R^\eta\log R},
\]
which, together with \eqref{eq:delta_iterated}, gives
        \begin{equation*}
        \begin{aligned}
            \sup_{0\le k\le n}d &\left( \{ m_{k,\gamma}^{R,S} \}_{|\gamma|\le J}, \{ m_{k,\gamma}^{\RPC} \}_{|\gamma|\le J} \right)
            \le  C R^{\eta d} e^{\,cR^\eta-cR^\eta\log R}
            + C(R^\eta h)^{-1} e^{\,cR^\eta-\delta R^2}
            + C R^{-\eta}e^{cR^{\eta}}h.
        \end{aligned} 
    \end{equation*}
    Take $R  =  \delta^{-1/2}(-2\log h)^{1/2}$. Renaming all constants and retaining only the leading terms yields the final result. 
\end{proof}
    
    By Theorem $\ref{thm: moment error}$ and the definition of $d(\cdot,\cdot)$,  for any $\ell \leq 2L+S$, the unweighted moment error up to order $\ell$ satisfies
    \begin{equation}\label{low order moment}
        \sum_{|\gamma|\leq \ell} |m_{k,\gamma}^{R,S} - m_{k,\gamma}^{\mathrm{RPC}}| \leq \ell! \Delta_k ,
    \end{equation}
    which is small for a fixed order $\ell$. 
    
    We next prove the following lemma ensuring the existence of a probability measure $\mu_{k,\ell}^{\RPC}$ whose first $\ell$ moments coincide with those computed by the RPC algorithm.
    \begin{Lemma}\label{lemma:distribution matching}
        Let $\ell\in \mathbb{N}$ be a fixed moment order and set $R = \delta^{-1/2}(-2\log h)^{1/2}$. Under Assumptions \ref{asm:initial condition}--\ref{asm:positive measure}, there exists a constant $h_{\tm} > 0$ such that for all $h \in (0, h_{\tm})$, there exists a probability measure $\mu_{k,\ell}^{\RPC}$ supported on $I_R$ satisfying
        \begin{equation}
            m^{\RPC}_{k,\gamma} = \int_{I_R} x^{\gamma} \, d\mu_{k,\ell}^{\RPC}, \qquad \text{for all } |\gamma| \le \ell.
        \end{equation}
    \end{Lemma}
    \begin{proof}
        By Lemma~\ref{truncated determinant}, it suffices to prove that the linear functional $\mathcal T_\ell$ on $\mathcal{P}_\ell(I_R)$ defined by
        \[ \mathcal{T}_{\ell}(x^\gamma) = m^{\RPC}_{k,\gamma},\quad \text{for all } |\gamma|\leq \ell, \]
        is positive for every polynomial $p(x) = \sum_{|\gamma|\leq \ell} c_\gamma x^\gamma$ that is nonnegative on $I_R$ and satisfies $\sum_{|\gamma|\leq \ell} |c_\gamma|^2 = 1$. Denote $m^{R}_{k,\gamma} = \int_{I_R} x^\gamma d\mu_k$. By Assumption~\ref{asm:positive measure}, for $R \ge R^*$, there exists $h_0 > 0$ such that for all $h \in (0, h_0)$, there exists a constant $\lambda(\ell) > 0$ independent of $R$ satisfying
        \begin{equation}\label{eq:positive Euler}
        \sum_{|\gamma| \leq \ell} c_{\gamma}m^{R}_{k,\gamma} \geq \lambda(\ell).
        \end{equation}
        By Lemma~\ref{determinate lemma2}, there exist constants $h^*>0$ and $C_\delta>0$ such that for all $h\le h^*$ and all $|\gamma|\le\ell$,
        \[ |m_{k,\gamma} - m^{R}_{k,\gamma}| = \Bigl|\int_{I_R^c}x^\gamma d\mu_{k}\Bigr| \leq \mathbb{E}\bigl[|v_k|^{|\gamma|}\mathbf{1}_{\{|v_k|\geq R\}}\bigr] \leq R^{|\gamma|-2\ell}\mathbb{E}[|v_k|^{2\ell}] \leq C_\delta R^{-\ell}\ell!.\]
        Consider the telescoping difference series:
        \begin{equation}
            \begin{aligned}
            &(v_k^{(1)})^{\gamma_1} (v_k^{(2)})^{\gamma_2} \cdots (v_k^{(d)})^{\gamma_d} - (v_k^{R,S(1)})^{\gamma_1} (v_k^{(2)})^{\gamma_2} \cdots (v_k^{(d)})^{\gamma_d}\\
            &+ (v_k^{R,S(1)})^{\gamma_1} (v_k^{(2)})^{\gamma_2} \cdots (v_k^{(d)})^{\gamma_d} - \cdots - (v_k^{R,S(1)})^{\gamma_1} (v_k^{R,S(2)})^{\gamma_2} \cdots (v_k^{R,S(d)})^{\gamma_d}.
            \end{aligned}
        \end{equation}
        The difference between consecutive terms can be bounded using the Cauchy-Schwarz inequality. For the $j$-th difference, we have:
        \begin{equation}\label{eq:moment error of RS}
            \begin{aligned}
                &\quad|\mathbb{E}(v^{R,S(1)}_{k})^{\gamma_{1}}\cdots (v_{k}^{R,S(j)})^{\gamma_{j}}\cdots (v_{k}^{(d)})^{\gamma^{d}} - \mathbb{E}(v^{R,S(1)}_{k})^{\gamma_{1}}\cdots (v_{k}^{(j)})^{\gamma_{j}}\cdots (v_{k}^{(d)})^{\gamma^{d}}| \\
                &\leq \sum_{s\leq \gamma_j-1} |\mathbb{E}\left[ (v^{R,S(1)}_{k})^{\gamma_{1}}\cdots (v^{R,S(j)}_{k})^{\gamma_{j}-s-1}(v^{(j)}_{k})^{s}\cdots (v^{(d)}_{k})^{\gamma_{d}}\left(v^{R,S(j)}_{k}- v^{(j)}_{k}\right)\right]|\\
                &\leq \sum_{s\leq \gamma_j-1}\left[ \mathbb{E}\left((v^{R(1)}_{k})^{\gamma_{1}}\cdots (v^{R,S(j)}_{k})^{\gamma_{j}-s-1}(v^{(j)}_{k})^{s}\cdots (v^{(d)}_{k})^{\gamma_{d}}\right)^{2} \right]^{\frac{1}{2}}\left[ \mathbb{E}|v^{R,S(j)}_{k}- v^{(j)}_{k}| ^{2} \right]^{\frac{1}{2}}.
            \end{aligned}
        \end{equation} 
        The first expectation on the last line of equation \eqref{eq:moment error of RS} is bounded by
        \begin{equation*}
            \begin{aligned}
                \mathbb{E}(|v_{k}^{R,S}|^{|\gamma|}\cdot|v_{k}|^{|\gamma|})\leq \left[\mathbb{E}|v_{k}^{R,S}|^{2|\gamma|}\mathbb{E}|v_{k}|^{2|\gamma|}\right]^{1/2}
                &\leq C_\delta \ell!.
            \end{aligned}
        \end{equation*}
        For the second expectation, Lemma~\ref{convergent thm} gives $|m_{k,\gamma} - m^{R,S}_{k,\gamma}| \leq CC_{\delta} (\ell+1)!(e^{-\delta  R^2}  +R^rS^{-r} )$.
        Therefore,
        \[ |m^{R}_{k,\gamma} - m^{R,S}_{k,\gamma}|\leq |m^{R}_{k,\gamma} - m_{k,\gamma}| + |m_{k,\gamma} - m^{R,S}_{k,\gamma}|\leq C_\delta  R^{-\ell}\ell! +  CC_{\delta} (\ell+1)!\left(e^{-\delta  R^2}  + R^rS^{-r} \right).\]
        Combining this with \eqref{low order moment} and the relation between $S$ and $R$, we obtain, for $C>0$ depending on $K$, $T$, $\lambda_1$:
        \begin{equation}\label{eq:moment bdd}
            |m^{R}_{k,\gamma} - m^{\RPC}_{k,\gamma}|\leq C_\delta  R^{-\ell}\ell! +  CC_{\delta} (\ell+1)!\left(e^{-\delta  R^2}  + R^{r(1-\eta)} \right) + \ell! \Delta_k .
        \end{equation}
        For fixed $\ell$, we choose $R(h) = \delta^{-1/2}(-2\log h )^{1/2}$ so that $\lim_{h\to 0} R(h) = \infty$. By Theorem~\ref{thm: moment error}, there exists a  constant $N>0$ such that the error $\Delta_k$ satisfies the estimate 
        \[
        \Delta_k \lesssim (-\log h)^{-N},
        \]
        which tends to $0$ as $h\to 0$. Consequently, there exists $\hat h>0$ such that for all $h\in(0,\hat h)$, the right-hand side of \eqref{eq:moment bdd} is smaller than $\frac{\lambda(\ell)}{2C_{\ell}}$, where $C_{\ell} = \binom{\ell+d}{d}$. Then, using \eqref{eq:positive Euler} and \eqref{eq:moment bdd}, we obtain
        \[ \sum_{|\gamma| \leq \ell} c_{\gamma}m^{\RPC}_{k,\gamma}= \sum_{|\gamma| \leq \ell} c_{\gamma}m^{R}_{k,\gamma} + \sum_{|\gamma| \leq \ell} c_{\gamma}\left( m_{k,\gamma}^{\RPC} - m^{R}_{k,\gamma}\right)\geq\lambda(\ell) -C_{\ell}\max_{|\gamma|\leq \ell}|m_{k,\gamma}^{\RPC} - m^{R}_{k,\gamma}|\geq \lambda(\ell)/2, \]
        which is positive. Set $h_{\tm} = \min\{h_0, h^*, \hat h\}$. Then, for all $h \le h_{\tm}$, there exists a probability measure $\mu_{k,\ell}^{\RPC}$ supported on $I_R$ whose moments $\{m^{\RPC}_{k,\gamma}\}_{|\gamma|\leq \ell}$ coincide with those computed by the RPC algorithm.
    \end{proof}

    Let $\mu_k^{R,S}$ denote the distribution of $v_k^{R,S}$ and $\mu_k$ denote the distribution of $v_k$. Then, using the moment error bound \eqref{low order moment}, we establish Theorem \ref{W1distance} on the $W_1$ distance between $\mu_k$ and $\mu_{k,\ell}^{RPC}$.

\begin{proof}[Proof of Theorem \ref{W1distance}]
Throughout this proof, $C,c>0$ denote generic constants that may change from line to line and depend only on $d$, $K$, $T$, $\lambda_1$, $\lambda_2$, $\delta$, $\eta$, and $s$.

Set $R^{'} = \ell^s \leq R$ with $s\ge\tfrac{1}{2}$, where $\ell$ is a fixed moment-matching order. Let $I_{R^{'}} = [-R^*,R^*]^d$ and define $\bar\mu_k^{R,S}(\cdot) = \mu_k^{R,S}(\cdot \mid I_{R^{'}})$ and $\bar\mu_{k,\ell}^{\RPC}(\cdot) = \mu_{k,\ell}^{\RPC}(\cdot \mid I_{R^{'}})$. By the triangle inequality,
\begin{equation}\label{eq:W1triangle}
    W_1(\mu_k^{R,S}, \mu_{k,\ell}^{\RPC}) 
    \le W_1(\mu_k^{R,S}, \bar\mu_k^{R,S}) 
    + W_1(\bar\mu_k^{R,S}, \bar\mu_{k,\ell}^{\RPC}) 
    + W_1(\mu_{k,\ell}^{\RPC}, \bar\mu_{k,\ell}^{\RPC}).
\end{equation}
We bound the first and third terms using the same procedure. We present the argument for $\mu_k^{R,S}$. By the Kantorovich-Rubinstein duality, for any $1$-Lipschitz function $f$ with $f(0)$ \- $=0$,
\begin{equation}
\label{eqdual}
    \begin{aligned}
        \int f\, d\mu_k^{R,S} - \int f\, d\bar\mu_k^{R,S}
        &= \int_{I_{R^*}^c} f\, d\mu_k^{R,S} - \mu_k^{R,S}(I_{R^*}^c)\int_{I_{R^*}} f\, d\bar\mu_k^{R,S}.
    \end{aligned}
\end{equation}

Since $|f(x)|\le |x|$ and $|f(x)|\le R^{'}\sqrt d$ on $I_{R^{'}}$, we obtain
\begin{equation}
\label{eqW_11}
    \begin{aligned}
        W_1(\mu_k^{R,S}, \bar\mu_k^{R,S})
        \le \mathbb{E}\!\left[|v_k^{R,S}|\,\mathds{1}_{I_{R^{'}}^c}(v_k^{R,S})\right] + R^{'}\sqrt{d}\,\mu_k^{R,S}(I_{R^{'}}^c).
    \end{aligned}
\end{equation}
On the set $I_{R^{'}}^c$, since $\ell\ge 2$, we have $|x| \le \frac{|x|^\ell}{(R^{'})^{\ell-1}}$ and $R^{'}
\le
\frac{|x|^\ell}{(R^{'})^{\ell-1}}$. Hence by Lemma \ref{scaled_moment_bound} we get
\begin{equation*}\label{eq:tail_RS}
W_1(\mu_k^{R,S},\bar\mu_k^{R,S})
\le
\frac{2\,\E|v_k^{R,S}|^\ell}{(R^{'})^{\ell-1}}
\le
\frac{C \sqrt{\ell!}}{(R^{'})^{\ell-1}}.
\end{equation*} 
The same bound holds for $W_1(\mu_{k,\ell}^{\RPC}, \bar\mu_{k,\ell}^{\RPC})$ by Assumption~\ref{asm:stability}.

For the conditioning error on moments,  similar to \eqref{eqdual} \eqref{eqW_11}, we have for any $|\gamma| \le \ell$,
\begin{equation*}\label{eq:cond_moment}
    |m_\gamma^{\bar\mu_k^{R,S}} - m_{k,\gamma}^{R,S}| \le \E\!\left[|v_k^{R,S}|^{|\gamma|}\mathds{1}_{I_{R^{'}}^c}(v_k^{R,S})\right]
+
(R^{'})^{|\gamma|}\mu_k^{R,S}(I_{R^{'}}^c)
    \le \frac{CC_\delta\sqrt{\ell!}}{(R^{'})^{\ell-|\gamma|}},
\end{equation*}
and the same holds for $|m_\gamma^{\bar\mu_{k,\ell}^{\RPC}} - m_{k,\gamma}^{\RPC}|$ by Assumption~\ref{asm:stability}. Combining via the triangle inequality,
\begin{equation}\label{eq:moment_cond_tri}
|m_\gamma^{\bar\mu_k^{R,S}}-m_\gamma^{\bar\mu_{k,\ell}^{\RPC}}|
\le
|m_{k,\gamma}^{R,S}-m_{k,\gamma}^{\RPC}|
+
\frac{C \sqrt{\ell!}}{(R^{'})^{\ell-|\gamma|}}.
\end{equation} 
Since $\bar\mu_k^{R,S}$ and $\bar\mu_{k,\ell}^{\RPC}$ are supported on $I_{R^{'}}$, Lemma~\ref{moment matching} yields
\begin{equation*}\label{eq:W1interior}
    W_1(\bar\mu_k^{R,S}, \bar\mu_{k,\ell}^{\RPC}) 
    \le \frac{C_d R^{'}}{\ell} + C_d\, 3^\ell\, R^{'}
    \sqrt{\sum_{|\gamma| \le \ell} \left(\frac{m_\gamma^{\bar\mu_k^{R,S}} - m_\gamma^{\bar\mu_{k,\ell}^{\RPC}}}{(R^{'})^{|\gamma|}}\right)^{\!2}},
\end{equation*}
where $C_d$ depends only on $d$. Substituting~\eqref{eq:moment_cond_tri}, using $(R^{'})^{|\gamma|} \ge R^{'}$ for $|\gamma| \ge 1$, and applying $\sqrt{\sum a_i^2} \le \sum |a_i|$, we obtain
\begin{equation*}
\begin{aligned}
W_1(\bar\mu_k^{R,S},\bar\mu_{k,\ell}^{\RPC})
&\le
\frac{C_dR^{'}}{\ell}
+
C_d\,3^\ell\,R^{'}
\sum_{|\gamma|\le \ell}\frac{|m_{k,\gamma}^{R,S}-m_{k,\gamma}^{\RPC}|}{(R^{'})^{|\gamma|}}
+
C\,3^\ell\,R^{'}\,\sqrt{C_\ell}\,\frac{\sqrt{\ell!}}{(R^{'})^\ell},
\end{aligned}
\end{equation*} 
where $C_\ell = \binom{\ell+d}{\ell}$. Combining all contributions with~\eqref{eq:W1triangle}, the total bound is 
\begin{equation}\label{eq:W1raw}
W_1(\mu_k^{R,S},\mu_{k,\ell}^{\RPC})
\le
\frac{CR^{'}}{\ell}
+
C\,3^\ell\sqrt{C_\ell}\,\frac{\sqrt{\ell!}}{(R^{'})^{\ell-1}}
+
C\,3^\ell\,\ell!\,\Delta_k.
\end{equation}

Note $\binom{\ell+d}{d}\le C \ell^{d}$ for fixed~$d$. By Stirling's formula, $\sqrt{\ell!}\le C \ell^{1/4}(\ell/e)^{\ell/2}$ and $\ell!\le C \sqrt{\ell} (\ell/e)^{\ell}$.
Substituting into~\eqref{eq:W1raw}, the second term satisfies
\[
  3^{\ell}\sqrt{C_\ell} \frac{\sqrt{\ell!}}{\ell^{s(\ell-1)}}
  \;\le\;
  C \ell^{\frac{d}{2}+\frac{1}{4}+s} 
  \exp\big(\ell [c
    +(\tfrac{1}{2}-s)\log\ell
    ]\big),
\]
and the third term satisfies $
  C 3^{\ell} \ell!\;\Delta_k
  \;\le\;
  C \sqrt{\ell}\;
  \exp (\ell\log\ell )\;\Delta_k.$
Taking the leading term of $\Delta_k$, rearranging the constant, and substituting all the above estimates into \eqref{eq:W1raw}, we obtain that 
\begin{equation*}
                \begin{aligned}
                     W_1 \bigl(\mu_k^{R,S}, \mu_{k,\ell}^{\RPC}\bigr)
        &\le 
        C \ell^{ s-1}
        + 
        C \ell^{ \frac{d}{2}+\frac{1}{4}+s} 
        e^{c \ell + \bigl(\frac{1}{2}-s\bigr) \ell\log\ell}+ 
        C \ell^{\frac12} e^{\ell \log \ell} (-\log h)^{-m}.
    \end{aligned}
 \end{equation*}
Since the joint law of $(v_k,v_k^{R,S})$ provides a coupling between $\mu_k$ and
$\mu_k^{R,S}$,
\begin{equation}
\label{W_1vk}
\begin{aligned}
W_1(\mu_k,\mu_k^{R,S})
\le \mathbb E|v_k-v_k^{R,S}| \leq \bigl(\mathbb E|v_k-v_k^{R,S}|^2\bigr)^{1/2}.
\end{aligned}
\end{equation}
In particular, for each fixed $\ell\in\mathbb N$, we can choose $h_\ell=h_\ell(d,K,T,\delta,\ell)\in(0,h_{\tm}]$ such that, for all $0<h\le h_\ell$, by retaining only the leading term, we have $W_1 \bigl(\mu_k^{R,S}, \mu_{k,\ell}^{\RPC}\bigr)\le C \ell^{-s} + C(-\log h)^{-\frac{r(1-\eta)}{2}}$.
\end{proof}

\section{Numerical Results}
\label{sec:numericalres}

In this section, we validate the RPC framework through numerical experiments focusing on several long-time stochastic dynamics, including coupled nonlinear SDEs, systems with random parameters and random initial conditions, intermittency, and non-Gaussian invariant measures. We demonstrate the superior computational efficiency of RPC compared to both MC and gPC methods, while maintaining high accuracy. Moreover, our approach handles systems where traditional model reduction methods (e.g., gPC) often fail to capture the long-time evolution, and where MC fails to converge even with extremely large sample sizes. We also provide comparisons with DgPC \cite{Bal2}, which enables long-time simulations, including convergence to invariant measures of SDEs via appropriate restarts of gPC-based simulations. Using a representative example, we show that the frequent restarts and projections in our RPC method are computationally more advantageous.

All numerical results are benchmarked against exact analytical solutions or, where these are unavailable, high-precision MC simulations. All MC and gPC methods considered in this work utilize parallel computing techniques to significantly accelerate computations and reduce computational time.

\paragraph{Example 4.1.} First, we consider a 2D additive noise test model proposed in \cite{intermtestmodel1}
\begin{equation*}
    \label{intereq1}
    \begin{aligned}
        du(t) &= -(b_u + a_u v(t))u(t)   dt + \sigma_u   dW_u(t), \\
        dv(t) &= -(b_v + a_v u(t)) v(t)   dt + \sigma_v   dW_v(t),
    \end{aligned}
\end{equation*}
where $a_u, a_v \geq 0$, $b_u, b_v > 0$ are damping parameters, drift parameters $ \sigma_u, \sigma_v > 0$ are diffusion constants, and  $W_u, W_v$ are two independent Wiener processes. This model serves as a framework for studying filtering in turbulent signals with intermittent instabilities. Despite its relatively simple structure, the model exhibits rich dynamical behavior, and has therefore been adopted as a test model in \cite{intermitemodel2,Bal2}.

When $a_v=0$, this model becomes conditionally linear, and the mean and variance of $u(t)$ can be computed analytically. Although $v(t)$ follows Gaussian dynamics independent of $u(t)$, the quadratic nonlinearity in the $u$-equation transforms $u(t)$ into highly non-Gaussian, intermittent dynamics. In fact, this two-dimensional coupled model can exhibit rich dynamical behavior for appropriate parameter choices. Interested readers may refer to \cite{intermitemodel2} for a more comprehensive exploration.

We set $a_u = 1$, $a_v = 0$, $b_u = 1.2$, $b_v = 0.5$, $\sigma_u = 0.5$, $\sigma_v = 0.5$ with the independent random initial condition $u_0 = N(1, \sigma_u^2/8b_u)$, $ v_0 = N(0, \sigma_v^2/8b_v), T=12$ and $\Delta t = 0.012$. In this parameter regime, the non-Gaussian dynamics of $u(t)$ are characterized by intermittent bursts of large-amplitude, transient instabilities followed by quiescent phases. This behavior is characteristic of turbulent modes in the dissipative range.

We first utilize the gPC method to compute the evolution of $u,v$. Given the gPC expansion 
\begin{equation*}
u(t,\xi) \approx u_{K,N}(t,\xi_1,\cdots,\xi_K) := \sum_{|\alpha|\leq N} u_{\alpha}(t)\prod_{i=1}^{K} H_{\alpha_i}(\xi_i),
\end{equation*}
we present numerical results for gPC of $u(t)$ and $v(t)$ with $N=2,3$ and for $K=4,10,20$.

      \begin{figure}
        \centering
        \includegraphics[width=15cm]{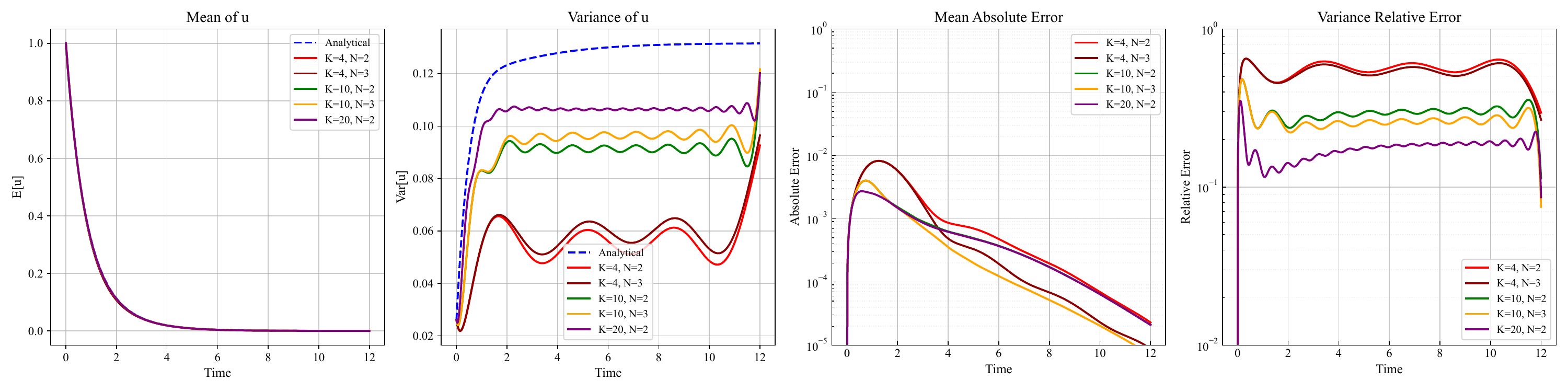}
        \caption{The mean and variance solved by gPC compared with the analytical solution}
        \label{gpc1}
      \end{figure}
As shown in Figure \ref{gpc1}, standard gPC achieves reasonable accuracy only for short-term mean predictions. Its variance, however, exhibits persistent non-physical oscillations that fail to capture the smooth analytical dynamics. These artifacts stem from the global polynomial basis's inability to resolve the system's sharp intermittent bursts and strong nonlinear coupling ($a_u=1, b_u=1.2$). Any apparent variance convergence at the final time is misleading, as the overall evolution remains dominated by large oscillatory errors. Furthermore, increasing the basis size (larger $K,N$) provides only marginal improvements and fundamentally fails to eliminate these oscillations.

Then we employ the DgPC method \cite{Bal2} with $L=2$ and $50$ restarts to solve this problem as a reference for comparison with RPC using $L=2$.
       \begin{figure}
        \centering
        \includegraphics[width=15cm]{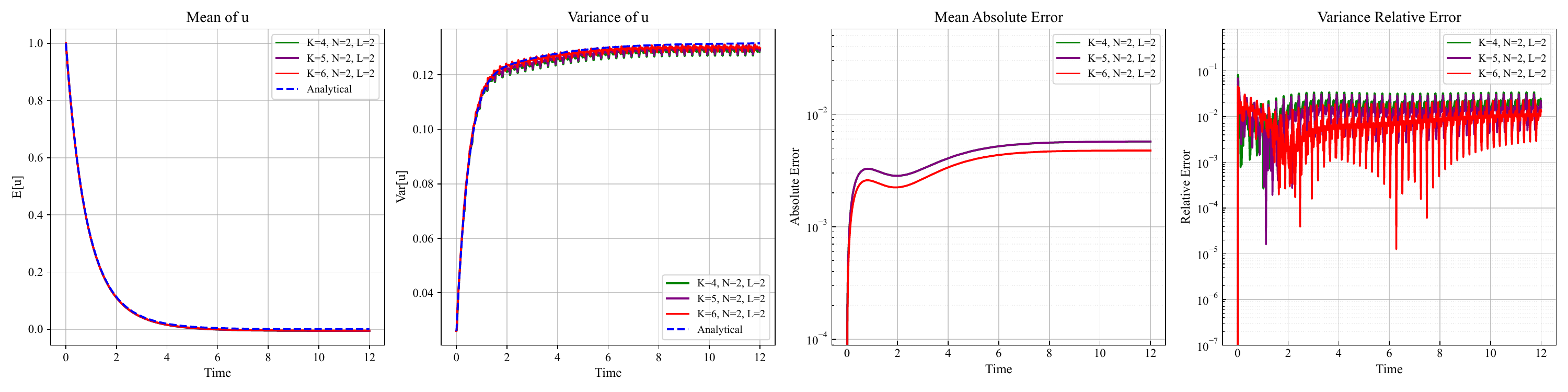}
        \caption{The mean and variance obtained by DgPC compared with the analytical solution}
        \label{dgpc1}
      \end{figure}
We can see from Figure $\ref{dgpc1}$ that DgPC performs significantly better than standard gPC, although small oscillations persist due to the restart procedure. The DgPC method achieves a relative error of approximately $10^{-2}$ at the final time. Further improvements in accuracy can be achieved by increasing $L$ and $K$, but to ensure a fair comparison of computational costs at a fixed accuracy, we do not pursue this here.

In contrast, Figure \ref{couple1} presents the evolution of $u(t)$ computed by our RPC method with $L=2,3$ and $S=2$. RPC achieves excellent agreement with the analytical mean and variance over the entire time interval. It accurately captures the convergence to the invariant measure and eliminates the non-physical oscillations observed in both gPC and DgPC.

       \begin{figure}
        \centering
        \includegraphics[width=15cm]{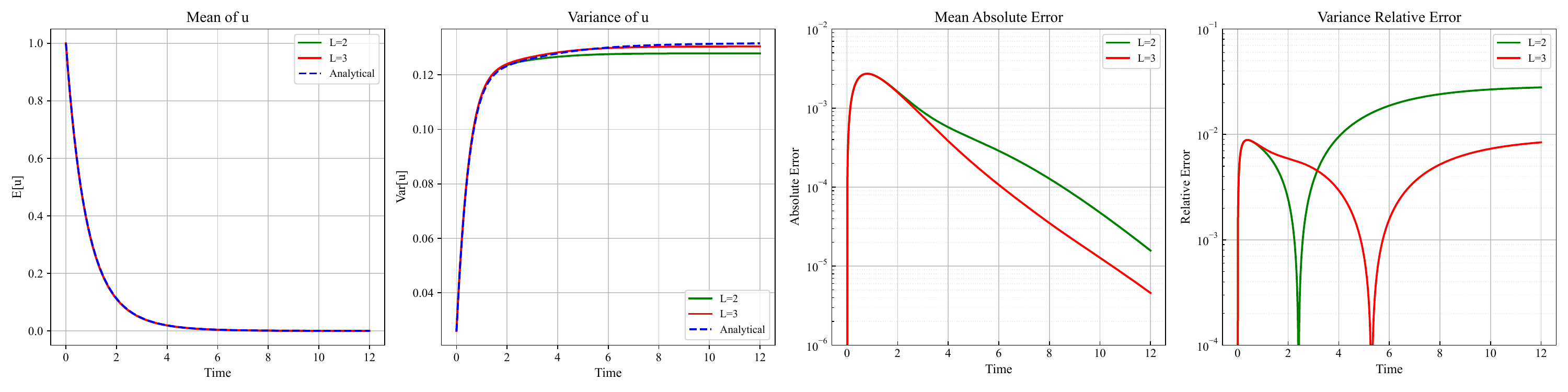}
        \caption{The mean and variance solved by RPC compared with the analytical solution}
        \label{couple1}
      \end{figure}

Moreover, our superior accuracy comes with substantial computational advantages. Even the best-performing gPC configuration ($K=20,N=2$) cannot match our method's effectiveness under the basic setting $L=2$, while requiring $1193.12$s compared to our $13.15$s-more than $100$ times slower. Similarly, when both methods use the same polynomial order $L=2$, our accuracy matches that of DgPC with $K=6,N=2$, yet DgPC requires $280.13$s, whereas our computation is significantly faster.

Finally, we present the evolution of the non-Gaussian statistics and memory effects of the component $u$ computed by RPC in Figure \ref{fig:nongaussian}. As shown in the first panel, the covariance $\text{Cov}(u(t), u_0)$ decays exponentially over time, illustrating the system's fading memory of its initial state, which is consistent with convergence to an invariant measure independent of the initial condition. Even though our algorithm employs a restarting procedure, it still effectively preserves the underlying correlation structure, thereby accurately capturing this long-term memory effect. Concurrently, the excess kurtosis of $u$ increases significantly from zero, stabilizing at values characteristic of a heavy-tailed distribution. The eigenvalues of the numerical Hankel matrix remain uniformly bounded with $\lambda_{\min} > 10^{-6}$ and $\lambda_{\max} < 10^2$ throughout the simulation, demonstrating the numerical stability in constructing the orthogonal polynomial basis.

\begin{figure}
    \centering
    \includegraphics[width=15cm]{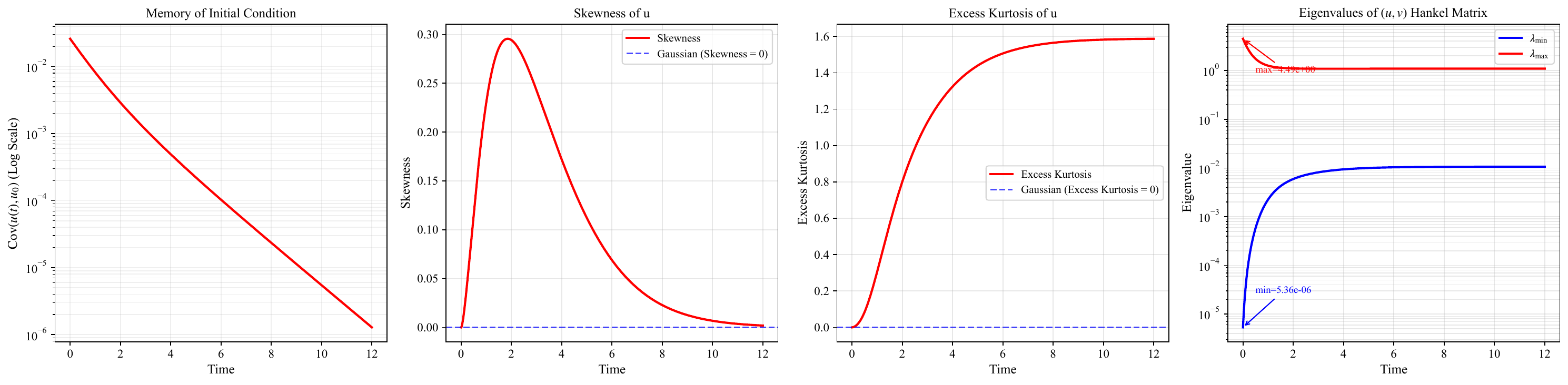}
    \caption{Statistical analysis of component $u$ computed by RPC $(L=3)$. The panels display (from left to right) the covariance $\text{Cov}(u(t), u_0)$ on a semi-logarithmic scale, the time evolution of skewness and excess kurtosis, and the eigenvalue ($\lambda_{\min}$, $\lambda_{\max}$) of the Hankel matrix.}
    \label{fig:nongaussian}
\end{figure}

    Note that our RPC framework can also handle SDEs with random parameters by augmenting the state space. For a random parameter $\xi$ in SDE,  we simply add the trivial equation $d\xi = 0$ to the system. This simple augmentation allows us to treat random coefficients within the same unified framework, without any algorithmic modifications. We validate this capability using  model $(\ref{intereq1})$ with a random coefficient $a_u \sim U(0.1,1.1)$, which is independent of the initial conditions $(u_0, v_0)$  and of $W_u(t)$, $W_v(t)$ . We set $a_v=0.05$  with all other parameters the same as before. This configuration yields a fully coupled nonlinear stochastic system with a random coefficient.
    \begin{figure}
    \centering
    \includegraphics[width=15cm]{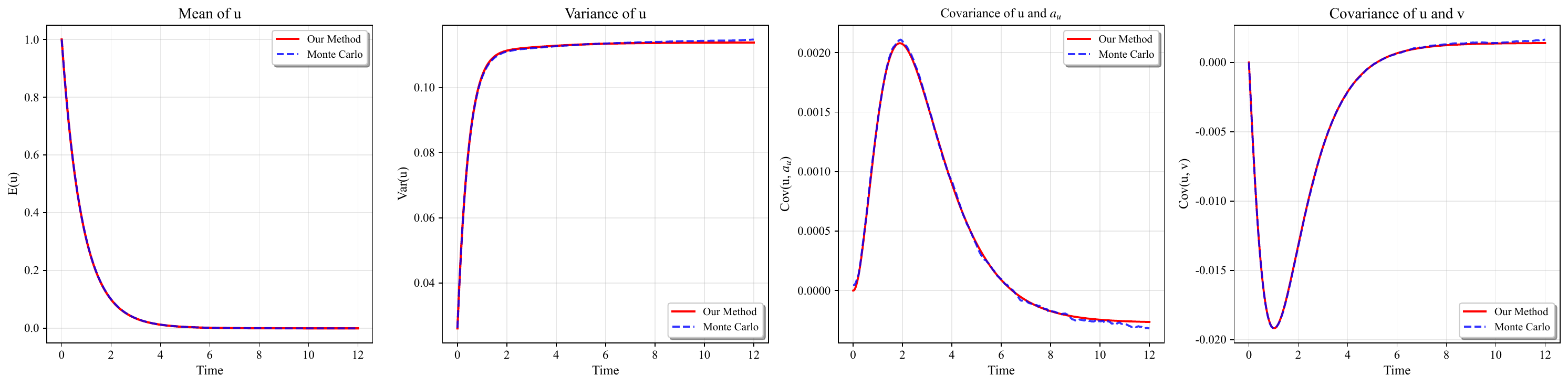}
    \caption{First and second order statistics information of  $u$ and random coefficient $a_u$ computed by RPC.}
    \label{fig:randomcoe}
\end{figure}
    
    Figure $\ref{fig:randomcoe}$ compares our method with $L=3$ against the MC reference solution. RPC accurately captures not only the long-time dynamics of $u$, but also the time-varying correlations between $u$ and $a_u$, which demonstrates that RPC correctly propagates parametric uncertainty through the nonlinear dynamics.

\paragraph{Example 4.2.} For this scenario, we consider a coupled stochastic differential system with multiplicative nonlinear noise \cite{LIU2023137}
\begin{equation*}
    \begin{aligned}
\mathrm{d}u(t) &= \left[10-3u(t)-v(t)\right]\mathrm{d}t + \left[0.5+ 0.1 v(t)\right]\mathrm{d}W_u(t), \\
\mathrm{d}v(t) &= \left[5- u(t) -3 v(t) - v(t)^3\right]\mathrm{d}t +  \left[0.3+ 0.1 u(t) + 0.1v(t)^2\right]\mathrm{d}W_v(t).
\end{aligned}
\end{equation*}
where $W_u(t)$ and $W_v(t)$ are independent standard Brownian motions, initialized with $u_0 \sim {N}(0.3,0.2^2)$ and $v_0 \sim N(0.5,1.2^2)$.  The core challenge of this model lies in the strong competition between the quadratic noise amplification ($0.1v^2$) and the highly dissipative cubic drift ($-v^3$). It produces occasional large-amplitude excursions that are rapidly suppressed, ultimately driving the system toward a non-Gaussian invariant measure with heavy tails. Following \cite{LI2022110210, LIU2023137}, a terminal time of $T=5$ is sufficient to capture this stationary distribution.

We evaluate our RPC method with $L = 3$ and $S = 3$. The reference solution is computed using the Milstein scheme with $10^7$ MC samples and sufficiently small time steps. We then compare RPC against this reference by computing the first four statistical moments, with relative errors at the final time $T$ displayed above each subplot.

\begin{figure}
    \centering
    \includegraphics[width=15cm]{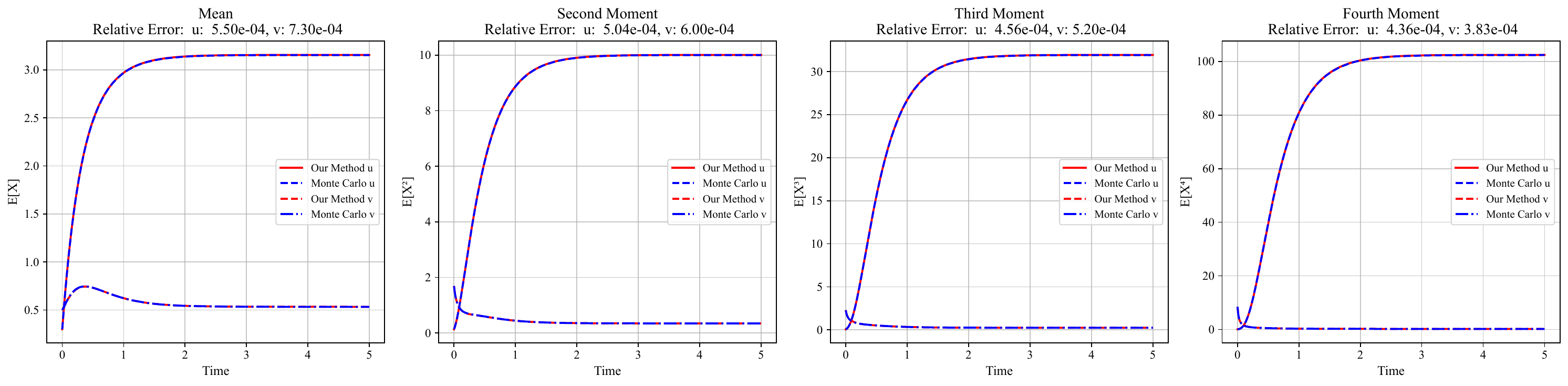}
    \label{EX53}
    \caption{Comparison of the first four order moments with the Monte Carlo solution }
\end{figure}

As shown in Figure \ref{EX53}, despite the complex coupling between state-dependent noise and polynomial nonlinearities, RPC accurately matches the MC reference trajectories and achieves relative errors on the order of $10^{-4}$ for all four moments. The moment evolution of $v$ clearly exhibits the characteristic dynamics: initial growth followed by rapid relaxation to equilibrium values, consistent with the competition between quadratic noise amplification and cubic dissipation that leads the system to the non-Gaussian invariant measure. 

This accuracy is especially notable for multiplicative noise systems, where traditional Euler-based MC methods require extremely large sample sizes to obtain convergence of higher-order moments, due to state-dependent diffusion. Our approach, even though it uses an Euler-type framework, achieves high precision with only parameter $L = 3$, while faithfully capturing the system's non-Gaussian invariant measure, demonstrating significant computational advantages over traditional sampling methods.

\paragraph{Example 4.3.} Next, we consider a triad system with a Gaussian invariant measure proposed in \cite{Sapsis2013BlendedRS}
\begin{equation*}
\begin{aligned}
du(t) &= (-\gamma_1 u(t) + \lambda_{12} w(t) + \lambda_{13} v(t) + \beta_1 v(t) w(t)) dt + \sigma_1 dW_u(t)  \\
dw(t) &= (-\gamma_2 w(t) - \lambda_{12} u(t) + \lambda_{23} v(t) + \beta_2 u(t)v(t)) dt + \sigma_2 dW_w(t)  \\
dv(t) &= (-\gamma_3 v(t) - \lambda_{13} u(t) - \lambda_{23} w(t) + \beta_3 w(t) u(t)) dt + \sigma_3 dW_v(t) 
\end{aligned}
\end{equation*}
with $\beta_1 + \beta_2 + \beta_3 =0$. This is a three-dimensional system whose quadratic part is both divergence-free and energy-preserving. The linear portion consists of a negative definite dissipative operator and a skew-symmetric operator. This simplified model captures key elements of fluid dynamics: the nonlinear terms represent advection, the dissipative terms represent viscous dissipation, the skew-symmetric component represents Coriolis effects, and the stochastic forcing approximates interactions with other modes.

To evaluate our RPC method, we investigate two distinct dynamical regimes and benchmark against a weak Runge-Kutta MC solution with $10^6$ realizations. In practice, convergence is difficult to achieve due to the strong nonlinearity and parameter biases, especially for high-order moments. To alleviate this issue, we employ antithetic variates as a variance reduction technique to suppress strong sampling noise.

\textbf{Case 1:} When we take $\frac{\sigma_1^2}{2\gamma_1} = \frac{\sigma_2^2}{2\gamma_2} = \frac{\sigma_3^2}{2\gamma_3} = E$, the statistical dynamics of the system converges asymptotically in time to an invariant measure with equal energy distribution among the three degrees of freedom. In this case, the invariant measure is Gaussian and given by $p \propto \exp(-\|\mathbf{u}\|^2/2E)$. 

This system serves as a benchmark to demonstrate the limitations of traditional dynamical low-rank methods, which, due to energy equipartition in the invariant measure, can capture at most $1/3$ (single-mode) or $2/3$ (two-mode) of the steady-state covariance $\cite{Sapsis2013BlendedRS}$.

We first set $E=1$ to ensure the standard Gaussian distribution. The parameters are chosen as $\gamma_1 = 0.4,\gamma_2=2, \gamma_3 =2, \lambda_{12}=0.03, \lambda_{13}= 0.06, \lambda_{23}=0.09, \beta_1=2 ,\beta_2=\beta_3=-1, \sigma_1 = \sqrt{0.8}, \sigma_2=\sigma_3=2$, with initial conditions $u_0 \sim N(-1, 0.25) , w_0 \sim N(0.5,2) , v_0 \sim N(-0.5, 0.0225)$. The initial variables are mutually independent and independent of the stochastic forcing. We compare the performance of our method with the reference solution for the first four central moments.

\begin{figure}
        \centering
        \includegraphics[width=15cm]{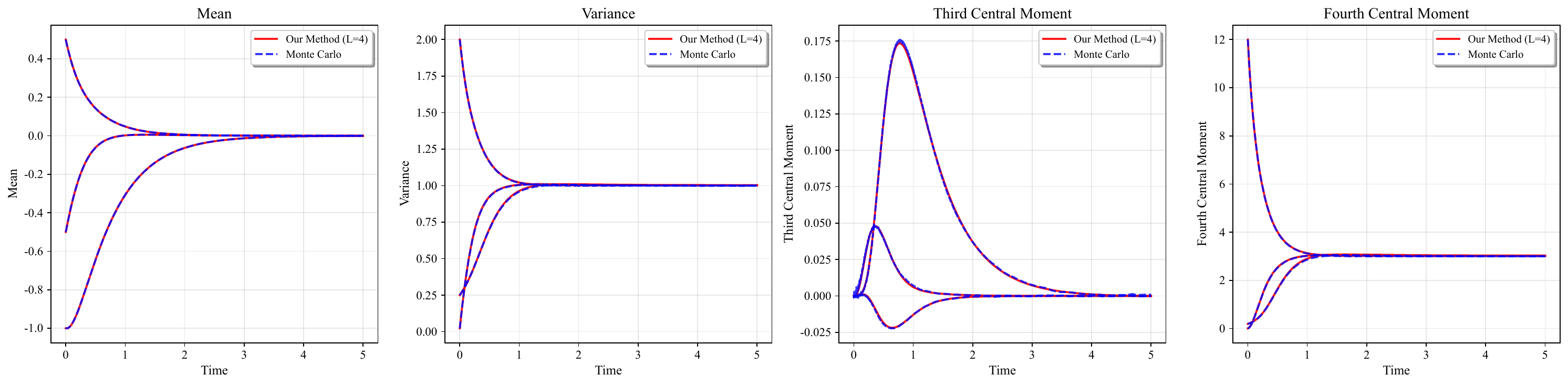}
        \caption{The first fourth order moment compared with the Monte Carlo solution}
        \label{EX56_ourmethod}
    \end{figure}
As shown in Figure \ref{EX56_ourmethod}, RPC accurately tracks all four statistical central moments, demonstrating excellent agreement with the MC reference and cleanly converging to the standard normal distribution as the invariant measure. A particularly important observation concerns the evolution of the third central moment. Even with variance reduction techniques, the MC reference with $10^6$ samples still exhibits visible statistical fluctuations because of incomplete convergence. In contrast, RPC yields a smooth and stable trajectory that converges precisely to the theoretical value. This result highlights a key advantage of our approach over sampling-based methods, as RPC accurately captures high-order transient dynamics without introducing numerical noise or requiring extremely large sample sizes.

\textbf{Case 2: } To further validate the ability of our method to capture non-Gaussian statistics, we consider a non-equilibrium regime where $\frac{\sigma_i^2}{2\gamma_i} \neq E$. This parameter choice drives the system towards a non-Gaussian invariant measure. In this simulation, the parameters are chosen as $\gamma_1 = 0.9, \gamma_2=1.2, \gamma_3 =1.5, \lambda_{12}=\lambda_{13}=\lambda_{23}=0.1, \beta_1=1.2, \beta_2=0.6, \beta_3=-1.8$. The noise intensities are specified by $\frac{\sigma_1^2}{2\gamma_1} = 0.6, \frac{\sigma_2^2}{2\gamma_2} = 0.4, \frac{\sigma_3^2}{2\gamma_3} = 0.3$, and the initial conditions are set to $u_0 \sim N(-0.5, 0.09), w_0 \sim N(0.2, 0.09), v_0 \sim N(0.5, 0.04)$. 

\begin{figure}
        \centering
        \includegraphics[width=15cm]{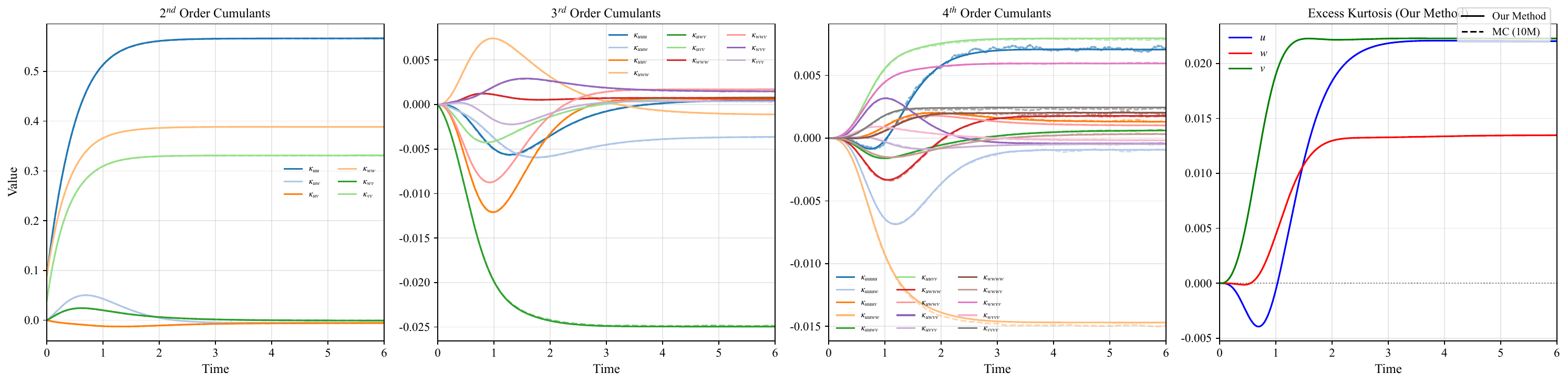}
        \caption{Comparison of the first four cumulants with the MC solution.}
        \label{EX56_nongaussian}
    \end{figure}

Figure \ref{EX56_nongaussian} shows the time evolution of the cumulants up to the fourth order and the excess kurtosis computed by RPC and MC. In contrast to the previous case, the higher-order cumulants (especially $\kappa_3$ and $\kappa_4$) deviate from zero,  indicating the non-Gaussian character of the solution. RPC shows excellent agreement with the MC reference solution, accurately resolving the complex nonlinear interactions and the resulting non-Gaussian features. These results, together with those in the Gaussian case, demonstrate that our method is highly effective for solving stochastic systems regardless of whether the invariant measure is Gaussian or non-Gaussian.

Besides the MC and gPC approaches considered above, SDEs can also be solved via Fokker-Planck equations \cite{risken1989fokker,pavliotis2014stochastic}. However, solving the corresponding Fokker-Planck equation with classical methods becomes computationally intractable for three-dimensional or higher-dimensional problems, due to the exponential growth of the computational cost with dimension. In contrast, RPC demonstrates remarkable efficiency and accuracy for such 3D coupled nonlinear stochastic systems, as evidenced by the close agreement with MC reference solutions for all four statistical moments. This illustrates the strong potential of RPC for high-dimensional problems.

\paragraph{Example 4.4.} Finally, to evaluate the potential of RPC for higher-dimensional problems governed by complex nonlinear interactions, we apply it to the stochastic Lorenz-96 model \cite{Sapsis2013BlendedRS}. This system serves as a canonical model in atmospheric science for the time evolution of scalar quantities on a latitude circle. The dynamic of this $6$-dimensional system is governed by the following SDE:
\begin{equation*}
\label{L96}
    dX_k(t) = \left[ (X_{k+1} - X_{k-2})X_{k-1} - X_k + F \right] dt + \sigma dW_k(t), \quad k=1, \dots, d,
\end{equation*}
with periodic boundary conditions and noise intensity $\sigma=0.08$. We specifically select a forcing parameter of $F=0.9$. This choice is motivated by the convergence requirements of polynomial chaos approximations, which depend on sufficient smoothness and regularity of the underlying probability density function. The $F=0.9$ regime guarantees the existence of a stable attractor and maintains sufficient regularity of the probability measure, while still exhibiting strong non-Gaussian features induced by the quadratic nonlinearity. 

To quantify the statistical evolution of the system approximated by RPC, we show the results of two physically meaningful macroscopic observables: the energy of the mean, defined as $ \frac{1}{2}\sum_{k=1}^d \mathbb{E} (X_k(t))^2 $, which represents the kinetic energy of the mean state; and the total variance, defined as $ \sum_{k=1}^d \text{Var}(X_k(t))$, which represents the total turbulent energy or aggregate uncertainty. In addition, we compute the
third-order central moment tensor $M_{ijk}$ in the statistical steady state to characterize the structure of nonlinear mode interactions.  We use $T=25$ for long-time evolution and $\Delta t =0.01$.

\begin{figure}
    \centering
    \includegraphics[width=15cm]{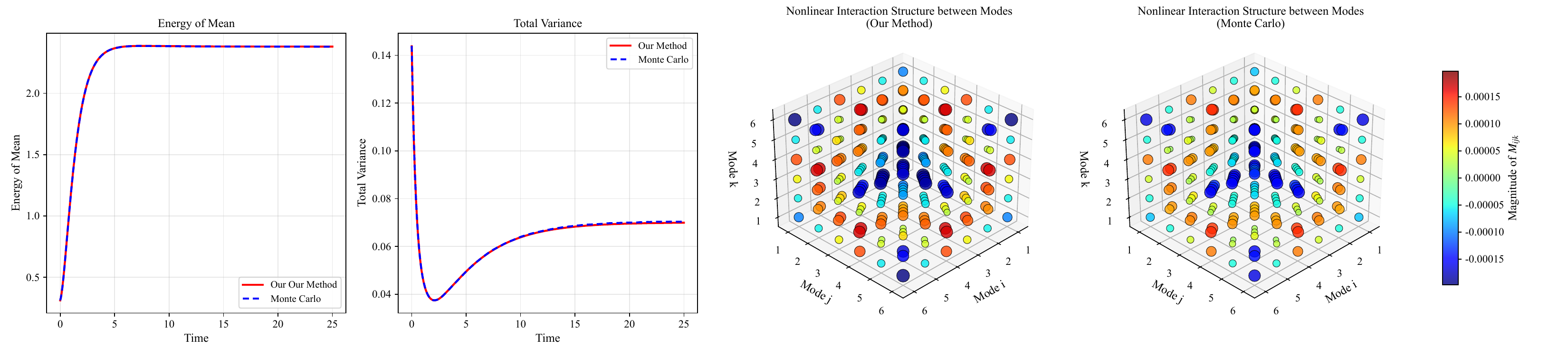}
    \caption{Comparison of the mean energy, total variance, and the steady-state structure of nonlinear interactions characterized by the third-order moment tensor with the MC solution.}
    \label{lorenz96}
\end{figure}

Figure \ref{lorenz96} presents a comparison between RPC ($L=2$,$S=2$) and MC simulations ($1,000,000$ paths). The results show excellent agreement for both the rapid transient growth of the mean energy and the non-monotonic evolution of the total variance. Furthermore, the close visual correspondence in the mode-mode nonlinear interaction structure confirms that RPC accurately captures the complex, multi-dimensional non-Gaussian correlations inherent in the system.

\section{Conclusions}\label{sec:conclusions}

\quad  This paper proposes a recursive polynomial chaos evolution (RPC) method as a model reduction technique for the long-time evolution of SDEs.
Rather than relying on a fixed orthogonal basis or a large number of sample paths, our method recursively evolves the orthonormal polynomial basis at each discrete time step. At every step, we compute a new orthogonal polynomial basis tailored to the corresponding probability measure. The system is then evolved by projecting the next state onto this updated basis and incorporating new Brownian increments. This constitutes a flexible framework that enables us to capture complex long-time dynamical behaviors of stochastic systems with reasonable high accuracy and low computational cost. We present two implementations within this framework: the first one directly evolves the orthonormal basis and associated algebraic structures via the It\^o-Taylor expansion while the second one evolves the moments and reconstructs the orthonormal polynomials at each time step through a Gram-Schmidt orthogonalization procedure from the Hankel matrix of moments.

We provide a convergence analysis in terms of both moments and the Wasserstein distance. Numerical results validate the good performance of RPC across a variety of challenging problems. We demonstrate that RPC consistently and accurately captures complex long-term phenomena, including invariant measures and non-Gaussian dynamics, where traditional PCE fails. Moreover, the resulting dynamics are stable and smooth, free from the oscillations that affect PCE during the time evolution or the statistical fluctuations that characterize insufficiently converged MC results. This advantage is also pronounced in multiple dimensions, as demonstrated by 3D and 6D test cases, a regime in which grid-based methods such as Fokker-Planck solvers become very expensive or intractable, and MC simulations are computationally less efficient.

Several interesting directions remain for future work. One we already mentioned is a global theory of stability of RPC, which concerns the growth conditions of the moments and the stability conditions for the construction of orthonormal polynomials. Generalizing RPC to non-Markovian SDEs, such as stochastic delay differential equations (SDDEs) and stochastic functional differential equations (SFDEs), represents a valuable and natural extension. The method may encounter difficulties when applied to systems with highly chaotic dynamics \cite{majdapnas,Majda_Harlim_2012}, and developing efficient generalizations to solve such problems remains an open issue. Another promising direction is the adaptation of RPC to high-dimensional SDEs ($d\geq 10$) and stochastic partial differential equations (SPDEs), which may require further integration with other model reduction techniques \cite{kazashi2025dynamical,CHENG2013753,Sapsis2009DynamicallyOF,Lujianfengdynamical,Lubich2014} to handle the curse of dimensionality.

\appendix
\section{Proofs of main results}
\label{more_proofs}
\subsection{Proof of Lemma~\ref{prop:closure}}\label{proof of proposition 3.1}
\begin{proof}
It suffices to verify closure for the most demanding term in~\eqref{tripleproduct_RPC}, namely the diffusion contribution
\begin{equation*}\label{eq:diffusion_triple}
  \frac{h}{2}\sum_{i,l=1}^{d}\sum_{j=1}^{m}
  \mathcal{E}\!\left[
    \frac{\partial^{2}}{\partial v^{(i)}\partial v^{(l)}}
      \bigl(T^{\pRPC}_{k,\alpha} T^{\pRPC}_{k,\beta} 
            T^{\pRPC}_{k,\gamma}\bigr)
    \Sigma^{(i,l)}
  \right],
  \qquad (\alpha,\beta,\gamma)\in\mathcal{J}_{L,S}.
\end{equation*} 
Here, $\Sigma^{(i,l)}:=\sum_{j=1}^{m}\sigma_h^{(i,j)}\sigma_h^{(l,j)}$ has degree at most $S$. We detail the calculation for the representative case where both derivatives act on $T^{\pRPC}_{k,\alpha}$; the remaining terms follow analogously and have strictly lower total degrees. We denote the partial derivative operator as $\partial^2_{i,l} \equiv \frac{\partial^2}{\partial v^{(i)}\partial v^{(l)}}$. Expanding in the basis, we write
\begin{equation*}\label{eq:deriv_expand}
  \partial^2_{i,l} T^{\pRPC}_{k,\alpha} 
  = \sum_{|\eta|\leq|\alpha|-2}
    D_{\alpha,\eta}^{(i,l)} T^{\pRPC}_{k,\eta},
  \qquad
  \Sigma^{(i,l)}
  = \sum_{|\mu|\leq S}
    s_{\mu}^{(i,l)} T^{\pRPC}_{k,\mu}.
\end{equation*}
A further expansion yields
\begin{equation*}\label{eq:expansion_example}
\mathcal{E} \left[ (\partial^2_{i,l} T^{\pRPC}_{k,\alpha}) T^{\pRPC}_{k,\beta} T^{\pRPC}_{k,\gamma} \Sigma^{(i,l)} \right]
= \sum_{|\eta| < |\alpha|} \sum_{|\mu| \leq S} D_{\alpha,\eta}^{(i,l)} s_\mu^{(i,l)} \mathcal{E} \left[ T^{\pRPC}_{k,\eta} T^{\pRPC}_{k,\beta} T^{\pRPC}_{k,\gamma} T^{\pRPC}_{k,\mu} \right]. 
\end{equation*} 
To reduce four-fold expectations to a bilinear expression in triple products, we project the pair $T^{\pRPC}_{k,\eta} T^{\pRPC}_{k,\beta}$ onto the truncated basis using the multiplication rule \eqref{product calculation}, with the projection level adapted to the index magnitudes. When $|\eta|+|\beta|\leq J - L = L+S$, we project onto the full truncated basis of degree $L$:
\begin{equation*}\label{eq:proj_case1}
  T^{\pRPC}_{k,\eta} T^{\pRPC}_{k,\beta}
  \approx
  \sum_{|\nu|\leq L}\Gamma^{\pRPC}_{k,\eta\beta\nu} T^{\pRPC}_{k,\nu},
\end{equation*}
so that the four-fold expectation becomes $\sum_{|\nu|\leq L} \Gamma^{\pRPC}_{k,\eta\beta\nu}  \Gamma^{\pRPC}_{k,\nu\gamma\mu}$. Both triple products belong to $\mathcal{J}_{L,S}$: for the first, $|\eta|+|\beta|+|\nu|\leq (L+S)+L = J$; for the second, $|\nu|+|\gamma|+|\mu|\leq L+L+S = J$. When $|\eta|+|\beta|> L+S$, a projection onto degree~$L$ will produce triple products $\Gamma^{\pRPC}_{k,\eta\beta\nu}$ with total degree exceeding~$J$. Instead, we lower the projection level to $J-|\eta|-|\beta| \leq L$:
\begin{equation*}\label{eq:proj_case2}
 T^{\pRPC}_{k,\eta} T^{\pRPC}_{k,\beta}
  \approx
  \sum_{|\nu|\leq J-|\eta|-|\beta|}
\Gamma^{\pRPC}_{k,\eta\beta\nu} T^{\pRPC}_{k,\nu}.
\end{equation*}
By construction, $|\eta|+|\beta|+|\nu|\leq J$ for the first triple product. For the second, $|\nu|+|\gamma|+|\mu|\leq  J-|\eta|-|\beta|+|\gamma|+S$. Since $|\eta|+|\beta|>L+S$ by assumption and $|\gamma| \leq L$, this yields $|\nu|+|\gamma|+|\mu|< J - L+|\gamma|\leq J$. The drift contributions yield a weaker constraint that is subsumed by the above analysis.

It follows that all triple products on the right-hand side of \eqref{tripleproduct_RPC} satisfy total degree at most $J$ and hence belong to $\mathcal{J}_{L,S}$, which establishes the closure of pRPC under index set $\mathcal{J}_{L,S}$.
\end{proof}

\subsection{Proof of Lemma \ref{eq:boundW1}}
\label{proof of Lemma 4.6}
 \begin{proof}
        Define $\tilde\mu$ to be the distribution of $x/\hat R$ when $x\sim\mu$, and
        define $\tilde\nu$ analogously.  Then $\tilde\mu$ and $\tilde\nu$ are
        supported on $[-1,1]^d$, and $W_1(\mu,\nu)=\hat R  W_1(\tilde\mu,\tilde\nu)$.
        The moments transform as
        $m^{\tilde\mu}_\gamma=m^\mu_\gamma/(\hat R)^{|\gamma|}$, so it suffices to prove the lemma for $\hat R=1$.  We write $\Delta m_\gamma=m^{\tilde\mu}_\gamma-m^{\tilde\nu}_\gamma$ for brevity.
        
        By the Kantorovich-Rubinstein duality and the fact that any Lipschitz function can be arbitrarily well approximated by a smooth function \cite{musco2025sharper}, the Wasserstein distance admits the representation
        \begin{equation*}\label{eq:KR}
        W_1(\tilde\mu,\tilde\nu)
        = \sup_{\substack{f:[-1,1]^d\to\R\\[1pt]
        f\text{ is 1-Lipschitz and smooth}}}
        \int_{[-1,1]^d}  f(x)\bigl(\tilde\mu(x)-\tilde\nu(x)\bigr)  dx
        \end{equation*}
        with $f(0)=0$.
        Let $U_j$ denote the $j$-th Chebyshev polynomial of the first kind on $[-1,1]$.  For a multi-index
        $\gamma=(\gamma_1,\dots,\gamma_d)$, set $U_\gamma(x)=\prod_{i=1}^d U_{\gamma_i}(x_i)$
        and define the weight function $\omega(x)=\prod_{i=1}^d(1-x_i^2)^{-1/2}$.  The
        normalized Chebyshev polynomials
        $\bar U_\gamma(x)=\sqrt{2^{\nnz(\gamma)}/\pi^d} U_\gamma(x)$, where
        $\nnz(\gamma)=|\{i:\gamma_i\neq 0\}|$, form an orthonormal basis of the weighted space $L^2_\omega([-1,1]^d)$ with inner product
        $\langle f,g\rangle_\omega=\int_{[-1,1]^d}f  g \omega  dx$. Every $1$-Lipschitz smooth function
        $f: [-1,1]^d \- \to \R$ admits a Chebyshev expansion
        $f=\sum_{\gamma \in \mathbb{N}_{0}^d}c_\gamma \bar U_\gamma $ converging uniformly, where
        $c_\gamma =\langle f,\bar U_\gamma \rangle_\omega$, and the Parseval identity reads 
        \begin{equation}\label{eq:parseval}
        \sum_{\gamma \in \mathbb{N}_{0}^d} c_\gamma ^2
        = \int_{[-1,1]^d} f(x)^2  \omega(x)  dx \leq  d  \pi^d.
        \end{equation}
        
        The following multivariate version of Jackson's theorem is proved in \cite{musco2025sharper}.  Given an $1$-Lipschitz smooth function $f:[-1,1]^d\to\R$ and an integer $m\ge 2$, there exists a polynomial 
        \begin{equation*}
            \tilde f(x)
        = \sum_{\gamma\in\{0,\dots,2m-2\}^d}\tilde c_\gamma \bar U_\gamma(x)
        \end{equation*}
        satisfying $\|f-\tilde f\|_\infty\le 9 d/m$ and $|\tilde c_\gamma|\le|c_\gamma|$ for every $\gamma$. We apply this with $m=\bigl\lfloor s/(2d)\bigr\rfloor+1$. Since $2m-2=2\lfloor q/(2d)\rfloor\ge 0$ and $m\ge q/(2d)$, the approximation error satisfies $\|f-\tilde f\|_\infty\le 9d/m\le 18d^2/q$.
        Every monomial $x^\gamma$ appearing in $\tilde f$ has $\gamma_i\le 2m-2$ for each coordinate, and hence total degree $|\gamma|\le d(2m-2)\le q$.  
        
        Now split the inner product
        \begin{equation}\label{eq:split}
        \langle f,\tilde\mu-\tilde\nu\rangle
        = \langle\tilde f,\tilde\mu-\tilde\nu\rangle
        + \langle f-\tilde f,\tilde\mu-\tilde\nu\rangle,
        \end{equation}
        where we write $\langle f,\tilde\mu-\tilde\nu\rangle=\int f  d(\tilde\mu-\tilde\nu)$. The remainder is bounded by
        \begin{equation}\label{eq:remainder}
        \bigl|\langle f-\tilde f,\tilde\mu-\tilde\nu\rangle\bigr|
        \le 2\|f-\tilde f\|_\infty
        \le  \frac{36d^2}{q}.
        \end{equation}
        For the first term on the right hand side of \eqref{eq:split}, write $\tilde f(x)=\sum_\gamma a_\gamma  x^\gamma$ in the monomial basis. The Cauchy-Schwarz inequality gives
        \begin{equation}\label{eq:CS}
        \left|\langle\tilde f,\tilde\mu-\tilde\nu\rangle\right|
        = |\sum_{{\gamma\in\{0,\dots,2m-2\}^d}}|a_\gamma  \Delta m_\gamma | \le |\mathbf{a}| |\Delta\mathbf{m}|.
        \end{equation}  
        
        We now bound $|\mathbf{a}|$. Assume each univariate Chebyshev polynomial has a monomial expansion $U_j(x)=\sum_{l=0}^j t_{j,l}  x^l$. Define the matrix 
        $M_1\in\R^{(2m-1)\times(2m-1)}$ with entries
        \[
        (M_1)_{j,l}
        = \sqrt{\frac{2^{  \mathds{1} \{l\ge 1\}}}{\pi}} t_{j,l},
        \qquad 0\leq j,l\leq 2m-2,
        \]
        so that the coefficient of $x^l$ in $\bar U_j(x)$ is exactly
        $(M_1)_{j,l}$.  Since each multivariate Chebyshev polynomial factorizes as
        $\bar U_K(x)=\prod_{i=1}^d\bar U_{k_i}(x_i)$, the coefficient of $x^\gamma$
        in $\bar U_K(x)$ is $\prod_{i=1}^d(M_1)_{\gamma_i,\eta_i}$, which is precisely
        the $(\gamma,\eta)$-entry of the Kronecker product
        $M_d=M_1^{\otimes d}$. The coefficient vectors are related by $\mathbf{a}=M_d  \tilde{\mathbf{c}}$. Define $s_l=\sum_{l=0}^j t_{j,l}^2$. We claim that $\sqrt{s_l}\le(1{+}\sqrt{2}  )^l$ for every
        $l\ge 0$. This can be proved by induction using that $U_l(x)=2x  U_{l-1}(x)-U_{l-2}(x)$. 
                
        The norm of each column of $M_d$ indexed by $\eta=(\eta_1,\dots,\eta_d)$ equals
        $\prod_{i=1}^d\frac{2^{  \mathds{1}\{\eta_i\ge 1\}}}{\pi}  s_{\eta_i}$,
        which is at most
        $\bigl(\frac{2}{\pi}\bigr)^d\prod_{i=1}^d(1{+}\sqrt{2}  )^{2\eta_i}
        =\bigl(\frac{2}{\pi}\bigr)^d(1{+}\sqrt{2}  )^{2|\eta|}$. Therefore
        \begin{equation}
        \label{MdF}
            \begin{aligned}
                |M_d|
        &\leq \sum_{\eta\in\{0,\dots,2m-2\}^d}
         \Bigl(\frac{2}{\pi}\Bigr)^{ d}\beta^{|\eta|}
        = \Bigl(\frac{2}{\pi}\Bigr)^{ d}
        \prod_{i=1}^d\sum_{\eta_i=0}^{2m-2}\beta^{\eta_i} = \Bigl(\frac{2}{\pi}\Bigr)^{ d}
        \left(\frac{\beta^{2m-1}-1}{\beta-1}\right)^{ d}
        \leq \frac{(1{+}\sqrt{2}  )^{2q+d}}{\pi^{d}}.
            \end{aligned}
        \end{equation} 
        Substituting \eqref{eq:parseval} and \eqref{MdF} into~\eqref{eq:CS} yields:
        \begin{equation}\label{ft_bound}
        \bigl|\langle \tilde f,\tilde\mu-\tilde\nu\rangle\bigr| \leq|\mathbf{a}| |\Delta\mathbf{m}|\leq |M_d||\mathbf{\tilde c}||\Delta\mathbf{m}|\leq
        \frac{(1{+}\sqrt{2}  )^{q+d/2}}{\pi^{d/2}}
          \cdot\sqrt{d  \pi^d} \cdot \Big(\sum_{|\gamma|\le q}
          |\Delta m_\gamma|^2\Big)^{\!1/2}.
        \end{equation}
        Next, substituting \eqref{ft_bound} and \eqref{eq:remainder} into \eqref{eq:split} and rearranging, we obtain for every $1$-Lipschitz smooth function $f$ with
        $f(\mathbf{0})=0$:
        \begin{equation*}
         \bigl|\langle f,\tilde\mu-\tilde\nu\rangle\bigr|
        \le \frac{36d^2}{q} + C_d\cdot 3^q
          \Big(\sum_{|\gamma|\le q}
          |\Delta m_\gamma|^2\Big)^{ 1/2}.
        \end{equation*}
        where $C_d = \sqrt{d}(1+\sqrt{2})^{d/2}$ .
        Taking the supremum over all such $f$ completes the proof with $g(q)=C_d  3^q$.
    \end{proof}

\subsection{Proof of Lemma \ref{determinate lemma2}} 
    \label{proof of Lemma 3.3} 
    \begin{proof}
    A direct calculation yields the following lemma.
    \begin{Lemma}\label{lem:gaussian_mgf}
    Let $z \sim N(0, I_m)$ be a standard $m$-dimensional Gaussian random vector, let $a \in \R^m$ be a fixed vector, and let $B$ be an $m \times m$ symmetric matrix. If all eigenvalues of $B$ are strictly less than $\frac{1}{2}$, i.e., $I - 2B$ is positive definite, then
    \[
    \E\left[e^{a^T z + z^T B z}\right] = \det(I - 2B)^{-1/2} \exp\Big(\frac{1}{2} a^T (I - 2B)^{-1} a\Big).
    \]
    \end{Lemma}
    Let $\mathscr{F}_k$ be the $\sigma$-field generated by $v_0,\cdots,v_k$. Denote $\hat{v}_k = v_k + b_h(v_k)h$. By Lemma \ref{lem:gaussian_mgf},
    \begin{equation}\label{exp eq1}
        \begin{aligned}
            &\mathbb{E}\left[\exp\left(\hat\delta|v_{k+1}|^{2}\right) | \mathscr{F}_k\right]=\mathbb{E}\left[\exp\left(\hat\delta|\hat{v}_{k} + \sigma_{h}(v_{k})z_{k}\sqrt{h}|^{2}\right) | \mathscr{F}_k\right]\\
            = & \mathbb{E}\left[\exp\left(\hat\delta|\hat{v}_{k}|^{2} +  2\hat\delta\sqrt{h}\hat{v}_{k}^T\sigma_{h}(v_{k})z_{k}+ \hat\delta h|\sigma_{h}(v_{k})z_{k}|^{2}\right) | \mathscr{F}_k\right] \\
            =  & e^{\hat\delta |\hat{v}_k|^2}\det(I - 2\hat\delta h \sigma^T_{h}\sigma_{h}(v_{k}))^{-1/2} \exp\left(2\hat\delta^2h\hat{v}_{k}^T\sigma_{h} (I - 2\hat\delta h \sigma^T_{h}\sigma_{h}(v_{k}))^{-1}\sigma_{h}^T\hat{v}_{k}\right).
        \end{aligned}
    \end{equation}
    By Assumption \ref{asm:drift diffusion terms}, taking $\hat\delta < 1$ and $h^* \leq 1 / (4K^2) $, we have
    \begin{equation*}
        \det(I - 2\hat\delta h A)^{-1/2} \leq (1 - 2\hat\delta h K^2)^{-m/2}\leq e^{2m\hat\delta h K^2},
    \end{equation*}
    where the last inequality follows from $(1-x)^{-1}\leq e^{2x}$ for $x\in[0,1/2]$. Furthermore,
    \begin{equation*}
        \hat{v}_{k}^T\sigma_{h} (I - 2\hat\delta h \sigma^T_{h}\sigma_{h}(v_{k}))^{-1}\sigma_{h}^T\hat{v}_{k} \leq \frac{K^2|\hat{v}_k|^2}{1 - 2\hat\delta h K^2}\leq 2K^2|\hat{v}_k|^2.
    \end{equation*}
    Therefore, \eqref{exp eq1} implies
    \begin{equation*}\label{exp bdd eq1}
        \mathbb{E}\left[\exp\left(\hat\delta|v_{k+1}|^{2}\right) | \mathscr{F}_k\right] \leq \exp\left( \hat\delta(1+ 4K^2\hat\delta h) |\hat{v}_k|^2 + 2m\hat\delta hK^2 \right).
    \end{equation*}
    Since
    \begin{equation*}
        \begin{aligned}
            |\hat{v}_k|^2 &\leq (1+Kh)^2|v_k|^2 + 2(1+Kh)Kh|v_k| + K^2h^2 \\
        &\leq (1+Kh)^2|v_k|^2 + (1+Kh)Kh(|v_k|^2 + 1) + K^2h^2 ,
        \end{aligned}
    \end{equation*}
    and since $\hat\delta < 1$, there exists a constant $C$ depending only on $K$ such that
    \begin{equation}\label{exp eq2}
        \mathbb{E}\left[\exp\left(\hat\delta|v_{k+1}|^{2}\right) | \mathscr{F}_k\right] \leq \exp\left( \hat\delta(1+ Ch) |v_k|^2 + Ch \right).
    \end{equation}
    The bound~\eqref{exp eq2} exhibits a recursive form. Taking expectation on both sides and using induction yields, for all $0\leq k\leq n$,
    \begin{equation*}
        \mathbb{E}\left[\exp\left(\hat\delta|v_{k}|^{2}\right)\right] \leq  \mathbb{E}\left[\exp\left(\hat\delta e^{CT}|v_{0}|^{2} + CT\right)\right] 
    \end{equation*}
    Consequently, by initially choosing $ \hat\delta\leq \min\{\delta_0 e^{-CT},1\} $ and applying Assumption \ref{asm:initial condition}, we obtain that $v_k$ is exponentially integrable.

    For scheme \eqref{projection scheme}, since
    \[
    \max\left\{ \sup_{x \in I_R} |b_h(x) - p_b^{R,S}(x)|, \sup_{x \in I_R} |\sigma_h(x) - p_\sigma^{R,S}(x)| \right\} \leq C R^rS^{-r},
    \]
    we obtain via a similar argument that
    \begin{equation*}
        \begin{aligned}
            \mathbb{E}\left[\exp\left(\hat\delta|v^{R,S}_{k+1}|^{2}\right) \middle| \mathscr{F}_k\right] &\leq \mathbb{E}\left[\exp\left(\hat\delta|\hat{v}_{k}^{R,S} + \sigma_{h}(v^{R,S}_{k})z_{k}\sqrt{h}|^{2}\right) | \mathscr{F}_k\right]\\
            &\leq \exp\left( \hat\delta(1+ Ch) |v^{R,S}_k|^2 + Ch + Ch R^rS^{-r} \right).
        \end{aligned}
    \end{equation*}
    The desired result follows accordingly.
    \end{proof}

\subsection{Proof of Lemma \ref{Lemma 1}} 
\label{proof of Lemma 3.4}
\begin{proof}
    We prove that there exist constants $\delta > 0$, $C > 0$ and $h^*$ depending only on $K$ and $T$, such that for all $h \in (0, h^*]$ and $R > 0$:
    \begin{equation*}
        \mathbb{P}\left(\max_{0 \leq k \leq n} |v_k| > R/3\right) \leq Ce^{-\delta R^2}.
    \end{equation*}
    From Lemma \ref{determinate lemma2} and \ref{exp eq2}, there exists a constant $C_1 > 0$ depending on $K$ such that for $0 < \theta < \delta$,
    \begin{equation}\label{Rbdd eq1}
        \mathbb{E}\left[\exp\left(\theta|v_{k+1}|^{2}\right) | \mathscr{F}_k\right] \leq \exp\left( \theta(1+ C_1 h) |v_k|^2 + C_1 h \right),
    \end{equation}
    where $\delta$ is obtained from Lemma \ref{determinate lemma2}. Choose $\delta > 0$ such that $\delta e^{C_1 T} \leq \hat\delta$, and define $\Theta_k = \delta(1+C_1 h)^{n-k}$ for $k = 0, 1, \ldots, n.$ 
    From \eqref{Rbdd eq1}, we have
    \begin{equation*}
        \mathbb{E}\left[\exp\left(\delta_{k+1}|v_{k+1}|^{2}\right) | \mathscr{F}_k\right] \leq \exp\left( \delta_{k+1}(1+ C_1 h) |v_k|^2 + C_1 h \right) 
        = \exp\left( \delta_{k}|v_k|^2 + C_1 h \right).
    \end{equation*}
    Define $M_k := e^{\Theta_k |v_k|^2 - C_1 k h}$ for $k = 0, 1, \ldots, n$. Then
    \begin{equation*}
        \begin{aligned}
            \E[M_{k+1} | \mathscr{F}_k] 
            &= e^{-C_1(k+1)h} \E\left[e^{\delta_{k+1}|v_{k+1}|^2}  \big|  \mathscr{F}_k\right] \leq e^{-C_1(k+1)h} \cdot e^{\Theta_k |v_k|^2 + C_1 h} = M_k.
        \end{aligned}
    \end{equation*}
    Thus $\{M_k\}_{k=0}^n$ is a non-negative supermartingale with respect to $\{\mathscr{F}_k\}_{k=0}^n$. Applying Doob's maximal inequality, we obtain the uniform bound
    \begin{equation*}
        \mathbb{P}\left(\max_{0 \leq k \leq n} M_k > \lambda\right) \leq \frac{\E[M_0]}{\lambda}.
    \end{equation*}
    If $|v_k| > R/3$ for some $k \in \{0, \ldots, n\}$, then using $\Theta_k \geq \delta$ and $C_1 k h \leq C_1 T$,
    \[
    M_k = e^{\Theta_k |v_k|^2 - C_1 k h} \geq e^{\delta (R/3)^2 - C_1 T} = e^{\delta R^2/9 - C_1 T}.
    \]
    Taking $\lambda = e^{\delta R^2/9 - C_1 T}$ in Doob's inequality yields
    \begin{equation*}
        \begin{aligned}
            \mathbb{P}\left(\max_{0 \leq k \leq n} |v_k| > R/3\right) 
            &\leq \mathbb{P}\left(\max_{0 \leq k \leq n} M_k > \lambda\right)\leq \frac{e^{\delta |v_0|^2}}{e^{\delta R^2/9 - C_1 T}} = e^{\delta |v_0|^2 - \delta R^2/9 + C_1 T}.
        \end{aligned}
    \end{equation*}

    Define $\tau_R = \inf\{ k:  |v_k|\geq R/3 \}$.
    Then for all $0\leq k\leq n$,
    \begin{equation}\label{martingale est}
        \begin{aligned}
            \mathbb{E}\left[| \chi_{R}(v_{k}) - v^{R}_{k}|^{2}\right] &= \mathbb{E}\left[| \chi_{R}(v_{k}) - v^{R}_{k}|^{2}\mathds{1}_{\{\tau_R >n\}}\right] + \mathbb{E}\left[| \chi_{R}(v_{k}) - v^{R}_{k}|^{2}\mathds{1}_{\{\tau_R\leq n\}}\right]. 
        \end{aligned}
    \end{equation}
    On the event $\{\tau_R > n\}$, we have $v_0 = v_0^R$ by definition. By induction, it follows that $|v_k|< R/3$, $|v_k^R| < R/3$, and $v_k = v_k^R$ for all $0 \leq k \leq n$. Consequently, the first term in \eqref{martingale est} vanishes. For the second term, the Cauchy inequality yields
    \begin{equation*}
        \mathbb{E}\left[| \chi_{R}(v_{k}) - v^{R}_{k}|^{2}\mathds{1}_{\{\tau_R\leq n\}}\right] \leq \left(\mathbb{P}(\tau_R\leq n)\right)^{1/2} \left(\mathbb{E}\left[|\chi_R(v_k) - v_k^R|^4\right]\right)^{1/2}.
    \end{equation*}
    Since
    \begin{equation*}
        \mathbb{P}(\tau_R\leq n) = \mathbb{P}\left(\max_{0 \leq k \leq n} |v_k| > R/3\right) \leq Ce^{-{\delta}R^2},
    \end{equation*}
    and $\mathbb{E}\left[|\chi_R(v_k) - v_k^R|^4\right]$ is uniformly bounded by Lemma \ref{determinate lemma2}, this completes the proof.
\end{proof}

\section*{Acknowledgments}
GB's work was funded in part by NSF Grant DMS-2306411 and ONR Grant N00014-26-1-2017. ZZ was partly supported by the National Natural Science Foundation of China (Projects 92470103), the Hong Kong RGC grant (Projects 17304324 and 17300325), the Seed Funding Programme for Basic Research (HKU), and the Hong Kong RGC Research Fellow Scheme 2025. The computations were performed using the research computing facilities offered by Information Technology Services, the University of Hong Kong.

\bibliographystyle{siamplain}
\bibliography{references}

\end{document}